\documentclass[12pt]{amsart}
\usepackage{amssymb}
\usepackage{amsmath}
\usepackage{amsthm}
\usepackage{graphicx}
\usepackage{enumerate}
\usepackage{tikz}
\usepackage{tabularx}
\usepackage{tabmac_avi}

\newtheorem{thm}{Theorem}[section]
\newtheorem{prop}[thm]{Proposition}
\newtheorem{cor}[thm]{Corollary}
\newtheorem{lem}[thm]{Lemma}

\newtheorem{defn}[thm]{Definition}

\newtheorem{prob}[thm]{Problem}

\numberwithin{equation}{section}

\newcommand{\sn}{\mathfrak{S}_n}
\newcommand{\sr}{\mathfrak{S}_r}
\newcommand{\mfs}[1]{\mathfrak{S}_{#1}}
\newcommand{\zsn}{\mathbb{Z}[\sn]}

\newcommand{\zxn}{\mathbb{Z}[x_{1,1},\dotsc,x_{n,n}]}
\newcommand{\cqq}{\mathbb{C}[\qp12, \qm12]}
\newcommand{\zqq}{\mathbb{Z}[\qp12, \qm12]}

\newcommand{\csn}{\mathbb{C}[\sn]}

\newcommand{\eJ}{{\emptyset, J}}

\newcommand{\wo}{w_0}

\newcommand{\euv}{\epsilon_{u,v}}

\newcommand{\evw}{\epsilon_{v,w}}

\newcommand{\qp}[2]{q^{\frac{#1}{#2}}}
\newcommand{\qm}[2]{q^{\negthinspace\Bar\,\frac{#1}{#2}}}

\newcommand{\quv}{q_{u,v}}

\newcommand{\qvw}{q_{v,w}}

\newcommand{\qeu}{q_u}

\newcommand{\qev}{q_v}
\newcommand{\qew}{q_w}

\newcommand{\qiuv}{q_{u,v}^{-1}}

\newcommand{\qivw}{q_{v,w}^{-1}}

\newcommand{\qiev}{q_{e,v}^{-1}}
\newcommand{\qiew}{q_w^{-1}}
\newcommand{\qiey}{q_y^{-1}}
\newcommand{\qieyp}{q_{y'}^{-1}}

\newcommand{\qdiff}{\qp12 - \qm12}

\newcommand{\wT}{\widetilde T} 
\newcommand{\wTp}{\widetilde T'}
\newcommand{\ol}[1]{\overline{#1}}
\newcommand{\wR}[1]{\widetilde R_{#1}}
\newcommand{\wS}[1]{\widetilde S_{#1}}

\newcommand{\hnq}{H_n(q)}

\newcommand{\A}{\mathcal{A}}
\newcommand{\anq}{\mathcal{A}(n,q)}

\newcommand{\arnq}{\mathcal{A}_r}

\newcommand{\Ann}{\mathcal{A}_{[n],[n]}}
\newcommand{\annnq}{\mathcal{A}_{[n],[n]}}
\newcommand{\anmnq}{\mathcal{A}_{[n],M}(n;q)}

\newcommand{\AnM}{\mathcal{A}_{[n],M}}
\newcommand{\ALM}{\mathcal{A}_{L,M}}

\newcommand{\Hpej}{H'_{\smash\eJ}} 
\newcommand{\Hej}{H_\eJ}

\newcommand{\Wejp}{W^{\emptyset,J}_+}

\newcommand{\Wijp}{W^{I,J}_+}
\newcommand{\Wijm}{W^{I,J}_-}

\newcommand{\pej}[1]{p_{#1}^{\emptyset,J}}

\newcommand{\rej}[1]{r_{#1}^{\emptyset,J}}

\newcommand{\redexp}[2]{s_{i_{#1}} \ntnsp \cdots s_{i_{#2}}}

\newcommand{\imm}[1]{\mathrm{Imm}_{#1}}

\newcommand{\sumsb}[1]{\sum_{\substack{#1}}}  %%%for multi-line sums
\newcommand{\inv}{\textsc{inv}}
\newcommand{\rinv}{\textsc{rinv}}
\newcommand{\defeq}{\underset{\mathrm{def}}=} 
\newcommand{\dfct}{\textsc{d}}
\newcommand{\dnc}{\textsc{dnc}}
\newcommand{\dc}{\textsc{dc}}

\newcommand{\spn}{\mathrm{span}}

\newcommand{\wgt}{\mathrm{wgt}}
\newcommand{\type}{\mathrm{type}}
\newcommand{\incross}{\textsc{invnc}}
\newcommand{\cdncross}{\textsc{cdnc}}
\newcommand{\cross}{\textsc{c}}

\newcommand{\avoidsp}{avoids the patterns $3412$ and $4231${}}

\newcommand{\ul}[1]{\underline{#1}}  
\newcommand{\ssec}[1]{\subsection{#1}{$\negthinspace$}}

\newcommand{\tr}{{\negthickspace \top \negthickspace}}
\newcommand{\ntnsp}{\negthinspace}
\newcommand{\ntksp}{\negthickspace}
\newcommand{\nTksp}{\negthickspace\negthickspace}

\newcommand{\oqglnc}{{\mathcal O}_q(GL_n (\mathbb C))}

\newcommand{\oqslnc}{{\mathcal O}_q(SL_n (\mathbb C))}
\newcommand{\qdet}{\mathrm{det}_q}

\newcommand{\permmon}[2]{#1_{1,#2_1} \ntnsp\cdots {#1}_{n,#2_n}}
\newcommand{\sprod}[2]{s_{#1_1} \ntnsp\cdots s_{#1_#2}}
\newcommand{\ssm}{\smallsetminus}

\newcommand{\bs}{\backslash}  
\newcommand{\wleq}{\leq_W}
\newcommand{\wless}{<_W}  
\newcommand{\slambda}{\mathfrak{S}_\lambda}
\newcommand{\slambdamin}{\mathfrak{S}_\lambda^{-}}

\def\hhhpsp{\def\baselinestretch{0.13}\large\normalsize}
\def\ssp{\def\baselinestretch{1.0}\large\normalsize}
\begin{document}
\author{Ryan Kaliszewski}
\author{Justin Lambright}%\corref{cor}}   needed for elsarticle
%\ead{jjlambright@anderson.edu}
\author{Mark Skandera}
%\ead{mas906@lehigh.edu,mark.skandera@gmail.com}
\title[Bases and induced sign characters]
      {Bases of the quantum matrix bialgebra and induced sign characters
       of the Hecke algebra}
%\title[Induced sign characters]
%{Induced sign characters and wiring diagram 
%bases of the Hecke algebra}
%{A generalization of Kazhdan and Lusztig's R-polynomials}
%\cortext[cor]{Corresponding author}      needed for elsarticle
%\address{Department of Mathematics, Anderson University, ??? , Anderson, IN 4601?, United States.}
%\address{Department of Mathematics, Lehigh University, 14 East Packer Avenue, Christmas-Saucon Hall, Bethlehem, PA 18015, United States.}
%\begin{keyword}
%Kazhdan-Lusztig \sep Hecke algebra \sep immanant \sep symmetric group
%\end{keyword}       needed for elsarticle??

\bibliographystyle{../dart}

\date{\today}

\begin{abstract}
  We combinatorially describe entries of the transition matrices which
  relate monomial bases of the zero-weight space of the quantum
  matrix bialgebra.  This description leads to a combinatorial rule for
  evaluating induced sign characters of the type $A$ Hecke algebra $\hnq$
  at all elements of the form $(1 + T_{s_{i_1}}) \cdots (1 + T_{s_{i_m}})$,
  including the Kazhdan-Lusztig basis elements indexed by
  $321$-hexagon-avoiding permutations.
  This result is the first subtraction-free rule for evaluating all
  elements of a basis of the $\hnq$-trace space
  at all elements of a basis of $\hnq$.
\end{abstract}

\maketitle

\section{Introduction}\label{s:intro}

Define the {\em symmetric group algebra} $\zsn$
%generated over $\mathbb C$ by $s_1,\dotsc, s_{n-1}$, 
and 
the
%(type $A$ Iwahori-)
{\em (type $A$ Iwahori-) Hecke algebra} $\hnq$
%(of type $A$)
to be the algebras with
multiplicative identity elements $e$ and $T_e$,
respectively,
generated over $\mathbb Z$ and $\zqq$
by elements
$s_1,\dotsc, s_{n-1}$ and 
$T_{s_1},\dotsc, T_{s_{n-1}}$, 
subject to the relations
%The {\em symmetric group algebra} $\zsn$
%and 
%the
%{\em (type $A$ Iwahori-) Hecke algebra} $\hnq$
%have similar presentations as algebras over $\mathbb Z$ and $\zqq$
%respectively, with 
%multiplicative identity elements $e$ and $T_e$,
%generators 
%$s_1,\dotsc, s_{n-1}$ and 
%$T_{s_1},\dotsc, T_{s_{n-1}}$, 
%and relations
\begin{equation}\label{eq:hnqdef}
\begin{alignedat}{3}
s_i^2 &= e &\qquad
T_{s_i}^2 &= (q-1) T_{s_i} + qT_e &\qquad
&\text{for $i = 1, \dotsc, n-1$},\\
s_is_js_i &= s_js_is_j &\qquad
T_{s_i}T_{s_j}T_{s_i} &= T_{s_j}T_{s_i}T_{s_j} &\qquad
&\text{for $|i - j| = 1$},\\
s_is_j &= s_js_i &\qquad
T_{s_i}T_{s_j} &= T_{s_j}T_{s_i} &\qquad
&\text{for $|i - j| \geq 2$}.
\end{alignedat}
\end{equation}
%where $e$ and $\wT_e$ are the respective multiplicative identity elements.
Analogous to the natural basis $\{ w \,|\, w \in \sn \}$ of $\zsn$
is the natural basis $\{ T_w \,|\, w \in \sn \}$ of $\hnq$,
where we define
$T_w = T_{s_{i_1}} \ntksp \cdots T_{s_{i_\ell}}$
whenever $\sprod i\ell$
%$s_{i_1} \ntnsp\cdots s_{i_\ell}$
is a reduced (short as possible)
expression for $w$ in $\sn$.  We call $\ell$ the {\em length}
of $w$ and write $\ell = \ell(w)$.
We define the {\em one-line notation} $w_1 \cdots w_n$ of $w \in \sn$ by
letting any expression for $w$ act on the word $1 \cdots n$,
where each generator $s_j$ acts on an $n$-letter word by
swapping the letters in positions $j$ and $j+1$,
\begin{equation*}
  s_j \circ v_1 \cdots v_n = v_1 \cdots v_{j-1} v_{j+1} v_j v_{j+2} \cdots v_n.
\end{equation*}
It is known that $\ell(w)$ is equal
to $\inv(w)$, the number of inversions 
in the one-line notation $w_1 \cdots w_n$ of $w$.
%(See, e.g., \cite{BjBrentiBruhat}, \cite{Hum} for defintions.) 
%The ``longest'' element of $\sn$, written $w_0$, has one-line notation
%$n \cdots 1$ and length $n(n-1)/2$. 
The specialization of $\hnq$ at $\qp12 = 1$ is isomorphic to $\zsn$.
Two partial orders arising in the study of $\sn$ are the {\em Bruhat order}
$\leq$ and the {\em weak order} $\wleq$ defined by
\begin{equation}\label{eq:bruweakdef}
  \begin{alignedat}2
    u &\leq v &\quad &\text{if every reduced expression for $v$ contains a reduced expression for $u$,}\\
    u &\wleq v &\quad &\text{if some reduced expression for $v$ ends with a reduced expression for $u$.}
  \end{alignedat}
\end{equation}

In addition to the natural bases of $\zsn$ and $\hnq$, we have the
(signless) {\em Kazhdan-Lusztig bases}~\cite{KLRepCH}
%(modified, signless) {\em Kazhdan-Lusztig bases}~\cite{KLRepCH} 
$\{ C'_w(1) \,|\, w \in \sn \}$,
$\{ C'_w(q) \,|\, w \in \sn \}$,
defined in terms of certain {\em Kazhdan-Lusztig polynomials}
$\{ P_{v,w}(q) \,|\, v,w \in \sn \}$ in $\mathbb N[q]$ by
\begin{equation}\label{eq:KLbasis}
C'_w(1) = \sum_{v \leq w} P_{v,w}(1) v, 
\qquad
C'_w(q) = \qiew \sum_{v \leq w} P_{v,w}(q) T_v,
%C'_w(q) = \qiew \sum_{v \leq w} P_{v,w}(q) T_v.
\end{equation}
where we define
%$\qvw = \qp{\ell(w) - \ell(v)}2$.
$\qew = \qp{\ell(w)}2$.
%(See, e.g., \cite{Hum}.)
%(See, e.g., \cite{BjBrentiBruhat}.) 
%(See \cite{KLRepCH}.)
%We will find it convenient to modify the second basis to obtain
%\begin{equation}\label{eq:modKLbasis}
A modification $\{ \qew C'_w(q) \,|\, w \in \sn \}$
%\subseteq
of the second basis belongs to $\spn_{\mathbb N[q]} \{ T_w \,|\, w \in \sn \}$.
%\end{equation}
%Since $P_{w,w}(q) = 1$ for all $w$, this and the first basis in
%(\ref{eq:KLbasis}) are related to the natural bases of $\hnq$ and $\zsn$
%by unitriangular transition matrices.
%We have the identity $P_{v,w}(q) = 1$ when $w$ \avoidsp,
%i.e., when no subword $w_{i_1} w_{i_2} w_{i_3} w_{i_4}$ of $w_1 \cdots w_n$
%consists of letters which appear in the same relative order as 
%$3412$ or $4231$.
%%satisfies $w_{i_3} < w_{i_4} < w_{i_1} < w_{i_2}$
%%or $w_{i_4} < w_{i_2} < w_{i_3} < w_{i_1}$.  
%(See, e.g., \cite{BilleyLak}.)  ($\star$ Is this fact necessary or even
%useful in the current paper?)

Representations of $\zsn$ and $\hnq$ are often studied in terms of 
%$\mathbb C$- and $\cqq$-linear functions called
$\mathbb Z$- and $\zqq$-linear functionals called {\em characters}.
%(See, e.g., \cite{Sag}.)
The $\mathbb Z$-span of the $\sn$-characters is called the space of
{\em $\sn$-class functions}, and has
dimension equal to the number of integer partitions of $n$.  
%(See, e.g., \cite{Sag}, \cite{StanEC1}, \cite{StanEC2} for definitions.) 
Two well-studied bases 
%whose elements are indexed by partitions $\lambda$ of $n$
are the irreducible characters $\{ \chi^\lambda \,|\, \lambda \vdash n \}$,
and induced sign characters $\{ \epsilon^\lambda \,|\, \lambda \vdash n \}$,
%, and
%induced trivial characters $\{ \eta^\lambda \,|\, \lambda \vdash n \}$,
where $\lambda \vdash n$ denotes that $\lambda$ is a partition of $n$.
The $\zqq$-span of the $\hnq$-characters, called the space of
{\em $\hnq$-traces}, has the same dimension and analogous character bases
% and has
%dimension is equal to the number of integer partitions of $n$.  
%($\star$ See appropriate ref.)  
%Letting $\lambda \vdash n$ denote that $\lambda$ is a partition of $n$,
%we 
%($\star$ Let $\lambda \vdash n$ denote that $\lambda$ is a partition of $n$.)
%Three well-studied bases 
%whose elements are indexed by partitions $\lambda$ of $n$
%are the 
%irreducible characters $\{ \chi_q^\lambda \,|\, \lambda \vdash n \}$,
%induced sign characters $\{ \epsilon_q^\lambda \,|\, \lambda \vdash n \}$, and
%induced trivial characters $\{ \eta_q^\lambda \,|\, \lambda \vdash n \}$,
$\{ \chi_q^\lambda \,|\, \lambda \vdash n \}$,
$\{ \epsilon_q^\lambda \,|\, \lambda \vdash n \}$, 
%$\{ \eta_q^\lambda \,|\, \lambda \vdash n \}$,
specializing at $\smash{\qp12} = 1$ to the $\sn$-character bases.
%(See \cite{CHSSkanEKL} for information on more bases of these spaces.)
In each space, the {\em Kostka numbers}
$\{ K_{\lambda,\mu} \,|\, \lambda, \mu \vdash n \} \subseteq \mathbb N$
describe the expansion of
induced sign characters in the irreducible character basis,
just as they describe the expansion of elementary symmetric functions
in the Schur basis of the space of homogeneous degree $n$ symmetric functions,
\begin{equation*}
  \epsilon^\lambda = \sum_{\mu \vdash n}K_{\mu^\tr,\lambda} \chi^\mu, \qquad
  \epsilon_q^\lambda = \sum_{\mu \vdash n}K_{\mu^\tr,\lambda} \chi_q^\mu, \qquad
  e_\lambda = \sum_{\mu \vdash n}K_{\mu^\tr,\lambda} s_\mu.
\end{equation*}  
(See, e.g., \cite{StanEC2}.)
Here $\mu^\tr$ denotes the transpose or {\em conjugate}
of the partition $\mu$.

%, with
%the irreducible and induced sign character bases
%are related to one another by the same transition matrices
%which relate the Schur and elementary
%, complete homogeneous, monomial,
%and power sum
%bases of the space $\Lambda_n$ of homogeneous degree $n$ symmetric functions,
%\begin{equation*}

%Irreducible and induced sign
The above characters $\theta \in \{ \chi^\lambda, \epsilon^\lambda \}$
of $\sn$
%It is known that irreducible $\sn$-characters 
%$\{ \chi^\lambda \,|\, \lambda \vdash n \}$
satisfy $\theta^\lambda(z) \in \mathbb Z$ for 
all $z \in \zsn$ and $\lambda \vdash n$.
%Thus for any integer linear combination
%$\theta = \sum_\lambda c_\lambda \chi^\lambda$
%of these
%and any element $z \in \zsn$, we have
%$\theta(z) \in \mathbb Z$ as well.
%In some cases, we may
An ideal combinatorial formula for such evaluations
would define sets $R$, $S$
%$R = R(\theta,w)$, $S = S(\theta,w)$
%to each pair
%$(\theta^\lambda, z)$
%to 
%provide a combinatorial interpretation of 
%combinatorially interpret
%the integer 
so that we have
%$\theta^\lambda(z)$
%\begin{equation*}
$\theta(z) = (-1)^{|S|}|R|$,
or simply $\theta(z) = |R|$ if $\theta(z) \in \mathbb N$.
%\end{equation*}
%We summarize 
%known 
For $z$ in the natural or Kazhdan-Lusztig basis of $\zsn$ we have
the following results and open problems.

\vspace{2mm}
%\hhhsp
\hhhpsp
\begin{center}
\newcolumntype{R}{>{$}c<{$}}
%\begin{tabularx}{147.7mm}{|R|R|R|R|R|}%
\begin{tabularx}{148.35mm}{|R|R|R|R|R|}%
\hline
%& & & &\\
& & & &\\
& & & &\\
\theta
%: \zsn \rightarrow \mathbb C 
& \begin{matrix}
    \mbox{Do we have $\phantom{w_{T_i}^{\hat T}}\nTksp\nTksp\nTksp$} \\ 
    \mbox{$\theta(w)^{\phantom{\hat T}}\nTksp \in \mathbb N$} \\
    \mbox{for all $w \in \sn^{\phantom{\hat T}}\nTksp$ \,?}
  \end{matrix} 
& \begin{matrix}
%    \mbox{interpretation of} \\
    \mbox{Can we interpret $\phantom{w^{\hat T}}\nTksp\nTksp\nTksp$} \\ 
    \mbox{$\theta(w)^{\phantom{\hat T}}\nTksp$ as $(-1)^{|S|}|R|$} \\
    \mbox{for all $w \in \sn^{\phantom{\hat T}}\nTksp$ \,?}
  \end{matrix}
& \begin{matrix}
    \mbox{Do we have $\phantom{w_{T_i}^{\hat T}}\nTksp\nTksp\nTksp$} \\ 
    \mbox{$\theta(C'_w(1))^{\phantom{\hat T}}\nTksp \in \mathbb N$} \\
    \mbox{for all $w \in \sn^{\phantom{\hat T}}\nTksp$ \,?}
  \end{matrix}
& \begin{matrix}
    %    \mbox{interpretation of} \\
    \mbox{Can we interpret $\phantom{w^{\hat T}}\nTksp\nTksp\nTksp$} \\ 
%\mbox{$\theta(C'_w(1))$ as $|R|^{\phantom{\hat T}}\nTksp$ for} \\ 
%\mbox{$w$ avoiding $312^{\phantom{\hat T}}\nTksp$?}\end{matrix} \\
    \mbox{$\theta(C'_w(1))$ as $|R|^{\phantom{\hat T}}\nTksp$} \\
    \mbox{for all $w \in \sn^{\phantom{\hat T}}\nTksp$ \,?}
%    \mbox{for all $w^{\phantom{\hat T}}\nTksp$ avoiding} \\    
%    \mbox{$3412^{\phantom{\hat T}}\nTksp$ and $4231$?}
  \end{matrix} \\
& & & &\\
& & & &\\
\hline 
%& & & &\\
& & & &\\
& & & &\\
%& & & &\\
%\eta^\lambda & \mbox{yes} & \mbox{yes} & \mbox{yes} & \mbox{yes} \\
%& & & &\\
%& & & &\\
\epsilon^\lambda & \mbox{no} & \mbox{yes} & \mbox{yes} & \mbox{open} \\
& & & &\\
& & & &\\
\chi^\lambda & \mbox{no} & \mbox{open} & \mbox{yes} & \mbox{open}\\
& & & &\\
& & & &\\
%\psi^\lambda & \mbox{yes}  & \mbox{yes} & \mbox{yes} & \mbox{yes} \\
%& & & &\\
%& & & &\\
%\phi^\lambda & \mbox{no} & \mbox{yes} & \mbox{conj.\ by Stembridge, Haiman} & \mbox{open} \\
%& & & &\\
%& & & &\\
\hline
\end{tabularx}
\end{center}
\vspace{2mm}
\ssp

Of the four combinatorial interpretations asked for in the above table,
the only one which is known may be described as follows.
Let $\lambda = (\lambda_1, \dotsc, \lambda_r)$.  Then we have
\begin{equation*}
  \epsilon^\lambda(w) = (-1)^{|S(w)|}|R(w,\lambda)|,
\end{equation*}
where $S(w)$ is the set of inversions in the one-line notation of $w$
and $R(w,\lambda)$ is the set of labelings of the cycles of $w$ by
$1, \dotsc, r$ such that exactly $\lambda_i$ letters
are contained in the cycles labeled $i$
(\emph{cf.} formula for $M(f,p)_{\lambda,\mu}$ in \cite[p.\,9]{RemTrans}).
For example, consider the partition
$\lambda = (5,4) \vdash 9$ and permutation $w = 234167589 \in \mfs 9$
%which satisfies
with $\inv(w) = 5$.
Writing $w$ in cycle notation as $(1,2,3,4)(5,6,7)(8)(9)$,
%and computing $\inv(w) = 4$,
we may label the cycles in three ways so that $\lambda_1 = 5$ letters
belong to cycles labeled $1$ and $\lambda_2 = 4$ letters belong to
cycles labeled $2$:
\begin{equation*}
  \begin{matrix}
    \ul{(1,2,3,4)} & \ul{(5,6,7)} & \ul{(8)} & \ul{(9)} \\
    1              & 2            & 1        & 2,   \\
    1              & 2            & 2        & 1,   \\
    2              & 1            & 1        & 1.
  \end{matrix}
\end{equation*}
Thus we have
%\begin{equation*}
  $\epsilon^{\lambda}(w) = (-1)^5 3 = -3$.
%\end{equation*}

The number $\chi^\lambda(w)$ may be computed by the well-known
%algorithm
Murnaghan-Nakayama algorithm 
%Otherwise, $\chi^\lambda(w)$ 
but has no conjectured expression 
of the type stated above.
(See, e.g., \cite{StanEC2}.)  
Interpretations of
%$\theta(C'_w(1))$
$\epsilon^\lambda(C'_w(1))$ and
$\chi^\lambda(C'_w(1))$
are not known for general $w \in \sn$,
but nonnegativity follows from work of 
Haiman~\cite{HaimanHecke} and Stembridge~\cite{StemImm}.
%Interpretations of
%$\epsilon^\lambda(C'_w(1))$ and 
%$\chi^\lambda(C'_w(1))$
%for
In the special case that $w$ \avoidsp,
%avoiding the 
%patterns $3412$, $4231$
%pattern $312$
interpretations of these numbers
are given in \cite[Thm.\,4.7]{CHSSkanEKL}.
We say that $w \in \sn$ {\em avoids the pattern}
$p_1 \cdots p_k \in \mfs k$ if no subsequence $(w_{i_1}, \dotsc, w_{i_k})$
of $w_1 \cdots w_n$ consists of letters appearing in the same
relative order as $p_1 \cdots p_k$.

%$\star$ Mention results for \pavoiding special cases of the last column. 
%$\star$ Include more on the Hecke algebra, including $T_s T_w$ cases?

%$\star$ Define traces.

%$\star$ Include some table about what traces we can evaluate on
%which elements.

The characters $\theta_q \in \{ \chi_q^\lambda, \epsilon_q^\lambda \}$
of $\hnq$
%It is known that irreducible $\sn$-characters 
%$\{ \chi^\lambda \,|\, \lambda \vdash n \}$
satisfy $\theta_q(z) \in \mathbb Z[q]$ for 
all $z \in \hnq$ and $\lambda \vdash n$.
%It is known that 
%irreducible $\hnq$-characters 
%$\{ \chi_q^\lambda \,|\, \lambda \vdash n \}$
%satisfy $\chi_q^\lambda(T_w) \in \mathbb Z[q]$ for all $w \in \sn$.
%Thus for any integer linear combination
%$\theta_q = \sum_\lambda c_\lambda \chi_q^\lambda$ of these
%the irreducibleany 
%$\hnq$-characters, 
%and any element 
%$z \in \spn_{\mathbb Z[q]} \{ T_w \,|\, w \in \sn \}$, we have
%$\theta_q(z) \in \mathbb Z[q]$ as well.
An ideal combinatorial formula for such evaluations
would define
sequences $(S_k)_{k \geq 0}$, $(R_k)_{k \geq 0}$ of sets
%sets $R$, $S$
so that we have
%In some cases, we may associate
%to the pair $(\theta_q, z)$
%to combinatorially interpret
%provide a combinatorial interpretation of the polynomial 
%$\theta_q(z)$ as
%\begin{equation*}
$\theta_q(z) = 
\sum_k (-1)^{|S_k|}|R_k| q^k$,
or simply $\theta_q(z) = 
\sum_k |R_k| q^k$ if $\theta_q(z) \in \mathbb N[q]$.
%\end{equation*}
%We summarize 
%known 
%results and open problems in the following table.
For $z$ in the natural basis
or modified Kazhdan-Lusztig basis of $\hnq$ we have
the following results and open problems.

\vspace{2mm}
\hhhpsp
\begin{center}
\newcolumntype{R}{>{$}c<{$}}
%\begin{tabularx}{163.9mm}{|R|R|R|R|R|}%
%\begin{tabularx}{154.09mm}{|R|R|R|R|R|}%
%\begin{tabularx}{155.8mm}{|R|R|R|R|R|}%
\begin{tabularx}{151.8mm}{|R|R|R|R|R|}%
%\begin{tabularx}{\linewidth}{|R|R|R|R|R|}%
\hline
& & & &\\
& & & &\\
\theta_q \ntnsp
& \begin{matrix}
    \mbox{Do we have $\phantom{w_{T_i}^{\hat T}}\nTksp\nTksp\nTksp$} \\ 
    \mbox{$\theta_q(T_w)^{\phantom{\hat T}}\nTksp \in \mathbb N[q]$}\\
    \mbox{for all $w \in \sn^{\phantom{\hat T}}\nTksp$ \,?}
  \end{matrix}
& \begin{matrix}
%    \mbox{interpretation of}\\
    \mbox{Can we interpret $\phantom{w^{\hat T}}\nTksp\nTksp\nTksp$} \\
    \mbox{$\theta_q(T_w)^{\phantom{\hat T}}\nTksp$ as}\\
    \mbox{$\sum_k^{\phantom{\hat T}} (-1)^{|S_k|}|R_k|q^k$}\\
    \mbox{for all $w \in \sn^{\phantom{\hat T}}\nTksp$ \,?}
  \end{matrix}\ntksp
& \begin{matrix}
    \mbox{Do we have $\phantom{w_{T_i}^{\hat T}}\nTksp\nTksp\nTksp$} \\
    \mbox{$\theta_q(\qew C'_w(q))^{\phantom{\hat T}}\nTksp \in \mathbb N[q]$}\\
    \mbox{for all $w \in \sn^{\phantom{\hat T}}\nTksp$ \,?}
  \end{matrix}
& \begin{matrix}
%    \mbox{interpretation of}\\ 
    \mbox{Can we interpret $\phantom{w^{\hat T}}\nTksp\nTksp\nTksp$} \\
    \mbox{$\theta_q(\qew C'_w(q))^{\phantom{\hat T}}\nTksp$ as}\\
    \mbox{$\sum_k^{\phantom{\hat T}}\ntksp|R_k|q^k$}\\
    \mbox{for all $w \in \sn^{\phantom{\hat T}}\nTksp$ \,?}
%\mbox{$w$ avoiding $312^{\phantom{\hat T}}\nTksp$?}\end{matrix} \\
%\mbox{$w^{\phantom{\hat T}}\nTksp$ avoiding}\\ 
    %\mbox{$3412^{\phantom{\hat T}}\nTksp$ and $4231$?}
  \end{matrix} \\
& & & &\\
& & & &\\
\hline 
%& & & &\\
%& & & &\\
%\eta_q^\lambda & \mbox{no} & \mbox{open} & \mbox{yes} & \mbox{stated in Section 5}\\
& & & &\\
& & & &\\
%\vspace{1mm} 
\epsilon_q^\lambda & \mbox{no} & \mbox{open} & \mbox{yes} & \mbox{open}\\
& & & &\\
& & & &\\
%\vspace{1mm} 
\chi_q^\lambda \ntksp & \mbox{no} & \mbox{open} & \mbox{yes} & \mbox{open}\\
& & & &\\
& & & &\\
\hline
\end{tabularx}
\end{center}
\vspace{2mm}
\ssp

The polynomial $\chi_q^\lambda(T_w)$, 
and therefore $\epsilon_q^\lambda(T_w)$,
may be computed via a $q$-extension of the Murnaghan-Nakayama algorithm.
(See, e.g.,
%\cite{KerVerCFR}, \cite{KingWRepTr},
\cite{RamFrob}.)
However,
%Otherwise, $\theta_q^\lambda(T_w)$ 
neither of these has a conjectured expression of the type asked for above.
Interpretations of
%$\theta_q(\qew C'_w(q))$
$\epsilon_q^\lambda(\qew C'_w(q))$ and
$\chi_q^\lambda(\qew C'_w(q))$ 
are not known for general $w \in \sn$,
but results concerning containment in $\mathbb N[q]$ 
follow from work of Haiman~\cite{HaimanHecke}.
%In both cases, coefficients of these polynomials
%are unimodal and
%symmetric about $\qew$~\cite[Lem.\,1.1]{HaimanHecke}.
In the special case that $w$ \avoidsp,
formulas for these polynomials
%$\eta_q^\lambda(\qew C'_w(q))$, 
%$\epsilon_q^\lambda(\qew C'_w(q))$ and
%$\chi_q^\lambda(\qew C'_w(q))$
%$\psi_q^\lambda(\qew C'_w(q))$ 
are given in
\cite[Thms.\ 6.4, 8.1]{CHSSkanEKL}.

To obtain ideal combinatorial interpretations analogous to those asked for
above, we will consider the infinite spanning set of $\hnq$ which consists
of all elements of the form 
%third basis of $\hnq$:
%To describe yet another basis of $\hnq$ let us consider products of the form
%It is easy to see that the infinite set of elements
\begin{equation}\label{eq:wdelt}
(1 + T_{s_{i_1}}) \cdots (1 + T_{s_{i_m}}) = \qp m2 C'_{s_{i_1}}(q) \cdots C'_{s_{i_m}}(q), %\,|\, m \geq 0; 0 \leq i_1, \dotsc, i_m \leq n-1 \} 
\end{equation}
%of $\hnq$.
where
%the index sequence
$\sprod im$ varies over
all products of generators of $\sn$.
%It is well known that the
It is easy to see that if we arbitrarily choose
%exactly
one reduced expression
for each element of $\sn$, then the $n!$ corresponding products
(\ref{eq:wdelt}) form a basis for $\hnq$.
Different collections of reduced expressions can yield different bases.
%Thus the products corresponding to all possible
%reduced expressions, or even all possible expressions, form a spanning set
%of $\hnq$.  We summarize facts about evaluations of characters at these
%elements in the following table.
For $z$ belonging to the above spanning set,
%in the natural or Kazhdan-Lusztig basis of $\hnq$
we have the following results and open problems.

\vspace{2mm}
\hhhpsp
\begin{center}
\newcolumntype{R}{>{$}c<{$}}
%\begin{tabularx}{163.9mm}{|R|R|R|R|R|}%
\begin{tabularx}{132.5mm}{|R|R|R|}%
%\begin{tabularx}{\linewidth}{|R|R|R|R|R|}%
\hline
& &\\
& &\\
\theta_q \ntnsp
& \begin{matrix}
    \mbox{Do we have $\phantom{w_{T_i}^{\hat T}}\nTksp\nTksp\nTksp$} \\ 
    \mbox{$\theta_q((1 + T_{s_{i_1}}) \cdots (1 + T_{s_{i_m}}))^{\phantom{\hat T}}\nTksp \in \mathbb N[q]$}\\
%    \mbox{for all $s_{i_1} \cdots s_{i_m}^{\phantom{\hat T}}\nTksp$ \,?}
    \mbox{for all $\sprod im^{\phantom{\hat T}}\nTksp$ \,?}
  \end{matrix}
& \begin{matrix}
    \mbox{Can we interpret $\phantom{w^{\hat T}}\nTksp\nTksp\nTksp$} \\
%    \mbox{interpretation of}\\ 
    \mbox{$\theta_q((1+T_{s_{i_1}}) \cdots (1 + T_{s_{i_m}}))^{\phantom{\hat T}}\nTksp$ as}\\
%    \mbox{$\sum_k^{\phantom{\hat T}} |R_k|q^k$}\\
%    \mbox{for all $s_{i_1} \cdots s_{i_m}^{\phantom{\hat T}}\nTksp$ \,?}
%    \mbox{$\sum_k^{\phantom{\hat T}} |R_k|q^k$ for all $s_{i_1} \cdots s_{i_m}^{\phantom{\hat T}}\nTksp$ \,?}
    \mbox{$\sum_k^{\phantom{\hat T}} |R_k|q^k$ for all $\sprod im^{\phantom{\hat T}}\nTksp$ \,?}
%    \mbox{for all $w \in \sn^{\phantom{\hat T}}\nTksp$ \,?}
%  \end{matrix}\ntksp
  \end{matrix} \\
& &\\
& &\\
\hline 
%& & & &\\
%& & & &\\
%\eta_q^\lambda & \mbox{no} & \mbox{open} & \mbox{yes} & \mbox{stated in Section 5}\\
& &\\
& &\\
%\vspace{1mm} 
\epsilon_q^\lambda & \mbox{yes} & \mbox{Stated in Section 5}\\
& &\\
& &\\
%\vspace{1mm} 
\chi_q^\lambda \ntksp & \mbox{yes} & \mbox{open}\\
& &\\
& &\\
\hline
\end{tabularx}
\end{center}
\vspace{2mm}
\ssp

Results concerning containment in $\mathbb N[q]$ 
follow from work of Haiman, since
every product of the form $\qeu C'_u(q) \qev C'_v(q)$ belongs
to $\spn_{\mathbb N[q]} \{ \qew C'_w(q) \,|\, w \in \sn \}$.
(See \cite[Appendix]{HaimanHecke}.)
Interpretation of the polynomials
$\epsilon_q^\lambda((1 + T_{s_{i_1}}) \cdots ( 1 + T_{s_{i_m}}))$
is new, and is the first result of its kind to include evaluation
of all elements of a basis of the $\hnq$-trace space
at all elements of a basis of $\hnq$.
Its justification depends upon the transition matrices which
relate natural bases of Drinfeld's quantum matrix bialgebra,
and an identity in this bialgebra which was stated by
Konvalinka and the third author~\cite[Thm.\,5.4]{KSkanQGJ}.

%($\star$ Rewrite this.)
%In Sections~\ref{s:prelim} -- \ref{s:qmb}
In Section~\ref{s:qmb}
%--\ref{s:mbases} 
we introduce the quantum matrix bialgebra $\A$
%= \anq$
and prove combinatorial formulas for the entries of transition matrices
that relate monomial bases of the zero-weight space of $\A$.
%review important facts about $\sn$.
%In Section~\ref{s:mbases} we state and prove combinatorial
%formulas for the entries of transition matrices that relate
%various monomial bases of the zero-weight space of $\A$.
In Section~\ref{s:planar}
%--\ref{s:evalthm}
we define a function $\sigma: \A \rightarrow \zqq$
which allows us to compute
%\begin{equation*}
  $\theta_q((1 + T_{s_{i_1}}) \cdots (1 + T_{s_{i_m}}))$
%\end{equation*}
for any linear function $\theta_q: \hnq \rightarrow \zqq$
in terms of a generating function in $\A$ for $\theta_q$
and a wiring diagram for the product $\sprod im$.
%($\star$ Is that an exaggeration?)
In Sections~\ref{s:tabx}--\ref{s:indsgneval} we use the map $\sigma$
to combinatorially evaluate induced sign characters of $\hnq$
at all elements of the spanning set (\ref{eq:wdelt}).
We finish with some open problems in Section~\ref{s:probs}.

%In Section~\ref{s:pathposet} we introduce a partial order on paths in
%zig-zag networks, and show how this poset and results in the literature lead 
%to combinatorial interpretations for the evaluations
%of $\sn$-class functions at the above Kazhdan-Lusztig basis elements
%of $\zsn$.
%For the remainder of the article,
%we concentrate on combinatorial interpretations of trace evaulations
%of the form $\theta_q(\qew C'_w(q))$ where $w$ \avoidsp.
%In Sections~\ref{s:hnqinterph} -- \ref{s:hnqinterpe} 
%we interpret
%$\eta_q^\lambda(\qew C'_w(q))$
%and $\epsilon_q^\lambda(\qew C'_w(q))$. 
%In Section~\ref{s:chromsf} we recall the relationship between
%$\sn$-class functions and Stanley's chromatic symmetric functions,
%and prove that $\hnq$-traces are similarly related to the
%Shareshian-Wachs chromatic quasisymmetric functions.
%In Sections~\ref{s:hnqinterps} -- \ref{s:hnqinterpp} 
%we use results of Shareshian-Wachs and Athanasiadis to
%interpret
%$\chi_q^\lambda(\qew C'_w(q))$
%and $\psi_q^\lambda(\qew C'_w(q))$. 
%One of our interpretations proves a formula conjectured by 
%Haiman~\cite[Conj.4.1]{HaimanHecke}.
%Finally, in Section~\ref{s:hnqinterpm} we state several results 
%concerning 
%$\phi_q^\lambda(\qew C'_w(q))$.

\section{Bases of the zero-weight space of the quantum matrix bialgebra}\label{s:qmb}
%\section{The quantum polynomial ring $\anq$}

%($\star$ Say something about why this bialgebra is important in general.)
The study of quantum groups in the 1980s led to the study of
algebras of functions on these and to the related
%The
\emph{quantum matrix bialgebra} $\A = \anq$.
%The \emph{quantum polynomial ring}
$\A$
is the associative algebra with unit $1$ generated over $\zqq$
%is generated as a noncommutative $\cqq$-algebra 
%generated
by $n^2$ variables $x=(x_{1,1},\dots,x_{n,n})$,
subject to the relations
\begin{equation}\label{eq:anqdef}
\begin{alignedat}{2}
x_{i,\ell}x_{i,k} &= \qp12x_{i,k} x_{i,\ell}, &\qquad 
x_{j,k} x_{i,\ell} &= x_{i,\ell}x_{j,k},\\
x_{j,k}x_{i,k} &= \qp12x_{i,k} x_{j,k}, &\qquad 
x_{j,\ell} x_{i,k} &= x_{i,k} x_{j,\ell} + (\qdiff) x_{i,\ell}x_{j,k},
\end{alignedat}
\end{equation}
for all indices $ 1 \leq i < j \leq n $ and 
$ 1 \leq k < \ell \leq n $.
The counit and coproduct maps
\begin{equation*}
\varepsilon(x_{i,j}) = \delta_{i,j}, \qquad 
\Delta(x_{i,j}) = \sum_{k=1}^n x_{i,k} \otimes x_{k,j}
\end{equation*}
give $\A$ a bialgebra structure.  
While $\A$ is not a Hopf algebra, 
two Hopf algebras 
closely related to it
are the quantum coordinate rings of $SL_n(\mathbb C)$ and $GL_n(\mathbb C)$,
%Each of these is isomorphic to a quotient of $\A$,
\begin{equation*}
\oqslnc \cong \mathbb C \otimes \A/(\qdet(x)-1),
\qquad
\oqglnc \cong \mathbb C \otimes \A[t]/(\qdet(x)t-1),
\end{equation*}
where
\begin{equation}\label{eq:qdetdef}
\qdet(x) \defeq \sum_{v \in \sn} (-\qm 12)^{\ell(v)} \permmon xv
= \sum_{v \in \sn} (-\qm 12)^{\ell(v)} x_{v_1,1} \cdots x_{v_n,n}
\end{equation}
is the ($n \times n$) {\em quantum determinant} of the matrix $x = (x_{i,j})$.  
(The second equality holds in $\A$ but not in the noncommutative ring
$\zqq\langle x_{1,1}, \dotsc, x_{n,n} \rangle$.)
%, which follows from the relations (\ref{eq:anqdef}),
%does not hold in an arbitrary noncommutative ring in $n^2$ variables.)
The antipode maps of these Hopf algebras are 
\begin{equation*}
\mathcal S (x_{i,j}) = (-\qp12)^{j-i} \qdet(x_{[n]\ssm\{j\}, [n]\ssm \{i\}}),
\qquad
\mathcal S (x_{i,j}) = 
\frac{(-\qp12)^{j-i}\qdet(x_{[n]\ssm\{j\}, [n]\ssm \{i\}})}{\qdet(x)},
\end{equation*}
respectively, where
%we define
\begin{equation}\label{eq:submatrixdef}
%  \begin{gathered}
    [n] \defeq \{ 1, \dotsc, n \}, \qquad
    x_{L,M} \defeq (x_{\ell,m})_{\ell \in L, m \in M},
\end{equation}
%and where $x_{I,J}$
%$x_{[n] \ssm \{j\}, [n] \ssm \{i\} }$ is the $(n-1) \times (n-1)$
%submatrix of $x$ obtained by deleting row $j$ and column $i$,
and %the $(n-1) \times (n-1)$ quantum determinant
$\qdet(x_{L,M})$ is defined analogously to (\ref{eq:qdetdef}),
assuming $|L| = |M|$.
%in the obvious way.  
%($\star$ Cite Manin? or Soibelman and Vaksman?)
%More generally, for any row subset $I$ and column subset $J$,
%we will denote the $I,J$ submatrix of $x$ by $x_{I,J}$.
Specializing $\A$ at $\qp12 = 1$, we obtain 
%Notice that $ \mathcal{A}(n;1) $ is
the commutative 
%polynomial 
ring $\mathbb{Z}[x_{1,1}, \dotsc, x_{n,n}]$.
%While $\anq$ is not a Hopf algebra,

$\A$ has a natural $\zqq$-basis 
$\{ x_{1,1}^{a_{1,1}} \cdots x_{n,n}^{a_{n,n}} \,|\, 
a_{1,1}, \dotsc, a_{n,n} \in \mathbb{N} \}$
of monomials in which variables 
appear in lexicographic order, and 
%  We can use
the relations (\ref{eq:anqdef})
provide an algorithm for expressing any other
% any non-lexicographically ordered 
monomial in terms of this basis.
(See, e.g., \cite[Lem.\,2.1]{ZhangHowe}.)
%We can use the relations (\ref{eq:anqdef}) to express any monomial as a linear
%combination of monomials in lexicographic order, which we shall call
%\emph{standard}. ($\star$ Do we actually use that term later, or not?)
%Thus as a $\mathbb{C}[\qp12, \qm12]$-module, 
%$\anq$ is spanned by monomials whose variables appear in lexicographic order. 
%$\A$ has a natural grading by degree,
%\begin{equation}
%\A = \bigoplus_{r \geq 0} \A_r,
%\end{equation}
%where $\A_r$
%consists of polynomials which are homogeneous of degree $r$.
%We may further decompose each
%homogeneous component 
%$\arnq$ by considering pairs $(L,M)$ of multisets of $r$
%integers, written as weakly increasing words
%\begin{equation*}
%L = \ell_1 \cdots \ell_r, \quad M = m_1 \cdots m_r.
%\end{equation*}
%This leads to the multigrading
%\begin{equation}
%\A = \bigoplus_{r \geq 0} \bigoplus_{L,M} \ALM,
%\end{equation}
%where $\A_{L,M}$ is the $\zqq$-span of monomials whose row indices and
%column indices (with multiplicity) are equal to the multisets $L$
%and $M$, respectively. 
The submodule $\annnq$
%Thus the multigraded component $\annnq$ is
spanned by the monomials
\begin{equation}
\{ x^{u,v} \defeq x_{u_1,v_1} \cdots x_{u_n,v_n} \mid u,v \in \sn \}
%\{x_{1,w_1} \cdots x_{n,w_n} \mid w \in \sn \}.
\end{equation}
is called the {\em zero-weight space} of $\A$ and
%By (\ref{eq:anqdef})
%By, e.g., \cite{Brundan}
%This component 
has the natural basis $\{ x^{e,w} \,|\, w \in \sn \}$.
The relations (\ref{eq:anqdef}) imply 
%It is easy to see 
that the monomials $\{ x^{u,v} \,|\, u,v \in \sn \}$
satisfy
\begin{equation}\label{eq:switchsides}
x^{u,v} = \begin{cases}
x^{s_iu,s_iv} 
&\text{if $s_iu < u$ and $s_iv > v$},\\ 
x^{s_iu,s_iv} + (\qdiff) x^{s_iu,v} &\text{if $s_iu<u$ and $s_iv<v$}.
%x^{s_iu,v} = \begin{cases}
%x^{u,s_iv} &\text{if $s_iu > u$ and $s_iv > v$},\\ 
%^{u,s_iv} + (\qdiff) x^{u,v} &\text{if $s_iu>u$ and $s_iv<v$},
%x^{su,v} = \begin{cases}
%x^{u,sv} &\text{if $su > u$ and $sv > v$},\\ 
%x^{u,sv} + (\qdiff) x^{u,v} &\text{if $su>u$ and $sv<v$},
\end{cases}
\end{equation}
%($\star$ change to $x^{u,v} = $ ?) 
It follows that for fixed $t, u, v \in \sn$ satisfying $t \wleq u$, 
%and therefore that
each monomial $x^{u,v}$ belongs to $\sum_w \mathbb N[\qdiff] x^{t,w}$.
%\begin{equation}\label{eq:switchsides}
%x^{s_iu,v} = \begin{cases}
%x^{u,s_iv} &\text{if $s_iu > u$ and $s_iv > v$, 
%or if $s_iu < u$ and $s_iv < v$},\\
%x^{u,s_iv} + (\qdiff) x^{u,v} &\text{if $s_iu>u$ and $s_iv<v$},\\
%x^{u,s_iv} - (\qdiff) x^{u,v} &\text{if $s_iu<u$ and $s_iv>v$}.
%\end{cases}
%\end{equation}
%Thus for all $w \in \sn$ we have the identity
%\begin{equation}\label{eq:xwandinverse}
%x^{w,e} = x^{e,w^{-1}}.
%\end{equation}
%On the other hand, we do not in general have the equality of $x^{v,w}$ 
%and $x^{w^{-1}, v^{-1}}$.
%Applying (\ref{eq:switchsides}) recursively to express $x^{v,w}$ in terms
%of the natural basis, 
%we obtain an expression of the form
%%%%%
%\begin{equation}\label{eq:straighten}
%\sum_{w \geq u^{-1}v} \ntksp\ntksp \mathbb N[\qdiff] x^{e,w}
%\quad \bigcap \quad
%\ntksp\ntksp \sum_{w \geq u^{-1}v} \ntksp\ntksp 
%q_{u^{-1}v,w}^{-1}
%\mathbb Z[q] x^{e,w}
%\quad \bigcap \quad
%\ntksp\ntksp \sum_{w \geq u^{-1}v} \ntksp\ntksp 
%q_{u^{-1}v,w}
%\mathbb Z[q^{-1}] x^{e,w}.
%\end{equation}
%($\star$ How do we know that we can write $w \geq u^{-1}v$?)
In particular, since nonnegative powers of $\qdiff$ are linearly independent,
there are uniquely defined polynomials
$\{r_{u,v,t,w}(q_1)\mid w\in\sn\}$ in $\mathbb{N}[q_1]$
which satisfy
\begin{equation}\label{eq:rdefn}
x^{u,v} = \sum_{w \in \sn} r_{u,v,t,w}(\qp12 - \qm12)x^{t,w}.
%x^{u,v} = \sum_{w \geq u^{-1}v} r_{u,v,e,w}(\qp12 - \qm12)x^{e,w}
\end{equation}
%In order to state some properties of these polynomials,
%and combinatorially interpret their coefficients,
%we recall a well-known property of the
%Bruhat order (See \cite[Lem.\,2.2.10]{BBCoxeter}).
%\begin{prop}\label{p:midascent}
%Fix elements $u,v \in \sn$ and a generator $s$.
%($\star$ Can it be any transposition, or must it be an adjacent transposition?)
%\begin{enumerate}
%\item
%If $us < u$ and $sv < v$ then $uv < usv$.
%\item If $us > u$ and $sv > v$ then $uv < usv$.
%\end{enumerate}
%\end{prop}
%Now we show that for fixed $u \in \sn$, the set $\{ x^{u,v} \,|\, v \in \sn \}$
%is a basis of $\Ann$.
Some of these polynomials are identically $1$ or $0$.
\begin{prop}\label{p:ruvtw}
For fixed $t, u, v \in \sn$ with $t \wleq u$, the polynomials 
$\{r_{u,v,t,w}(q_1)\mid w\in\sn\}$ in $\mathbb{N}[q_1]$
satisfy
%\begin{enumerate}
%\item $r_{u,v,e,u^{-1}v}(q_1) = 1$.
%\item 
$r_{u,v,t,w}(q_1) = 0$ unless $w \geq tu^{-1}v$,
and $r_{u,v,t,tu^{-1}v}(q_1) = 1$, i.e.,
\begin{equation*}
  x^{u,w}
= x^{t, tu^{-1}w} + 
\ntksp\ntksp \sum_{v > tu^{-1}w} \ntksp\ntksp 
r_{u,w,t,v}(\qdiff) x^{t,v}.
\end{equation*}
%\end{enumerate}
\end{prop}
%\begin{proofidea}
%  Here is the idea.
%\end{proofidea}
%\begin{proofomitted}
%\end{proofomitted}
\begin{proof}
By definition we have $r_{t,v,t,w}(q_1) = \delta_{v,w}$.
%($\star$ See Justin, Props 1 and 2.)
Thus the claim holds when $u = t$.
Now fix $t$ and assume that the claim holds for $\ell(u) \leq \ell(t)+ k-1$,
consider $u$ of length $\ell(t) + k \geq \ell(t) + 1$,
and let $s$ be a left descent of $ut^{-1}$ and therefore of $u$.
It follows that we have $t \wleq su$.

By (\ref{eq:switchsides}) and the linear independence of powers of $\qdiff$,
we have 
\begin{equation}\label{eq:rrecur}
r_{u,v,t,w}(q_1) = \begin{cases}
r_{su,sv,t,w}(q_1) &\text{if $sv > v$,}\\
r_{su,sv,t,w}(q_1) + (q_1) r_{su,v,t,w}(q_1) &\text{if $sv < v$}.
\end{cases}
\end{equation}
%that $x^{u,v} = x^{su, sv}$.
Suppose that $w \not \geq tu^{-1}v$. Since
$t(su)^{-1}sv = tu^{-1}v$ and $\ell(su) = \ell(t) + k-1$, we have
by induction that
%, the condition
%we have
%\begin{equation*}
%\begin{gathered}
%$w \ngeq u^{-1}v = (su)^{-1}sv$
%implies that
% \quad \Rightarrow \quad
%\begin{equation*}
  %r_{u,v,e,w}(q_1) =
$r_{su,sv,t,w}(q_1) = 0$
%r_{u,v,e,w}(q_1)
%\end{equation*}
in both cases of (\ref{eq:rrecur}).
%of the above cases.
Furthermore since $tu^{-1}s < tu^{-1}$,
we have by \cite[Lem.\,2.2.10]{BBCoxeter}
%($\star$ Cite Bjorner and Brenti? $us < u$ and $sv < v$ imply $uv < usv$.)
%Proposition~\ref{p:midascent}
that $tu^{-1}sv > tu^{-1}v$ when $sv < v$.
Thus in the second case above, the condition $w \ngeq tu^{-1}v$ also implies
that $w \ngeq tu^{-1}sv$, which by induction implies that 
$r_{su,v,t,w}(q_1) = 0$.
Therefore we have $r_{u,v,t,w}(q_1) = 0$ unless $w \geq tu^{-1}v$.

Setting $w = tu^{-1}v$ in (\ref{eq:rrecur}) we have
%\begin{equation*}
$r_{su,sv,t,tu^{-1}v}(q_1) = 1$
%\end{equation*}
in both cases.
When $sv < v$ we also have
$r_{su,v,t,tu^{-1}v}(q_1) = 0$
since $tu^{-1}v \ngeq tu^{-1}sv$.
Therefore we have $r_{u,v,t,tu^{-1}v}(q_1) = 1$.
%\begin{equation}\label{eq:rrecur2}
%r_{u,v,e,u^{-1}v}(q_1) = \begin{cases}
%r_{su,sv,e,u^{-1}v}(q_1) &\text{if $sv > v$,}\\
%r_{su,sv,e,u^{-1}v}(q_1) + q_1 r_{su,v,e,u^{-1}v}(q_1) &\text{if $sv < v$}.
%\end{cases}
%\end{equation}
%,\\
%  We also have
%r_{u,v,e,u^{-1}v}(q_1) = r_{su,sv,e,(su)^{-1}sv}(q_1) = 1.
%\end{gathered}
%\end{equation*}
%Assume that $sv > v$. 
%Now assume that $sv < v$.  
%Thus
%Since $(su)^{-1}sv = u^{-1}v$, 
%we have
%\begin{equation*}
%$r_{u,v,e,w}(q_1) = r_{su,sv,e,w}(q_1) = 0$ unless $w \geq (su)^{-1}sv = u^{-1}v$,
%and $r_{u,v,e,u^{-1}v}(q_1) = r_{su,sv,e,(su)^{-1}sv}(q_1) = 1$.
%Now assume that $sv < v$.
%If $w \ngeq u^{-1}v$, then by induction $r_{su,sv,e,w}(q_1) = 0$.
%By Theorem~\ref{p:midascent}, we have $u^{-1}sv > u^{-1}v$,
%thus $w \ngeq u^{-1}v$ also implies that $w \ngeq u^{-1}sv$ and we have
%by induction that $r_{su,v,e,w}(q) = 0$.  It follows that $r_{u,v,e,w}(q_1) = 0$.
\end{proof}
%Furthermore, we have the following.
%\begin{prop}\label{r
%and $r_{u,v,e,u^{-1}v}(q_1) = 1$.  ($\star$ Doublecheck this $=1$ part and explain
%more if necessary.)
%Thus 
\begin{cor}\label{c:xuvbases} 
For each fixed $u \in \sn$, the set $\{ x^{u,v} \,|\, v \in \sn \}$
is a basis for $\annnq$.
%, related to the natural basis 
%$\{x^{e,w} \,|\, w \in \sn \}$ by
%%the transition matrix whose $v,w$ entry
%the $n! \times n!$ transition matrix whose $v,w$ entry 
%is $\smash{r_{u,v,e,w}(\qp12 - \qm12)}$.
\end{cor}
\begin{proof}
  Setting $t = e$ in (\ref{eq:rdefn}) and applying Proposition~\ref{p:ruvtw}
  we have
  \begin{equation*}
    x^{u,v} = x^{e,u^{-1}v} + \sum_{w > u^{-1}v} r_{u,v,e,u^{-1}v}(\qdiff)x^{e,w}.
  \end{equation*}
  Now ordering the monomials $x^{u,v^{(1)}},\dotsc,x^{u,v^{(n!)}}$ so
  that $u^{-1}v^{(1)},\dotsc,u^{-1}v^{(n!)}$ is a linear extension
  of the Bruhat order, we have a unitriangular system of equations.
  \end{proof} 
%  Let $(w^{(1)}, \dotsc, w^{(n!)})$ be a linear extension of the Bruhat order,
%  and define the sequence $(v^{(1)}, \dotsc, v^{(n!)})$ by $v^{(i)} = uw^{(i)}$.
%  Expanding each monomial $x^{u,v^{(i)}}$ in the (ordered) natural basis
%  $(x^{e,w^{(1)}}, \dotsc, x^{e,w^{(n!)}})$ we obtain $n!$ equations of the form
%  \begin{equation*}
%    x^{u,v^{(i)}} =
%    r_{u,v^{(i)},e,w^{(1)}}(\qdiff) x^{e,w^{(1)}} + \cdots +
%    r_{u,v^{(i)},e,w^{(n!)}}(\qdiff) x^{e,w^{(n!)}}.
%  \end{equation*}
%  By Proposition~\ref{p:ruvtw} we have that
%  $r_{u,v^{(i)},e,w^{(j)}}(\qdiff) = 0$
%  when $w^{(j)} \not \geq u^{-1}v^{(i)} = w^{(i)}$, in particular when $j < i$.
%  We also have $r_{u,v^{(i)},e,w^{(i)}}(\qdiff) = 1$.
%  Thus our system of equations is unitriangular and therefore invertible.
%\end{proof}
By the unitriangularity of the coefficient matrix in the above proof,
we may extend the statement containing (\ref{eq:rdefn}).
% \,|\, v,w \in \sn \}$ are the entries of a
%Since the columns of each matrix in Corollary~\ref{c:xuvbases}
%may be reordered suitably 
%to produce a unitriangular matrix, each inverse transition matrix
%has entries in $\mathbb Z[\qdiff]$.
\begin{cor}
  For 
%We will show more generally that 
for fixed $t, u \in \sn$, not necessarily related in the weak order, 
%with $t \leq u$ in the weak order, 
there are uniquely defined polynomials
$\{r_{u,v,t,w}(q_1)\mid v,w\in\sn\}$ in $\mathbb{Z}[q_1]$
which satisfy
\begin{equation}\label{eq:rdefn2}
%x^{u,v} = \sum_{w \geq u^{-1}v} r_{u,v,t,w}(\qp12 - \qm12)x^{t,w},
x^{u,v} = \sum_{w \in \sn} r_{u,v,t,w}(\qp12 - \qm12)x^{t,w}.
\end{equation}
\end{cor}
%\begin{proof}
%  The matrix of coefficients appearing in the proof of
%  Corollary~\ref{c:xuvbases} is invertible over $\mathbb{Z}[\qdiff]$.
%  \end{proof}
%We will show specifically in Theorem~\ref{t:uexpanduprime}
%that for $t \leq u$ in the weak order, 
%we have $r_{u,v,t,w}(q_1) \in \mathbb N[q_1]$
%%\mid v,w\in\sn\}$ in $\mathbb{Z}[q_1]$
%and $r_{u,v,t,tu^{-1}v}(q_1) = 1$.
%($\star$ It this correct, or is something similar correct?)

%($\star$ Mention that $x^{v,e} = x^{e,v^{-1}}$?)

Now we
turn to the problem of
%Meanwhile,
%To
%It is possible to
%combinatorially interpret
combinatorially interpreting
%the
%the nonnegative integer
coefficients of the polynomials
$\{ r_{u,v,t,w}(q_1) \,|\, u,v,t,w \in \sn \}$
%, t \wleq u \}$,\\
when $t \wleq u$.
%To do this,
To begin,
we consider
%let us turn to
a seemingly unrelated generating function for certain walks in the weak order.
%(Is this the weak order?)
\begin{defn}\label{d:pdefn}
Fix permutations $t, u, v, w \in \sn$ with $t \wleq u$,
% in the weak order,
and a reduced expression $s_{i_1} \ntnsp\cdots s_{i_k}$ 
%be any reduced expression 
for $ut^{-1}$.
Define $C_{u,v,t,w}^b(s_{i_1} \cdots s_{i_k})$ to be the set of sequences
$\pi = (\pi^{(0)}, \dotsc, \pi^{(k)})$ satisfying
\begin{enumerate}
\item $\pi^{(0)} = v$, $\pi^{(k)} = w$,
\item $\pi^{(j)} \in \{ s_{i_j} \pi^{(j-1)}, \pi^{(j-1)} \}$ for $j = 1,\dotsc, k$,
\item $\pi^{(j)} = s_{i_j}\pi^{(j-1)}$ if $s_{i_j}\pi^{(j-1)} > \pi^{(j-1)}$ for $j = 1,\dotsc, k$,
\item $\pi^{(j)} = \pi^{(j-1)}$ for exactly $b$ values of $j$ 
for $j = 1,\dotsc, k$,
\end{enumerate}
and define 
%, fix $k \leq \ell$ and define
%$t = s_{i_k} \cdots s_{i_\ell}$.  
%Define 
the polynomial
\begin{equation}\label{eq:pdefn}
%p_{u,v,t,w}(q_1) = 
%p_{u,v,t,w}(q_1; s_{i_1}\ntnsp \cdots s_{i_k}) 
p_{u,v,t,w}(q_1; \redexp 1k) 
%= \sum_b c_b q_1^b \in \mathbb N[q_1]
= \sum_b |C_{u,v,t,w}^b(\redexp 1k )| q_1^b \in \mathbb N[q_1].
\end{equation}
%by $c_b = |C_{u,v,t,w}^b(s_{i_1}\cdots s_{i_k})|$.
%declaring $c_b$ to be the number of sequences
\end{defn}
\noindent
Observe that we have $p_{t,v,t,w}(q_1;\emptyset) = \delta_{v,w}$. 
We also have the following recursive formula.

\begin{prop}\label{p:simplerecur}
Fix $t,u,v,w \in \sn$ with $t <_W u$,
% in the weak order,
and fix a reduced expression 
$s_{i_1} \cdots s_{i_k}$ for $ut^{-1}$.
%Fix $t,v \in \sn$.  
Then we have
%Then the polynomials
%$\{ p_{u,w,t,v}(q_1) \,|\, w \in \sn \}$
%$\{ p_{u,w,t,v}(q_1) \,|\, w \in \sn, u \geq_W t \}$
%are uniquely defined by any 
%the recursive formula
\begin{equation}\label{eq:simplerecur}
p_{u, v, t, w}(q_1;s_{i_1}\ntnsp \cdots s_{i_k}) = 
\begin{cases}
p_{s_{i_1}\ntnsp u, s_{i_1}\ntnsp v, t, w}(q_1; s_{i_2}\ntnsp \cdots s_{i_k}) 
&\text{if $s_{i_1}\ntnsp v > v$},\\
p_{s_{i_1}\ntnsp u, s_{i_1}\ntnsp v, t, w}(q_1; s_{i_2}\ntnsp \cdots s_{i_k}) 
+ q_1 p_{s_{i_1}\ntnsp u, v, t , w}(q_1; s_{i_2}\ntnsp \cdots s_{i_k}) 
&\text{if $s_{i_1}\ntnsp v < v$}.
\end{cases}
\end{equation}
%and the initial conditions
%\begin{equation*}
%$p_{t, w, t, v}(q_1) = \delta_{w,v}$.
%\end{equation*}
%where the polynomials on the right-hand side are defined by (\ref{eq:pdefn})
%and the fixed reduced expression $s_{i_2} \cdots s_{i_k}$ for 
%$s_{i_1}\ntnsp ut^{-1}$.
\end{prop}
\begin{proof}
The coefficient of $q_1^b$ on the left-hand side of
(\ref{eq:simplerecur}) is $|C_{u,v,t,w}^b(s_{i_1}\ntnsp \cdots s_{i_k})|$.
Since $t <_W u$ and $s_{i_1}$ is a left descent for $ut^{-1}$,
we have that $t \wleq s_{i_1}\ntnsp u$ and that $s_{i_2} \cdots s_{i_k}$ is
a reduced expression for $s_{i_1}\ntnsp ut^{-1}$.
Thus the coefficient of $q_1^b$ on the right-hand side of 
(\ref{eq:simplerecur}) is equal to the cardinality of
\begin{equation*}
D \defeq 
\begin{cases}
C_{s_{i_1}\ntnsp u,s_{i_1}\ntnsp v,t,w}^b(s_{i_2}\ntnsp \cdots s_{i_k})
&\text{if $s_{i_1}\ntnsp v > v$},\\
C_{s_{i_1}\ntnsp u,s_{i_1}\ntnsp v,t ,w}^b(s_{i_2}\ntnsp \cdots s_{i_k}) 
\cup C_{s_{i_1}\ntnsp u,v,t,w}^{b-1}(s_{i_2}\ntnsp \cdots s_{i_k})
&\text{if $s_{i_1}\ntnsp v < v$}.
\end{cases}
\end{equation*}
We claim that the 
%Consider applying 
%the injective 
map
%be the set of sequences counted by the coefficient $c_b$
%in Equation (\ref{eq:pdefn}), and consider the map
\begin{equation}\label{eq:removepi0}
\pi = (\pi^{(0)}, \pi^{(1)}, \dotsc, \pi^{(k)}) \mapsto 
(\pi^{(1)}, \dotsc, \pi^{(k)})
\end{equation}
is a bijection from $C_{u,v,t,w}^b(s_{i_1}\ntnsp \cdots s_{i_k})$ to $D$.
Clearly it is injective, since each element of 
$C_{u,v,t,w}^b(s_{i_1}\ntnsp \cdots s_{i_k})$ satisfies $\pi^{(0)} = v$.
%is a bijection from $C_{u,w,t,v}^b$ to
%$C_{s_{i_1}u,s_{i_1}w,t,v}^b \cup C_{s_{i_1}u,w,t,v}^{b-1}$.
%Observe that $C_{s_{i_1}u,s_{i_1}w,t,v}^b$ and $C_{s_{i_1}u,w,t,v}^{b-1}$
%are well defined, since
%$t \leq s_{i_1}u$ in the weak order, and
%$s_{i_2} \cdots s_{i_{k-1}}$ is a reduced expression for
%$s_{i_1}ut^{-1}$.

%If 
%$v$ satisfies 
%$s_{i_1}\ntnsp v > v$, then
To see that the map (\ref{eq:removepi0}) is well-defined and surjective,
assume first that $s_{i_1}\ntnsp v > v$.  
Then $\pi$ satisfies $\pi^{(1)} = s_{i_1}\ntnsp v$
and we have $b \leq k-1$.  
It follows that for $b = 0, \dotsc, k-1$, the sequence
$(\pi^{(1)}, \dotsc, \pi^{(k)})$ satisfies the conditions
\begin{enumerate}
\item[($1'$)] $\pi^{(1)} = s_{i_1} \ntnsp v$, $\pi^{(k)} = w$,
\item[($2'$)] $\pi^{(j)} \in \{ s_{i_j} \pi^{(j-1)}, \pi^{(j-1)} \}$ 
for $j = 2,\dotsc, k$,
\item[($3'$)] $\pi^{(j)} = s_{i_j}\pi^{(j-1)}$ if $s_{i_j}\pi^{(j-1)} > \pi^{(j-1)}$
for $j = 2,\dotsc, k$,
\item[($4'$)] $\pi^{(j)} = \pi^{(j-1)}$ for exactly $b$ values of $j$ 
for $j = 2,\dotsc, k$.
\end{enumerate}
%Since $t <_W u$ and $s_{i_1}$ is a left descent for $ut^{-1}$,
%we have that $t \wleq s_{i_1}\ntnsp u$ and that $s_{i_2} \cdots s_{i_k}$ is
%a reduced expression for $s_{i_1}\ntnsp ut^{-1}$.
%Thus the map (\ref{eq:removepi0})
%is a bijection from $C_{u,v,t,w}^b(s_{i_1}\ntnsp \cdots s_{i_k})$ 
Thus $(\pi^{(1)}, \dotsc, \pi^{(k)})$ belongs 
to $C_{s_{i_1}\ntnsp u,s_{i_1}\ntnsp v,t,w}^b(s_{i_2}\ntnsp \cdots s_{i_k})$.
Moreover, since prepending $v$ to any sequence in 
$C_{s_{i_1}\ntnsp u,s_{i_1}\ntnsp v,t,w}^b(s_{i_2}\ntnsp \cdots s_{i_k})$
produces a sequence belonging to 
$C_{u,v,t,w}^b(s_{i_1}\ntnsp \cdots s_{i_k})$, the map (\ref{eq:removepi0}) 
is surjective as well.

%$C_{s_{i_1}\ntnsp u,s_{i_1}\ntnsp v,t,w}^b(s_{i_2}\ntnsp \cdots s_{i_k})$
%produces a sequence in 
%Thus the sequence $(\pi^{(1)}, \dotsc, \pi^{(k-1)})$ belongs to
%Observe that 
%$C_{s_{i_1}u,s_{i_1}w,t,v}^b$, 
%and $C_{s_{i_1}u,w,t,v}^{b-1}$
%are well defined, since
%which is well-defined, since
%Observing that
%$t \wleq s_{i_1}\ntnsp u$ 
%in the weak order, 
%and since
%$s_{i_2} \cdots s_{i_k}$ is a reduced expression for
%$s_{i_1}\ntnsp ut^{-1}$.
%, we see that
%Moreover, our assumption $s_{i_1}w > w$ implies that

Now assume that $s_{i_1} \ntnsp v < v$.  Then $\pi$ satisfies 
$\pi^{(1)} = s_{i_1} \ntnsp v$
or $\pi^{(1)} = v$.
If $\pi^{(1)} = s_{i_1} \ntnsp v$, 
then the sequence $(\pi^{(1)}, \dotsc, \pi^{(k)})$
satisfies conditions ($1'$) -- ($4'$) above.  Otherwise it satisfies
conditions ($2'$) -- ($3'$) and
\begin{enumerate}
\item[($1''$)] $\pi^{(1)} = v$, $\pi^{(k)} = w$,
\item[($4''$)] $\pi^{(j)} = \pi^{(j-1)}$ for exactly $b-1$ values of $j$ 
for $j = 2,\dotsc, k$.
\end{enumerate}
Thus the sequence $(\pi^{(1)}, \dotsc, \pi^{(k)})$
%the map (\ref{eq:removepi0}) is a bijection from 
belongs to
%$C_{u,v,t,w}^b(s_{i_1}\ntnsp \cdots s_{i_k})$ to
\begin{equation*}
C_{s_{i_1}\ntnsp u,s_{i_1}\ntnsp v,t ,w}^b(s_{i_2}\ntnsp \cdots s_{i_k}) 
\cup C_{s_{i_1}\ntnsp u,v,t,w}^{b-1}(s_{i_2}\ntnsp \cdots s_{i_k}).
%which is well-defined by the above argument.
\end{equation*}
Moreover, since prepending $v$ to any sequence in this union 
%$C_{s_{i_1}\ntnsp u,s_{i_1}\ntnsp v,t,w}^b(s_{i_2}\ntnsp \cdots s_{i_k})$
produces a sequence belonging to 
$C_{u,v,t,w}^b(s_{i_1}\ntnsp \cdots s_{i_k})$, we again have 
surjectivity.
%We will show in Section~\ref{s:immbasis} 
%that for each fixed $u \in \sn$, the set 
%$\{ x^{u,v} \mid v \in \sn \}$ is a $\cqq$-basis for $\annnq$.
%($\star$ necessary?) Let us call $\{ x^{e,v} \mid v \in \sn \}$ the
%{\em natural} basis of $\annnq$.
\end{proof}

For fixed $t, u, w$
%with $t \wleq u$
and reduced expression $s_{i_1}\ntnsp \cdots s_{i_k}$
as in Definition~\ref{d:pdefn},
the above
initial conditions and 
%the following 
recursive formula allow one to compute
%$\{p_{u,v,t,w}(q_1; i_1, \dotsc, i_k) \,|\, v \in \sn \}$ 
$\{p_{u,v,t,w}(q_1; s_{i_1}\ntnsp \cdots s_{i_k}) \,|\, v \in \sn \}$ 
by considering the sets
\begin{equation*}
\{ p_{s_{i_k}\ntnsp t, v, t, w}(q_1; s_{i_k}) \,|\, v \in \sn \}, \quad
\{ p_{s_{i_{k-1}}\ntnsp s_{i_k}\ntnsp t, v, t, w}(q_1; s_{i_{k-1}}\ntnsp s_{i_k}) \,|\, v \in \sn \}, \dotsc,
%\{p_{u,v,t,w}(q_1; s_{i_1}, \dotsc, s_{i_k}) \,|\, v \in \sn \}. 
%\{ p_{s_{i_k}\ntnsp t, v, t, w}(q_1) \,|\, v \in \sn \}, \quad
\end{equation*}
in order.  Somewhat surprisingly, these polynomials 
{\em do not} depend upon the choice of a reduced expression for $ut^{-1}$,
although each set $C_{u,v,t,w}^b(s_{i_1} \cdots s_{i_k})$ 
does depend upon such a choice.
Also, perhaps surprisingly, these polynomials provide a combinatorial
interpretation for entries of the transition matrices relating
pairs $(\{x^{u,v} \,|\, v \in \sn \}, \{ x^{t,w} \,|\, w \in \sn \})$
of bases of the zero-weight space of $\A$.

\begin{thm}\label{t:uexpanduprime}
For  
%permutations 
%$u, t, v, w$ in $\sn$ with $t \wleq u$.
$t,u$ in $\sn$ with $t \wleq u$,
%Then 
the polynomials $\{ r_{u,v,t,w}(q_1) \,|\, v,w \in \sn \}$
defined in (\ref{eq:rdefn}) satisfy
%we have
%\begin{equation}
%\begin{aligned}
%\begin{enumerate}
%\item
$r_{u,v,t,w}(q_1) = p_{u,v,t, w}(q_1, s_{i_1} \ntnsp \cdots s_{i_k})$,
where $s_{i_1} \ntnsp \cdots s_{i_k}$ is any reduced expression for $ut^{-1}$.
%\item $r_{u,v,t,w}(q_1) = 0$ if $w \ngeq tu^{-1}v$,
%\item $r_{u,v,t,tu^{-1}v}(q_1) = 1$.
%, ($\star$ Is this false?)
%for all $w$ and 
%\end{enumerate}
%\end{aligned}
%\end{equation}
%in the weak order.  
%Then for all $w \in \sn$ we have
\end{thm}
\begin{proof}
Observe that the claimed equality holds when $t = u$,
since 
%we necessarily have $t = e$ in this case,
%and 
\begin{equation*}
  r_{u,v,u,w}(q_1) = p_{u,v,u,w}(q_1; \emptyset) = \delta_{v,w}
  \end{equation*}
%in this case.
by (\ref{eq:rdefn2}) 
and (\ref{eq:pdefn}).
%the condition $t \leq u$ in the weak order forces $t = u$ in this case.
%$t = u$,
%since 
%The claim therefore holds when $u = e$, 
%since the condition $t \leq u$ in 
%the weak order forces $t = u$ in this case.
%($\star$ Mention transposition map, bar involution?)
Now assume the equality
to hold for $u$ and $t$ differing in length by at most $k - 1$,
and consider the case that $u$ and $t$ differ in length by $k$.
%in the weak order satisfying of length at least $\ell(u) - k + 1$.
%Consider permutations $u, t$ of lengths $\ell$ and $\ell - k$, 
%respectively, and 
Let $s$ be a left descent of $ut^{-1}$, 
and therefore a left descent of $u$.
%Extend $s$ to a reduced expression
%$s s_{j_2} \cdots s_{j_k}$ for $sut^{-1}$.  

%(1)
Expanding both sides of (\ref{eq:rdefn}) in terms of the basis
$\{ x^{t,w} \,|\, w \in \sn \}$ and using induction we obtain
\begin{equation}\label{eq:randp}
r_{u,v,t,w}(q_1) = 
\begin{cases}
p_{su,sv,t,w}(q_1; s_{j_2} \cdots s_{j_k}) &\text{if $sv > v$},\\
p_{su,sv,t,w}(q_1; s_{j_2} \cdots s_{j_k}) 
+ q_1 p_{su,v,t,w}(q_1; s_{j_2} \cdots s_{j_k}) &\text{if $sv < v$},
\end{cases}
\end{equation}
where $s_{j_2} \cdots s_{j_k}$ is an arbitrary reduced expression for $sut^{-1}$.
%Letting $s$ vary over
Since $s$ is an arbitrarily chosen left descent of $ut^{-1}$, 
we have the desired result.
%By (\ref{eq:switchsides})
%and induction we have
%\begin{equation*}
%x^{u,v} 
%= \begin{cases}
%\displaystyle{\sum_{w \in \sn}}
%p_{su,sv,t,w}(\qdiff) x^{t, w} 
%&\text{if $sv > v$},\\
%\Big( \displaystyle{\sum_{w \in \sn}}  
%p_{su,sv,t,w}(\qdiff) 
%+ (\qdiff) \ntnsp \displaystyle{\sum_{w \in \sn}} 
%p_{su,v,t,w}(\qdiff) \Big) x^{t, w} 
%&\text{if $sv < v$},
%\end{cases}
%\end{equation*}
%By Proposition~\ref{p:simplerecur}, 
%the coefficient of $x^{t,w}$ in this expression is 
%$p_{u,v,t,w}(\qdiff;ss_{i_2}\ntnsp \cdots s_{i_k})$,
%and by (\ref{eq:rdefn2}) this coefficient does not depend upon the reduced
%expression $ss_{j_2}\ntnsp \cdots s_{j_k}$.
%(2) 
%By (\ref{eq:randp}) we have
%\begin{equation}\label{eq:randr}
%r_{u,v,t,w}(q_1) = 
%\begin{cases}
%r_{su,sv,t,w}(q_1) &\text{if $sv > v$},\\
%r_{su,sv,t,w}(q_1) + q_1 r_{su,v,t,w}(q_1) &\text{if $sv < v$}.
%\end{cases}
%\end{equation}
%If $r_{u,v,t,w}(q_1)$ is nonzero, then at least one of the polynomials on the
%right-hand side of (\ref{eq:randr}) must be nonzero.
%If $r_{su,sv,t,w}(q_1)$ is nonzero, then in both cases we have by induction tha%t 
%$w \geq t(su)^{-1}sv = t u^{-1}v$.
%If $sv < v$ and $r_{su,v,t,w}(q_1)$ is nonzero, then we have by induction that 
%$w \geq t(su^{-1})v = t u^{-1}sv$
%and by Proposition~\ref{p:midascent}
%that $t u^{-1}sv > t u^{-1}v$.
%(3) Substituting $w = tu^{-1}v$ in (\ref{eq:randr}) and using induction,
%we see that the polynomials on the right-hand side become
%$r_{su,sv,t,tu^{-1}v}(q_1) = r_{su,sv,t,t(su)^{-1}sv}(q_1) = 1$
%(in both cases)
%and
%$r_{su,v,t,tu^{-1}v}(q_1) = 0$,
%since 
%$t u^{-1}sv > t u^{-1}v$
%when $sv < v$.
\end{proof}
%where $t$ and $w$ are fixed.
%($\star$ Switch $v$ and $w$ notation throughout?) 
%The polynomials $\{ p_{u,v,t,w}(q_1) \,|\, u,v,t,w \in \sn \}$ 
%satisfy a simple recurrence.
%We will show in ($\star$ ?) that the assumption $t \wleq u$ implies
%that the definition (\ref{eq:pdefn}) 
%%is independent of the chosen
%does not depend upon the choice of a reduced
%expression for $ut^{-1}$.
%Fixing a pair $(t,w)$ of permutations,
%one may use the recurrence (\ref{eq:simplerecur}) and initial conditions
%\begin{equation*}
%p_{t, v, t, w}(q_1) = \delta_{v,w}
%\end{equation*}
%implied by the definition (\ref{eq:pdefn})
%to generate all such polynomials for $u \geq t$ in the weak order,
%and arbitrary $v$.
%The omission of the sequence $(s_{i_1}, \dotsc, s_{i_k})$ from our notation
%is justified by the following fact.
%Evaluating the polynomials defined in (\ref{eq:pdefn}) at $q_1 = \qdiff$,
%we obtain entries of the transition matrices relating bases of the form
%$\{ x^{u,v} \,|\, v \in \sn \}$ and $\{ x^{t, w} \,|\, w \in \sn \}$
%when $t \wleq u$.
%Specifically, the following result shows that

%\begin{cor}\label{c:unitri}
%  For any $u,v,t,w \in \sn$ with $t \wleq u$ we have 
%  \begin{equation*}
%  x^{u,w}
%= x^{t, tu^{-1}w} + 
%\ntksp\ntksp \sum_{v > tu^{-1}w} \ntksp\ntksp 
%r_{u,w,t,v}(\qdiff) x^{t,v}.
%\end{equation*} 
%\end{cor}

Using Definition~\ref{d:pdefn} and Theorem~\ref{t:uexpanduprime}
we compute some special cases of the polynomials $r_{u,v,t,w}(q_1)$.
%($\star$ Make a nice proposition or example out of the following.)
%From the definition of the polynomials 
%$\{ r_{u,v,t,w}(q_1)\,|\, u,v,t,w \in \sn \}$
%it is clear that we have
\begin{prop}
Fix $w \in \sn$ and any generator $s$.  We have
%\begin{equation}\label{eq:strategicexample}
\begin{gather}
r_{ws,w,e,s}(q_1) = r_{w,w,e,e}(q_1) = 1,\label{eq:strategicexample}\\ 
r_{w,w,e,s}(q_1) = \begin{cases}
q_1 &\text{if $ws < w$},\\
0 &\text{if $ws > w$}. 
\end{cases} \label{eq:strategicexample2}
\end{gather}
%\end{equation}
\end{prop}
\begin{proof}
  (\ref{eq:strategicexample}) follows from Proposition~\ref{p:ruvtw}.
%  Theorem~\ref{t:uexpanduprime} (3).

  To see (\ref{eq:strategicexample2}),
  consider the coefficient
  of $q_1^b$ in $r_{w,w,e,s}(q_1)$ for $b \geq 0$, and let $\ell(w) = m$.
  By Theorem~\ref{t:uexpanduprime}
  %(1)
  this is equal to the number of $m$-step walks
  \begin{equation*}
    \pi = (\pi^{(0)} \ntnsp= w,\,\ \pi^{(1)},\ \dotsc,\ \pi^{(m)} \ntnsp= s)
  \end{equation*}
  in the weak order satisfying conditions stated in Definition~\ref{d:pdefn}.
  In particular, $b$ of the indices $j \in \{1, \dotsc, m-1\}$ satisfy
  $\pi^{(j-1)} = \pi^{(j)}$, while the others satisfy $\pi^{(j-1)} > \pi^{(j)}$.
  Since $\ell(w) - \ell(s) = m-1$, the coefficient must be $0$ unless $b = 1$.
  Furthermore, the coefficient must be $0$ if $ws > w$, equivalently
  %because
  %this condition is equivalent to
  $s \not \wless w$, because in this case the shortest walk in the weak order
  from $w$ to $s$ consists of $m+1$ steps with no repetition.
%  Equivalently, the coefficient is $0$ if $ws > w$.
  %($\star$ Define weak order earlier and mention equivalence of
  %weak order - Bruhat order inequalities.)

  Suppose therefore that we have $b = 1$ and $ws < w$,
  equivalently $s \wless w$.
  Then $w$ has a reduced expression of the form
  $s_{i_1} \cdots s_{i_{m-1}} s$,
  %for $w$
  and the sequence
  \begin{equation}\label{eq:onewalk}
    (w,\,\ s_{i_1}\ntnsp w,\,\ s_{i_2}s_{i_1}\ntnsp w,\ \dotsc,\
    s_{i_{m-1}}\ntnsp\cdots s_{i_1}\ntnsp w = s,\,\ s)
  \end{equation}
  is one walk satisfying the conditions of Definition~\ref{d:pdefn}
  using the above reduced expression (which may be chosen arbitrarily
  by Theorem~\ref{t:uexpanduprime}).
  Assume that another such walk satisfies the conditions of the definition,
  using the same reduced expression.
  Then for some index $j < m-1$ this walk satisfies $\pi^{(j)} = \pi^{(j-1)}$
  and has the form
%($\star$ Explain this better, and space the following equation better.)
  \begin{equation*}
%    \begin{multline*}
    (w,\,\ s_{i_1}\ntnsp w,\ \dotsc,\,\ s_{i_j}\ntnsp\cdots s_{i_1}\ntnsp w,\,\
    s_{i_j}\ntnsp\cdots s_{i_1}\ntnsp w,\,\
    s_{i_{j+2}}s_{i_j}\ntnsp\cdots s_{i_1}\ntnsp w,\ \dotsc,\
    ss_{i_{m-1}}\ntnsp\cdots s_{i_{j+2}} s_{i_j}\ntnsp\cdots s_{i_1}\ntnsp w = s).
%    \end{multline*}
  \end{equation*}
  But the equation in the last component of this walk implies that we have
  \begin{equation*}
    w = s_{i_1}\ntnsp\cdots s_{i_j} s_{i_{j+2}}\ntnsp\cdots s_{i_{m-1}},
  \end{equation*}
  contradicting the fact that $\ell(w) = m$.
  It follows that (\ref{eq:onewalk}) is the only walk satisfying
  the conditions of Definition~\ref{d:pdefn} for the chosen reduced
  expression, and that the coefficient of $q_1^1$ is $1$
  when $ws < w$.
  %our choice of $s_{i_1} \cdots s_{i_{m-1}}s$ as a reduced expression
%  Let us therefore fix a reduced expression $\sprod im$ for $w$.
%  The coefficient in question then is equal to the number of
%  sequences
%  
%  \begin{enumerate}
%    \item $\pi^{(j)} \in \{ s_{i_j}\pi^{(j-1)},\pi^{(j-1)} \}$ for $j = 1\dotsc%,m$,
%    \item $\pi^{(j)} = s_{i_j}\pi^{(j-1)}$ if $s_{i_j}\pi^{(j-1)} > \pi^{(j-1)}%$
%      for $j = 1\dotsc,m$,
%    \item $\pi^{(j)} = \pi^{(j-1)}$ for exactly $b$ values of $j = 1\dotsc,m$.
%  \end{enumerate}
\end{proof}  
%($\star$ Is there a shorter proof of (\ref{eq:strategicexample2})
%that uses a recursive formula like
%(\ref{eq:randp}), based on {\em right} multiplication by a generator?)

%\section{Wiring diagrams and path matrices}\label{s:planar}
\section{Wiring diagrams and the $q$-immanant evaluation theorem}\label{s:planar}
%\section{Wiring diagrams, path matrices, and trace evaluation}\label{s:planar}
%\section{From wiring diagrams to trace evaluations}\label{s:planar}
%combinatorial 
%interpretations
%of 
%$\sn$-
%class functions}\label{s:sninterp}

%($\star$ State the purpose of this section.)
To evaluate
%combinatorially interpret the evaluations of
induced sign characters at elements $(1+T_{s_{i_1}}) \cdots (1 + T_{s_{i_m}})$
of $\hnq$,
%having the form (\ref{eq:wdelt}),
we will associate to each such element
%represent ($\star$ or better word?)
%the elements by
a graph $G$ called a {\em wiring diagram},
a related matrix $B$,
%to represent
%
%($\star$ something like) To each wiring diagram we associate a matrix
%$B$
and a map $\sigma_B: \Ann \rightarrow \zqq$.
A generating function $\imm{\epsilon_q^\lambda}(x) \in \Ann$ for
$\{\epsilon_q^\lambda(T_w) \,|\, w \in \sn \}$ will then allow us to compute
%satisfy
\begin{equation}\label{eq:evalwithsigma}
  \epsilon_q^\lambda((1+T_{s_{i_1}}) \cdots (1 + T_{s_{i_m}}))
  = \sigma_B(\imm{\epsilon_q^\lambda}(x))
\end{equation}
and to combinatorially interpret the resulting polynomial.
%and will yield the desired interpretation in terms of $\lambda$ and $G$.

%that allows us to
%evaluate $\hnq$ traces at $D_w$.

\ssec{Wiring diagrams and the classical immanant evaluation identity}

%we will use standard $\sn$ wiring diagrams,
%concatenations of
%associate to each such element
Call a directed planar graph $G$ a
{\em wiring diagram}
%a related matrix $B$,
%and a map $\sigma_B: \Ann \rightarrow \zqq$.
%A generating function $\imm{\epsilon_q^\lambda}(x) \in \Ann$ for
%$\{\epsilon_q^\lambda(T_w) \,|\, w \in \sn \}$ will then allow us to compute
%\begin{equation}\label{eq:evalwithsigma}
%  \epsilon_q^\lambda((1+T_{s_{i_1}}) \cdots (1 + T_{s_{i_m}}))
%  = \sigma_B(\imm{\epsilon_q^\lambda}(x))
%\end{equation}
%and to combinatorially interpret the resulting polynomial.
%%\ssec{Wiring diagrams, path matrices, and path families}
%Call a directed planar graph $G$ a
%{\em planar network of order $n$} if it
%is acyclic and 
%may be embedded in a disc with $2n$ boundary vertices labeled
%clockwise as 
%{\em source $1, \dotsc,$ source $n$} 
%(with indegrees of $0$) and 
%{\em sink $n, \dotsc,$ sink $1$}
%(with outdegrees of $0$).
%In figures, we will draw 
%sources on the left and sinks on the right,
%implicitly labeled $1, \dotsc, n$ 
%from bottom to top. Edges will be implicitly oriented from left to right.
%Let $\pi = (\pi_1,\dotsc,\pi_n)$ be a sequence of source-to-sink paths
%in a planar network $G$ of order $n$.  We call $\pi$
%a (bijective) {\em path family} if there exists a permutation
%$w = w_1 \cdots w_n \in \sn$ such that 
%$\pi_i$ is a path from source $i$ to sink $w_i$.
%In this case, we say more specifically that $\pi$ has {\em type $w$}.
%We say that the path family {\em covers} $G$ if it contains every edge
%at least once.
%We will call a planar network of order $n$ a {\em wiring diagram}
if it is a concatenation of any combination of
%the $n-1$ simple transposition
%elementary
the diagrams
%$G_{[1,2]}, G_{[2,3]}, G_{[3,4]}, \dotsc, G_{[n-1,n]}$, defined by
\begin{equation}\label{eq:elenets}
  G_{\emptyset} = 
\begin{tikzpicture}[scale=.4,baseline=-40]
  \draw[-] (0,0) -- (1,0);
  \draw[-] (0,-1) -- (1,-1);
  \draw[-] (.5,-1.6) circle (1pt); \draw[-] (.5,-1.9) circle (1pt); \draw[-] (.5,-2.2) circle (1pt); 
  \draw[-] (0,-3) -- (1,-3);
  \draw[-] (0,-4) -- (1,-4);
  \draw[-] (0,-5) -- (1,-5);
  \draw[-] (0,-6) -- (1,-6);
\end{tikzpicture}
\;, \quad
  G_{[1,2]} = 
\begin{tikzpicture}[scale=.4,baseline=-40]
  \draw[-] (0,0) -- (1,0);
  \draw[-] (0,-1) -- (1,-1);
  \draw[-] (.5,-1.6) circle (1pt); \draw[-] (.5,-1.9) circle (1pt); \draw[-] (.5,-2.2) circle (1pt); 
  \draw[-] (0,-3) -- (1,-3);
  \draw[-] (0,-4) -- (1,-4);
  \draw[-] (0,-5) -- (1,-6); 
  \draw[-] (0,-6) -- (1,-5);
\end{tikzpicture}
\;, \quad
  G_{[2,3]} = 
\begin{tikzpicture}[scale=.4,baseline=-40]
  \draw[-] (0,0) -- (1,0);
  \draw[-] (0,-1) -- (1,-1);
  \draw[-] (.5,-1.6) circle (1pt); \draw[-] (.5,-1.9) circle (1pt); \draw[-] (.5,-2.2) circle (1pt); 
  \draw[-] (0,-3) -- (1,-3);
  \draw[-] (0,-4) -- (1,-5);
  \draw[-] (0,-5) -- (1,-4);
  \draw[-] (0,-6) -- (1,-6);
\end{tikzpicture}
\;, \quad
  G_{[3,4]} = 
\begin{tikzpicture}[scale=.4,baseline=-40]
  \draw[-] (0,0) -- (1,0);
  \draw[-] (0,-1) -- (1,-1);
  \draw[-] (.5,-1.6) circle (1pt); \draw[-] (.5,-1.9) circle (1pt); \draw[-] (.5,-2.2) circle (1pt); 
  \draw[-] (0,-3) -- (1,-4);
  \draw[-] (0,-4) -- (1,-3);
  \draw[-] (0,-5) -- (1,-5);
  \draw[-] (0,-6) -- (1,-6);
\end{tikzpicture}
\;, \quad \dotsc, \quad
  G_{[n-1,n]} = 
\begin{tikzpicture}[scale=.4,baseline=-40]
  \draw[-] (0,0) -- (1,-1);
  \draw[-] (0,-1) -- (1,0);
  \draw[-] (.5,-1.6) circle (1pt); \draw[-] (.5,-1.9) circle (1pt); \draw[-] (.5,-2.2) circle (1pt); 
  \draw[-] (0,-3) -- (1,-3);
  \draw[-] (0,-4) -- (1,-4);
  \draw[-] (0,-5) -- (1,-5);
  \draw[-] (0,-6) -- (1,-6);
\end{tikzpicture}
\;,
\end{equation}
%  \raisebox{-9mm}{\includegraphics[height=19mm]{g12}}\ , \qquad
%  G_{[2,3]} =
%  \raisebox{-9mm}{\includegraphics[height=19mm]{g23}}\ , \qquad
%  G_{[3,4]} =
%  \raisebox{-9mm}{\includegraphics[height=19mm]{g34}}\ , 
%  \dotsc, \
%  G_{[n-1,n]} =
%  \raisebox{-9mm}{\includegraphics[height=19mm]{gnminusonen}}\ , \qquad
%\end{equation}
%including the empty concatenation
%\begin{equation*}
%    G_{\emptyset} \defeq \raisebox{-9mm}{\includegraphics[height=19mm]{gempty}}\ ,
representing the elements 
$e, s_1, s_2, s_3, \dotsc, s_{n-1}$ of $\sn$, respectively. 
Each wiring diagram has $n$ implicit vertices on the left and right,
labeled 
{\em source $1, \dotsc,$ source $n$} 
%(with indegrees of $0$) 
and 
{\em sink $1, \dotsc,$ sink $n$},
respectively, from bottom to top.
Edges are implicitly oriented from left to right.
Let $\pi = (\pi_1,\dotsc,\pi_n)$ be a sequence of source-to-sink paths
in a
%an $\sn$
wiring diagram $G$.  We call $\pi$
a (bijective) {\em path family} if there exists a permutation
$w = w_1 \cdots w_n \in \sn$ such that 
$\pi_i$ is a path from source $i$ to sink $w_i$.
In this case, we say more specifically that $\pi$ has {\em type $w$}.
We say that the path family {\em covers} $G$ if it contains every edge
exactly once.

The number of
path families covering
%$G$
\begin{equation}\label{eq:wdnetwork}
G = G_{[i_1,i_1+1]} \circ \cdots \circ G_{[i_m,i_m+1]}
\end{equation}
%{\em the wiring diagram of} this expression.
%It is easy to see that a path family which covers $G$
%includes every edge exactly once.  Furthermore,
is $2^m$: for $j = 1,\dotsc,m$, the two paths
intersecting at the central vertex of $G_{[i_j,i_j+1]}$ either cross or
do not cross at that vertex.  In these two cases, we call the index $j$
a {\em crossing} or {\em noncrossing} of the path family, respectively.
%Whether or not the expression $\sprod im$ is reduced, we call
We call (\ref{eq:wdnetwork}) {\em the wiring diagram of} the expression
$\sprod im$, whether or not this expression is reduced.
%Wiring diagrams appear frequently in the study of $\sn$, where
%the wiring diagram $G$ (\ref{eq:wdnetwork}) typically encodes the
%sequence $\sprod im$ of generators in $\sn$.
%Whether or not this expression is reduced, we call $G$
%{\em the wiring diagram of} the expression $\sprod im$.
It is well known that if we have the equality
$\sprod im = v$ in $\sn$,
then $v$ is the type of the unique path family covering $G$
%(\ref{eq:wdnetwork})
in which all indices $1,\dotsc,m$ are crossings.
%, the two paths containing the central
%vertex of $G_{[i_j,i_j+1]}$ {\em cross} at that vertex.

Alternatively, one may use the same diagram $G$ to encode the element
$(1 + s_{i_1}) \cdots (1 + s_{i_m})$ of $\zsn$.
The $2^m$ terms in the expansion of this product may be written
and collected as
\begin{equation*}
  \sum_{\beta \in 2^{[m]}} s_{i_1}^{\beta_1} \cdots s_{i_m}^{\beta_m}
  = \sum_{v \in \sn} d_v v,
\end{equation*}
where the $2^m$ binary words $\beta = \beta_1 \cdots \beta_m$
%on the left
%in the first sum
correspond to path families covering $G$ by
%interpreted as
\begin{equation}\label{eq:mask}
  \beta_j = \begin{cases}
    1 &\text{if $j$ is a crossing},\\
%    1 &\text{if two paths cross at the central vertex of $G_{[i_j,i_j+1]}$},\\
    0 &\text{otherwise},
  \end{cases}
\end{equation}
and where we define $s_{i_j}^0 = e$.
Thus each coefficient $d_v \in \mathbb N$ in the second sum
counts the number of path families of type $v$ which cover $G$.

Similar to the above
%$\zsn$
encoding is the use of $G$
to encode the element
$(1 + T_{s_{i_1}}) \cdots (1 + T_{s_{i_m}})$ of $\hnq$.
Expanding this product and collecting terms we have
\begin{equation}\label{eq:hnqwire}
  \sum_{\beta \in 2^{[m]}} (T_{s_{i_1}})^{\beta_1} \cdots (T_{s_{i_m}})^{\beta_m}
  = \sum_{v \in \sn} a_v T_v,
\end{equation}
where binary words $\beta$ correspond to path families as in (\ref{eq:mask}).
Now the coefficients $a_v$ in the expansion belong to $\mathbb N[q]$ and
are defined in terms of a path familiy statistic called {\em defects}.
%crossings and noncrossings of paths in path families.
Call index $j$ a {\em defect} of path family $\pi$ if
the two paths containing the central vertex of $G_{[i_j,i_j+1]}$
%cross at that
%vertex, and call $j$ a {\em noncrossing} otherwise.
%Call a noncrossing
%{\em defective} if
%%the two paths in path families, indices $j$ such that
%the two paths containing the central vertex of
%$G_{[i_j,i_j+1]}$
%do not cross at that vertex, but
have previously crossed an odd number of times.
%, and {\em proper} otherwise.
(Equivalently, a crossing or noncrossing is defective if
the path entering the common vertex on top has a lower source index.)
We will call an index $j$ a {\em proper} crossing or noncrossing if it
is not defective.
%, and is proper otherwise.)
%(We we will call a noncrossing {\em proper} if it is not defective.
%$\star$ necessary ?)
Letting
%$\cross$ and
$\dfct(\pi)$ denote the number of
%proper and crossings and defective noncrossings, respectively, in $\pi$,
%Then
defects in $\pi$ we have~\cite[Prop.\,3.5]{Deodhar90}
%Now we have
%Now the coefficients $a_v$ in the expansion belong to $\mathbb Z[q]$ and
%are defined in terms of {\em noncrossings} of paths in path families,
%indices $j$ such that the two paths containing the central vertex of
%$G_{[i_j,i_j+1]}$ do not cross at that vertex.
%Call noncrossing $j$ {\em proper} if the two relevant paths
%cross an even number of times before $G_{[i_j,i_j+1]}$,
%and {\em defective} otherwise.
%Let $\pnc(\pi)$ ($\star$ necessary?)
%and $\dfct(\pi)$ denote the numbers of proper and
%defective noncrossings, respectively, in $\pi$.
%Now we have
%\begin{equation}\label{eq:defectdef}
%  a_v = \qp{m-\ell(v)}2 \sumsb{\pi\\\type(\pi) = v} \qp{\dfct(\pi)-\pnc(\pi)}2,
%\end{equation}
%or better, since $m = \ell(v) + \dfct(\pi) + \pnc(\pi)$,
\begin{equation}\label{eq:defectdef2}
  a_v =
%  \sumsb{\pi\\ \type(\pi) = v \\ \pi \text{ covers } G} q^{\dfct(\pi)}.
  \sum_\pi q^{\dfct(\pi)},
  %\qp{\cross(\pi)-\ell(v)}2,
\end{equation}
where the sum is over path families of type $v$ which cover $G$.
%Billey and Warrington showed that if $w$ avoids
%the patterns $321$, $56781234$, $56718234$, $46781235$, $46718235$
%and if $\sprod im$ is a reduced expression for $w$, then
%the sum (\ref{eq:hnqwireR}) is equal to $\qew C'_w(q)$~\cite[Thm.\,1]{BWHex}.
%wiring diagram of this expression encodes $\qew C'_w(q)$.
%as in (\ref{eq:hnqwire}).
    
%\begin{equation*}
  %\label{eq:bwhex}
%  \qiew C'_w(q) =

%In order to use wiring diagrams for the evaluations of Hecke algebra
One can enhance a wiring diagram by
%One often enhances a planar network by
associating to each edge
%we will place {\em weights} on their edges.
%To each edge of a planar network we associate
a {\em weight} belonging to some ring $R$, and by defining
the {\em weight of a path} to be the product of its edge weights.
If $R$ is noncommutative,
%then weights are multiplied
then one multiplies weights
%we will multiply weights
in the order that the corresponding edges appear in the path.
For a {\em family} $\pi = (\pi_1,\dotsc,\pi_n)$ of $n$
paths in a planar network,
one defines
%we will define
$\wgt(\pi) = \wgt(\pi_1) \cdots \wgt(\pi_n)$. 
%We will say that such a family {\em covers} the wiring diagram $G$
%if it contains every edge of $G$ exactly once.
%We will say that $\pi$ has {\em type $w \in \sn$} if
%$\pi_i$ is a path from source $i$ to sink $w_i$ for all $i$.
The {\em (weighted) path matrix} $B = B(G) = (b_{i,j})$ of $G$ is defined by
letting $b_{i,j}$ be the sum of weights of all paths in $G$ from source $i$
to sink $j$.
Thus the product $\permmon bw$ is equal to the sum of weights of all
path families of type $w$ in $G$ (covering $G$ or not).
A result known as {\em Lindstr\"om's Lemma}~\cite{KMG}, \cite{LinVRep}
asserts that for row and column sets $I$, $J$ with $|I| = |J|$,
the minor $\det(B_{I,J})$ is equal to the sum of weights of all nonintersecting
path families from sources indexed by $I$ to sinks indexed by $J$.
It is easy to show that path matrices respect concatenation:
$B(G_1 \circ G_2) = B(G_1)B(G_2)$.
%These families may or may not cover $G$.
%It is well known that the path matrix of the concatenation of
%planar networks is equal to the product of the corresponding path matrices. 
%($\star$ State/cite Lindstr\"om's Lemma?)

Assigning weights to the edges of $G$ (\ref{eq:wdnetwork}) can aid in
the evaluation of a linear function $\theta: \zsn \rightarrow \mathbb Z$
at $(1 + s_{i_1}) \cdots (1 + s_{i_m})$ by relating this evaluation to the
generating function
\begin{equation}\label{eq:imm}
  \imm{\theta}(x) \defeq \sum_{w \in \sn} \theta(w) \permmon xw
  \in \zxn,
\end{equation}
called the {\em $\theta$-immanant} in \cite[Sec.\,3]{StanPos}.
In particular,
for $j = 1,\dotsc, m$, we assign weight $1$ to the
$n-2$ horizontal edges of $G_{[i_j,i_j+1]}$, and we assign (commuting)
indeterminate weights
$z_{i_j,j,1}$, $z_{i_j,j,2}$, $z_{i_j+1,j,1}$, $z_{i_j+1,j,2}$
to the remaining nonhorizontal edges $a, b, c, d$, respectively,
\begin{equation}\label{eq:onestar}
\begin{tikzpicture}[scale=.7,baseline=-25]
  \draw[-] (0,0) -- (1.2,-1) node[above right,midway,xshift=-1mm,yshift=-1mm] {c};
  \draw[-] (0,-2) -- (1.2,-1) node[below right,midway,xshift=-1mm,yshift=1mm] {a};
  \draw[-] (2.4,0) -- (1.2,-1) node[above left,midway,xshift=1mm,yshift=-1mm] {d};
  \draw[-] (2.4,-2) -- (1.2,-1) node[below left,midway,xshift=1mm,yshift=2mm] {b};
  \draw[-,fill] (0,0) circle (2pt);  \draw[-,fill] (0,-2) circle (2pt);
  \draw[-,fill] (2.4,0) circle (2pt);  \draw[-,fill] (2.4,-2) circle (2pt);
  \draw[-,fill] (1.2,-1) circle (2pt);
\end{tikzpicture}\;.
%\raisebox{-7mm}
%{\includegraphics[height=15mm]{starabcd.eps}}
\end{equation} 
%labeled $a,b,c,d$, respectively.
%\begin{equation}\label{eq:onestar}
%\raisebox{-8mm}
%{\includegraphics[height=20mm]{starabcd.eps}}
%\end{equation} 
%labeled $a,b,c,d$, respectively.
Thus wiring diagrams
corresponding to expressions $s_{i_1} s_{i_2} s_{i_3} = s_1s_2s_1$
and $s_{i_4}s_{i_5}s_{i_6} = s_1s_2s_1$ are weighted differently
because of the different indexing of the generators.
Let $z_G$
%$\mathbf{z}$ 
be the product of all $4m$ indeterminates $z_{i,j,k}$,
and for $f \in \mathbb{Z}[z_{1,1,1} ,\dotsc, z_{i_m,m,2}]$, let $[z_G]f$ denote
the coefficient of $z_G$ in $f$.
%The following result was essentially observed by Stembridge in
Then we have the following immanant evaluation identity
for wiring diagrams (cf.\,\cite[p.\,1081]{StemConj}).
\begin{prop}\label{p:stem}
Assign weights to the edges of $G$ (\ref{eq:wdnetwork}) as above
and let $B$ be the resulting path matrix.
Then for any linear function $\theta: \zsn \rightarrow \mathbb Z$ we have
\begin{equation}\label{eq:stem}
\theta((1 + s_{i_1}) \cdots (1 + s_{i_m})) = [z_G]\imm{\theta}(B).
\end{equation}
\end{prop}
    
To illustrate, we let $n = 3$ and consider
the element
\begin{equation*}
  (1+s_1)(1+s_2)(1+s_1) = 2 + 2s_1 + s_2 + s_1s_2 + s_2s_1 + s_1s_2s_1
\end{equation*}
and its wiring diagram
%the network
%in (\ref{eq:wdnetwork}) be
\begin{equation}\label{eq:Gbraid}
%  \begin{aligned}
  G
= G_{[i_1,i_1+1]} \circ G_{[i_2,i_2+1]} \circ G_{[i_3,i_3+1]}
  = G_{[1,2]} \circ G_{[2,3]} \circ G_{[1,2]}.
\end{equation}
Assigning weights to the edges of $G$ we have
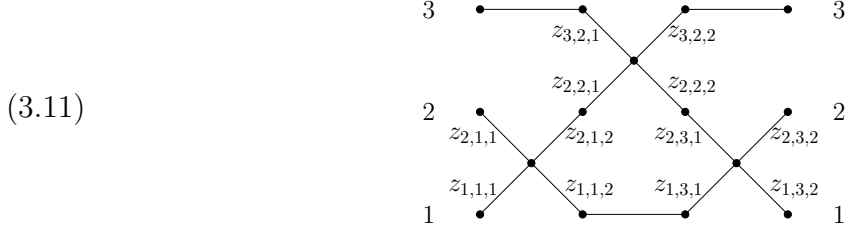
\begin{equation}\label{eq:Gbraidfig}
  \resizebox{6cm}{!}{
\raisebox{-2.2cm}{
    \begin{tikzpicture}
  \node at (-1,0) {\large $3$};
  \node at (-1,-2) {\large $2$};
  \node at (-1,-4) {\large $1$};
  \node at (7,0) {\large $3$};
  \node at (7,-2) {\large $2$};
  \node at (7,-4) {\large $1$};
  \draw[-] (0,0)  -- (2,0)
                  -- (3,-1) node[midway,left]  {\large $z_{3,2,1}$}
                  -- (4,0)  node[midway,right] {\large $z_{3,2,2}$}
                  -- (6,0);
  \draw[-] (0,-2) -- (1,-3) node[midway,left]  {\large $z_{2,1,1}$}
                  -- (2,-2) node[midway,right] {\large $z_{2,1,2}$}
                  -- (3,-1) node[midway,left]  {\large $z_{2,2,1}$}
                  -- (4,-2) node[midway,right] {\large $z_{2,2,2}$}
                  -- (5,-3) node[midway,left]  {\large $z_{2,3,1}$}
                  -- (6,-2) node[midway,right] {\large $z_{2,3,2}$};
  \draw[-] (0,-4) -- (1,-3) node[midway,left]  {\large $z_{1,1,1}$}
                  -- (2,-4) node[midway,right] {\large $z_{1,1,2}$}
                  -- (4,-4) 
                  -- (5,-3) node[midway,left]  {\large $z_{1,3,1}$}
                  -- (6,-4) node[midway,right] {\large $z_{1,3,2}$};
\filldraw (0,0) circle (.7mm);
\filldraw (2,0) circle (.7mm);
\filldraw (4,0) circle (.7mm);
\filldraw (6,0) circle (.7mm);
\filldraw (3,-1) circle (.7mm);
\filldraw (0,-2) circle (.7mm);
\filldraw (2,-2) circle (.7mm);
\filldraw (4,-2) circle (.7mm);
\filldraw (6,-2) circle (.7mm);
\filldraw (1,-3) circle (.7mm);
\filldraw (5,-3) circle (.7mm);
\filldraw (0,-4) circle (.7mm);
\filldraw (2,-4) circle (.7mm);
\filldraw (4,-4) circle (.7mm);
\filldraw (6,-4) circle (.7mm);                  
\end{tikzpicture}%.
}}
%\]
\end{equation}
and $z_G = z_{1,1,1} \cdots z_{3,2,2}$.
The weighted path matrix of $G$
is
%has four entries consisting of two terms,
%$(2,2)$-entry equal to
%\begin{equation}\label{eq:B1212}
%  \begin{gathered}
%  b_{1,1} = z_{1,1,1}z_{1,1,2}z_{1,3,1}z_{1,3,2}
%  + z_{1,1,1}z_{2,1,2}z_{2,2,1}z_{2,2,2}z_{2,3,1}z_{1,3,2}, \\
%  b_{1,2} = z_{1,1,1}z_{1,1,2}z_{1,3,1}z_{2,3,2}
%  + z_{1,1,1}z_{2,1,2}z_{2,2,1}z_{2,2,2}z_{2,3,1}z_{2,3,2}, \\
%  b_{2,1} = z_{2,1,1}z_{1,1,2}z_{1,3,1}z_{1,3,2}
%  + z_{2,1,1}z_{2,1,2}z_{2,2,1}z_{2,2,2}z_{2,3,1}z_{1,3,2}, \\
%  b_{2,2} = z_{2,1,1}z_{1,1,2}z_{1,3,1}z_{2,3,2}
%  + z_{2,1,1}z_{2,1,2}z_{2,2,1}z_{2,2,2}z_{2,3,1}z_{2,3,2},
%  \end{gathered}
%\end{equation}
\begin{equation}\label{eq:Bex}
  \begin{aligned}
    B &=
  \begin{bmatrix}
    z_{1,1,1}z_{1,1,2} & z_{1,1,1}z_{2,1,2} & 0 \\
    z_{2,1,1}z_{1,1,2} & z_{2,1,1}z_{2,1,2} & 0 \\
    0              & 0               & 1 
  \end{bmatrix}\ntksp\ntksp
  \begin{bmatrix}
    1 & 0              & 0               \\
    0 & z_{2,2,1}z_{2,2,2} & z_{2,2,1}z_{3,2,2} \\
    0 & z_{3,2,1}z_{2,2,2} & z_{3,2,1}z_{3,2,2} 
  \end{bmatrix}\ntksp\ntksp
  \begin{bmatrix}
    z_{1,3,1}z_{1,3,2} & z_{1,3,1}z_{2,3,2} & 0 \\
    z_{2,3,1}z_{1,3,2} & z_{2,3,1}z_{2,3,2} & 0 \\
    0              & 0               & 1 
  \end{bmatrix}\\
  &= \begin{bmatrix}
    z_{1,1,1}z_Uz_{1,3,2} + z_{1,1,1}z_Dz_{1,3,2}
    & z_{1,1,1}z_Uz_{2,3,2} + z_{1,1,1}z_Dz_{2,3,2}
    & z_{1,1,1}z_{2,1,2}z_{2,2,1}z_{3,2,2} \\
    z_{2,1,1}z_Uz_{1,3,2} + z_{2,1,1}z_Dz_{1,3,2}
    & z_{2,1,1}z_Uz_{2,3,2} + z_{2,1,1}z_Dz_{2,3,2}
    & z_{2,1,1}z_{2,1,2}z_{2,2,1}z_{3,2,2} \\
    z_{3,2,1}z_{2,2,2}z_{2,3,1}z_{1,3,2}
    & z_{3,2,1}z_{2,2,2}z_{2,3,1}z_{2,3,2}
    & z_{3,2,1}z_{3,2,2}
  \end{bmatrix},
  \end{aligned}
\end{equation}
where $z_U = z_{2,1,2}z_{2,2,1}z_{2,2,2}z_{2,3,1}$,
$z_D = z_{1,1,2}z_{1,3,1}$.
%We also have that $B$ is equal to the product of the weighted path matrices
%of the three factor networks,
%\begin{equation*}
%  B =
%  \begin{bmatrix}
%    z_{1,1,1}z_{1,1,2} & z_{1,1,1}z_{2,1,2} & 0 \\
%    z_{2,1,1}z_{1,1,2} & z_{2,1,1}z_{2,1,2} & 0 \\
%    0              & 0               & 1 
%  \end{bmatrix}
%  \begin{bmatrix}
%    1 & 0              & 0               \\
%    0 & z_{2,2,1}z_{2,2,2} & z_{2,2,1}z_{3,2,2} \\
%    0 & z_{3,2,1}z_{2,2,2} & z_{3,2,1}z_{3,2,2} 
%  \end{bmatrix}
%  \begin{bmatrix}
%    z_{1,3,1}z_{1,3,2} & z_{1,3,1}z_{2,3,2} & 0 \\
%    z_{2,3,1}z_{1,3,2} & z_{2,3,1}z_{2,3,2} & 0 \\
%    0              & 0               & 1 
%  \end{bmatrix}.
%\end{equation*}

Now we consider the linear function
$\theta: \mathbb Z[\mfs 3] \rightarrow \mathbb Z$
defined by
$\theta(e) = 1$, $\theta(s_1s_2s_1) = -1$, $\theta(w) = 0$ otherwise.
%\begin{equation*}
%  \theta(w) = \begin{cases}
%    1 & \text{if $w = e$}\\
%    -1 & \text{if $w = s_1s_2s_1$}\\
%    0 &\text{otherwise}
%  \end{cases}
%\end{equation*}
%This function satisfies
Computing the left-hand side of (\ref{eq:stem}) we have
\begin{equation*}
  \theta((1+s_1)(1+s_2)(1+s_1)) = 2 - 1 = 1.
  \end{equation*}
To compute the right-hand side of (\ref{eq:stem}),
we first factor the immanant as
\begin{equation*}
  \imm{\theta}(x) = x_{1,1}x_{2,2}x_{3,3} - x_{1,3}x_{2,2}x_{3,1} =
  \det(x_{13,13})x_{2,2}.
\end{equation*}
By Lindstr\"om's Lemma and inspection of
the wiring diagram (\ref{eq:Gbraidfig}), we have
%to that that the corresponding immanant
%satisfies
\begin{equation*}
  [z_G]\imm{\theta}(B) = [z_G]\det(B_{13,13})b_{2,2} = 1,
  \end{equation*}
since exactly one
family of paths $\pi = (\pi_1, \pi_2, \pi_3)$ from all sources
to the corresponding sinks satisfies
\begin{enumerate}
\item $\pi_1$ and $\pi_3$ do not intersect,
\item $\pi$ covers $G$ and therefore has weight $z_G$.
\end{enumerate}

%$q$-extension will not suffice,
%however:
%while the immanant (\ref{eq:imm}) in $\zxn$ does have
%a natural $q$-analog in $\A$, the evaluation of this element
%at a matrix is not well-defined.
%unless the entries of that matrix
%satisfy the relations (\ref{eq:anqdef}).
%Nevertheless,
%we will state and prove the desired $q$-analog in Section~\ref{s:evalthm}.

%relating this evaluation to the
%requires
%some care.
%\begin{equation*}
%\begin{aligned}
%{[z_G]}\imm{\theta}(B) &= [z_G](b_{1,1}b_{2,2}b_{3,3} - b_{1,3}b_{2,2}b_{3,1})\\
%&= [z_G](2z_Uz_Dz_{1,1,1}z_{1,3,2}z_{2,1,1}z_{2,3,2}z_{3,2,1}z_{3,2,2}\\
%&\qquad \qquad - z_{1,1,1}z_{2,1,2}z_{2,2,1}z_{3,2,2}z_{2,1,1}z_Dz_{2,3,2}z_{3,2,1}z_{2,2,2}z_{2,3,1}z_{1,3,2})\\
%&= [z_G](2z_G - z_G)\\
%&= 1.
%\end{aligned}
%\end{equation*}

%\section{The $q$-immanant evaluation theorem for wiring diagrams}
%\label{s:evalthm}
\ssec{The $q$-immanant evaluation identity for wiring diagrams}
It is natural to ask for a $q$-analog of Proposition~\ref{p:stem}
which applies to
%to aid in
the computation of
$\theta_q((1 + T_{s_{i_1}}) \cdots (1 + T_{s_{i_m}}))$
for a linear function $\theta_q: \hnq \rightarrow \zqq$,
%While the most naive approach will not suffice,
%Assigning weights to the edges of $G$ (\ref{eq:wdnetwork}) can similarly aid
%in the evaluation of
%In order to evaluate a linear function $\theta_q: \hnq \rightarrow \mathbb Z$
%at $(1 + T_{s_{i_1}}) \cdots (1 + T_{s_{i_m}})$ by relating this evaluation to the
and which uses the generating function
\begin{equation}\label{eq:qimm}
  \imm{\theta_q}(x) \defeq \sum_{w \in \sn} \theta_q(T_w) \qiew \permmon xw
  \in \Ann
\end{equation}
introduced in \cite[Eqn.\,(4.5)]{KSkanQGJ}.

We begin by assigning weights to the edges of the wiring diagram $G$
(\ref{eq:wdnetwork}) exactly as in (\ref{eq:onestar}).
But now
%For the purposes evaluating induced sign characters of $\hnq$,
%we will strategically weight edges of the wiring diagram (\ref{eq:wdnetwork})
%by $4m$ indeterminates which commute or quasicommute with one another.
%In particular,
%Let $\qp12$ be another indeterminate which commutes with all v
we define two indeterminates $z_{h,j,k}$, $z_{h',j',k'}$
to commute if $j \neq j'$ or if $k \neq k'$; otherwise we impose the relation
\begin{equation}\label{eq:zrelation}
z_{i_j+1,j,k} z_{i_j,j,k} 
= \qp12 z_{i_j,j,k} z_{i_j+1,j,k}.
\end{equation}
%for each of the $2m$ pairs $(j,k)$.
%\in [m] \times [2]$,
%where $q$ is the parameter
%of $\hnq$, which we define to commute with all indeterminates.
Let $Z_G$ be the quotient of the noncommutative ring
\begin{equation*}
    \mathbb Z[\qp12, \qm12]
    \langle z_{i_j,j,1}, z_{i_j,j,2}, z_{i_j+1,j,1}, z_{i_j+1,j,2}
    \,|\, j=1,\dotsc, m, \rangle
    \end{equation*}
modulo the ideal generated by the above commuting and quasicommuting
relations, and assume that $\qp12$, $\qm12$
commute with all other indeterminates.
Let $z_G$
%$\mathbf{z}$ 
be the product of all $4m$ indeterminates $z_{i,j,k}$, in lexicographic order.

This
small change in the indeterminates $z_{1,1,1}, \dotsc, z_{i_m,m,2}$
does not imply that the most naive $q$-analog of Proposition~\ref{p:stem}
holds, however.  Indeed, the evaluation of an element of $\A$ at a matrix
is not well defined unless the entries of that matrix satisfy the
relations (\ref{eq:anqdef}).
We therefore define the $\zqq$-linear map
%{\em Lindstr\"om's Lemma}~\cite{LinVRep},
%originally due to Karlin and McGregor~\cite{KMG},
%states that for any
%In the classical case of an unweighted planar network $G$ of order $n$
%and its path matrix $B = (b_{i,j})$ defined by
%\begin{equation*}
%  b_{i,j} = \# \text{ paths from source $i$ to sink $j$},
%\end{equation*}
%one may evaluate any element $f(x) \in \zxn$
%at the matrix $B$ by substituting $x_{i,j} = b_{i_j}$ in $f$
%for all $i,j \in [n]$.
%\ssec{A map $\sigma_B$ which interprets elements of $\Ann$ in terms of
%  path families}
%In the example of (\ref{eq:Gbraid}) - (\ref{eq:Bex}),
%products such as $b_{1,1}b_{2,2}b_{3,3}$ include terms such as
%\begin{equation*}
%  (z_{1,1,1}z_{1,1,2}z_{1,3,1}z_{1,3,2})
%  (z_{2,1,1}z_{2,1,2}z_{2,2,1}z_{2,2,2}z_{2,3,1}z_{2,3,2})
%  (z_{3,2,1}z_{3,2,2})
%\end{equation*}
%which are weights of path families which cover $G$, and terms such as
%\begin{equation*}
%  (z_{1,1,1}z_{1,1,2}z_{1,3,1}z_{1,3,2})
%  (z_{2,1,1}z_{1,1,2}z_{1,3,1}z_{2,3,2})
%  (z_{3,2,1}z_{3,2,2})
%\end{equation*}
%which are weights of path families which exclude some edges
%of $G$,
%while including others
%twice.
%We will be interested in those path families which {\em cover} a wiring diagram,
%and in the crossings of component paths within such a family.
%In order to extract only these path families from a product
%$b_{1,v_1} \cdots b_{n,v_n}$ of path matrix entries, we
%define the $\zqq$-linear map 
\begin{equation}\label{eq:sigmadef}
\begin{aligned}
%  \sigma_B: \anq_{[n],[n]} &\rightarrow \zqq\\
  \sigma_B: \annnq &\rightarrow \zqq\\
\permmon xv &\mapsto [z_G] \permmon bv,
\end{aligned}
\end{equation}
where $[z_G] \permmon bv$ denotes the coefficient of $z_G$ in $\permmon bv$,
%and this coefficient is
taken after $\permmon bv$ 
is expanded in the lexicographic basis of $Z_G$.
Note that the ``substitution'' $x_{i,j} \mapsto b_{i,j}$ is performed
only for monomials of the form $x^{e,v}$ in $\Ann$:
we define $\sigma_B(x^{u,w})$ by first expanding $x^{u,w}$
in the basis $\{ x^{e,v} \,|\, v \in \sn \}$, and {\em then} performing
the substitution.

%Let us continue with our
For example
%(\ref{eq:Gbraid}), let us and
let us compute
%\begin{example}
%Let $n=3$ and let $G$ be the star network
%Consider evaluating
$\sigma_B(x_{2,2}x_{1,1}x_{3,3})$ for the path matrix $B$ of the wiring
diagram in (\ref{eq:Bex}).
%Since the map $\sigma_B$ is defined on elements of the natural basis
%$\{ x^{e,v} \,|\, v \in \sn \}$ of $\annnq$,
%Thus we expand $x_{2,2}x_{1,1}x_{3,3}$ as
%Thus
Using (\ref{eq:anqdef}) and linearity of $\sigma_B$, we write
%express $x_{2,2}x_{1,1}x_{3,3}$ in terms of these
%we have
\begin{equation}\label{eq:x22x11x33}
  \begin{aligned}
    \sigma_B(x_{2,2}x_{1,1}x_{3,3})
    &= \sigma_B(x_{1,1}x_{2,2}x_{3,3}) + (\qdiff)\sigma_B(x_{1,2}x_{2,1}x_{3,3})\\
    &= [z_G]b_{1,1}b_{2,2}b_{3,3} + (\qdiff)[z_G]b_{1,2}b_{2,1}b_{3,3}.
    \end{aligned}
\end{equation}
Expanding $b_{1,1}b_{2,2}b_{3,3}$
%(\ref{eq:Bex})
and
%expand $b_{1,1}b_{2,2}b_{3,3}$ while
%(\ref{eq:Bex}) and
omitting terms with repeated indeterminates,
%and sorting indeterminates into lexicographic order (\ref{eq:zrelation}),
we have
%\begin{equation*}
%  z_G + z_{1,1,1} z_{2,1,2} z_{2,2,1} z_{2,2,2} z_{2,3,1} z_{1,3,2} z_{2,1,1} z_{1,1,2} z_{1,3,1} z_{2,3,2} z_{3,2,1} z_{3,2,2},
%\end{equation*}
%\begin{equation*}
%  \begin{align*}
%where the second term is a permutation of $z_G$.
%    \begin{multline*}
\begin{equation*}
  \begin{aligned}
    {[z_G]} b_{1,1}b_{2,2}b_{3,3} &=
  [z_G](z_{1,1,1} z_{1,1,2} z_{1,3,1} z_{1,3,2} z_{2,1,1} z_{2,1,2} z_{2,2,1} z_{2,2,2} z_{2,3,1} z_{2,3,2} z_{3,2,1} z_{3,2,2} \\
  &\quad+ z_{1,1,1} z_{2,1,2} z_{2,2,1} z_{2,2,2} z_{2,3,1} z_{1,3,2} z_{2,1,1} z_{1,1,2} z_{1,3,1} z_{2,3,2} z_{3,2,1} z_{3,2,2}).
% = z_G + qz_G. 
    %    \end{multline*}
  \end{aligned}
  \end{equation*}
%Observe that the first term in parentheses is $z_G$
%and the second is a permutation of $z_G$.
    %    \end{align*}
Sorting indeterminates into lexicographic order and using
(\ref{eq:zrelation}), we see that this is
\begin{equation*}
  [z_G] (z_G + q z_G) = 1+q.
  \end{equation*}
%see that this expression equals $z_G + qz_G$.
%Thus we have $\sigma_B(x_{1,1}x_{2,2}x_{3,3}) = 1+q$.
%Sorting the indeterminates in the second into lexicographical order and using
%(\ref{eq:zrelation}),
%we obtain $(1+q)z_G$.
%to sort the indeterminates 
%    and collecting terms, we have that
    %\[ \sigma_B(x_{1,1}x_{2,2}x_{3,3}) = (1+q). \]
%    \begin{equation*}
%      (1+q)
%z_{1,1,1}z_{1,1,2}z_{1,3,1}z_{1,3,2}z_{2,1,1}z_{2,1,2}z_{2,2,1}z_{2,2,2}z_{2,3,1}z_{2,3,2}z_{3,2,1}z_{3,2,2}
%    \end{equation*}
%Applying the same process to 
Similarly computing $[z_G] b_{1,2}b_{2,1}b_{3,3}$,
%from (\ref{eq:Bex}),
we obtain
%and using
%(\ref{eq:zrelation}), we obtain
%$(\qp12 + \qp32)z_G$.
%Thus we have $\sigma_B(x_{1,2}x_{2,1}x_{3,3}) = (\qp12 + \qp32)$.
$(\qp12 + \qp32)$.
%Through the same process we get that
%\[ \sigma_B(x_{1,2}x_{2,1}x_{3,3})=(\qp{1}{2}+\qp{3}{2}). \]
%Combining these two results with
Thus Equation \eqref{eq:x22x11x33} gives
%us that
\[\sigma_B(x_{2,2}x_{1,1}x_{3,3}) = (1+q) + (\qdiff)(\qp12 + \qp32) = q+q^2. \]

For special combinations of $u, v \in \sn$ and an expression
$\sprod im$,
there are simple rules for computing
%whose wiring diagram has weighted path matrix $B$,
%and two permutations $u,v$, one can compute
%we have that
$\sigma_B(x^{u,v})$.
%is $0$ or is a power of $\qp12$.
%quite easily if $u^{-1}v$ is related in the right way to the
%expression of the diagram.
\begin{prop}\label{p:sigmareduced}
  Fix $u, v \in \sn$, and an expression $\sprod i\ell$
  whose wiring diagram has weighted path matrix $B$.
%  (not necessarily reduced) for $w$.
%  Let $G$ be the corresponding wiring diagram
%  and let $B$ be the weighted path matrix of $G$.
  Then we have the following.
  \begin{enumerate}[(i)]
  %\item $\sigma_B(x^{e,v}) = 0$ unless $\sprod i\ell$ contains a subexpression for $v$.
  \item $\sigma_B(x^{u,v}) = 0$ unless $\sprod i\ell$ contains a subexpression for $u^{-1}v$.
%  \item $\sigma_B(x^{e,v}) = 0$ unless $v \leq w$,
%  \item If $\sprod i\ell$ is reduced, then
%    $\sigma_B(x^{e,w}) = \qp{\ell(w)}2$ and $\sigma_B(x^{e,v}) = 0$ unless $v \leq w$.
  \item If $\sprod i\ell$ is a reduced expression for $w \in \sn$,
    then
%    $\sigma_B(x^{e,w}) = \qp{\ell(w)}2$, and
    $\sigma_B(x^{e,w}) = \qew$, and
    $\sigma_B(x^{e,v}) = 0$ unless $v \leq w$.    
  \end{enumerate}
\end{prop}
\begin{proof}
%  ($\star$ improve this.)\\
  (i) Suppose that $b_{u_1,v_1} \cdots b_{u_n,v_n} \neq 0$.
  Then there is a path family $\pi$ of type $u^{-1}v$ which covers
  the wiring diagram $G$ of $s_{i_1}\cdots s_{i_\ell}$.
  Define the binary word $\beta = \beta_1 \cdots \beta_\ell$ by
  \begin{equation*}
    \beta_j = \begin{cases}
      1 &\text{if two paths of $\pi$ cross at the central vertex of
        $G_{[i_j,i_j+1]}$,}\\
      0 &\text{otherwise.}
    \end{cases}
  \end{equation*}
  Then we have $\type(\pi) = s_{i_1}^{\beta_1} \cdots s_{i_\ell}^{\beta_\ell}$ and
  the factors with $\beta_j = 1$ form a subexpression of
  $\sprod i\ell$ which is equal in $\sn$ to $\type(\pi)$.
  %($\star$ See e.g., some standard text about wiring diagrams, or include Bruhat order subexpression result in introduction.)

%  \noindent
  (ii) If $\sprod i\ell$ is reduced, then there is exactly
  one path family $\pi$ which covers $G$ and has type $u^{-1}v$.
  Its component paths cross at the central vertex of each
  simple transposition graph $G_{[i_j,i_j+1]}$.
  Each crossing causes the variables $z_{i_j,j,1}$, $z_{i_j+1,j,2}$
  to appear earlier than the variables $z_{i_j+1,j,1}$, $z_{i_j,j,2}$,
  contributing $\qp12$ to $\sigma_B(x^{e,w})$.
  Now consider a permutation $v$ with $v \not \leq w$. By definition
  of the Bruhat order (\ref{eq:bruweakdef})
  there is no subexpression of $s_{i_1}\cdots s_{i_\ell}$ which
  is an expression for $v$.  By (i.) we have $\sigma_B(x^{e,v}) = 0$.
%  \noindent
%  (iii.) If $u^{-1}v$ is greater than or incomparable to $w$, then so a%re all
%  permutations on the right-hand side of (\ref{eq:xuvexpand}).  Applyin%g
%  $\sigma_B$ to both sides and using (i), we have the desired result.
\end{proof}

%\begin{obs}\label{o:genshort}
%  Fix $u,v \in \sn$, let $G$ be a wiring diagram for the expression $s_i$,
%  and let $B$ be the weighted path matrix of $G$.  Then we have
%  $\sigma_B(x^{u,v}) = 0$ unless $u^{-1}v \in \{ e, s_i \}$.
%\end{obs}
%\begin{proof}
%  Omitted.
%  \end{proof}

An important special case of Proposition~\ref{p:sigmareduced}
concerns wiring diagrams for a single generator.
\begin{cor}\label{c:sigmareduced}
  Let wiring diagram $H$ of the reduced
  expression $s_{i_1} = s_j$ have weighted path matrix $C$.
  %be the weighted path matrix of $H$.
  For $u,v \in \sn$ we have
  \begin{equation*}
    \sigma_C(x^{u,v}) = 
    \begin{cases}
      \qp12 &\text{if $u = vs_j$},\\
      q &\text{if $u = v$ and $vs_j < v$},\\
      1 &\text{if $u = v$ and $vs_j > v$},\\
      0 &\text{if $u \not \in \{ v, vs_j \}$}.
    \end{cases}
  \end{equation*}
  \end{cor}
\begin{proof}
The matrix $C$ is obtained from the $n \times n$ identity matrix by
replacing its $[j,j+1]$, $[j,j+1]$ submatrix by
\begin{equation*}
\begin{bmatrix}
z_{j,1,1} z_{j,1,2} & z_{j,1,1} z_{j+1,1,2} \\
z_{j+1,1,1} z_{j,1,2} & z_{j+1,1,1} z_{j+1,1,2}
\end{bmatrix},
\end{equation*}
and we have $z_H = z_{j,1,1}z_{j,1,2}z_{j+1,1,1}z_{j+1,1,2}$.
%Let us simplify $\sigma_C(x^{u,v})$ for arbitrary $u$, $v$.
Using Proposition~\ref{p:ruvtw} to expand $x^{u,v}$
in the natural basis of $\annnq$ and recalling that
Proposition~\ref{p:sigmareduced} (ii) implies $\sigma_C(x^{e,w})$
to vanish unless
%we have
$w \in \{ e, s_j \}$, we can write
%Applying Proposition~\ref{p:ruvew} to $\sigma_C(x^{u,v})$ we have
\begin{equation*}
  \sigma_C(x^{u,v}) = \sigma_C(x^{e,u^{-1}v})
  + \sumsb{w > u^{-1}v\\w \leq s_j} r_{u,v,e,w}(\qdiff)\sigma_C(x^{e,w}).
\end{equation*}
Now observe that the restrictions on $w$ imply this expression to vanish unless
%we have
$u \in \{ v, vs_j \}$.
When $u = vs_j$ we have
\begin{equation}\label{eq:xvsjv}
  \begin{aligned}
  \sigma_C(x^{vs_j,v}) = \sigma_C(x^{e,s_j})
  &= [z_H] z_{j,1,1}z_{j+1,1,2}z_{j+1,1,1}z_{j,1,2} \\
  &= [z_H] \qp12 z_{j,1,1}z_{j,1,2}z_{j+1,1,2}z_{j+1,1,1} 
  = \qp12,
  \end{aligned}
\end{equation}
since $z_{j,1,2}$ commutes with $z_{j+1,1,1}$ and quasicommutes with $z_{j+1,1,2}$.
On the other hand, when $u = v$, we may use (\ref{eq:strategicexample})
and (\ref{eq:xvsjv}) to obtain
%simplify this, we have
\begin{equation}\label{eq:sigmagenerator}
%  \begin{gathered}
%    \begin{aligned}
%      \end{aligned}\\
    \begin{aligned}
      \sigma_C(x^{v,v}) &= \sigma_C(x^{e,e})
      + r_{v,v,e,s_j}(\qdiff)\sigma_C(x^{e,s_j}) \\
      &= [z_H] z_{j,1,1}z_{j,1,2}z_{j+1,1,1}z_{j+1,1,2}
      + \qp12 \begin{cases}
        \qdiff &\text{if $vs_j < v$}, \\
        0      &\text{otherwise}
      \end{cases}\\
      &= \begin{cases}
        q &\text{if $vs_j < v$},\\
        1 &\text{otherwise}.
      \end{cases}
    \end{aligned}
%    \begin{aligned}
%      \sigma_C(x^{u,v}) = 0 \text{ for } u \notin \{ v, vs_{i_j} \}.
%    \end{aligned}
%  \end{gathered}
\end{equation}
%Otherwise $\sigma_C(x^{u,v}) = 0$.
  \end{proof}
  
%Using the formula in Corollary~\ref{c:sigmareduced}, we can show that
The map $\sigma_B$ behaves well with respect to concatenation of
wiring diagrams.
\begin{prop}\label{p:sigmacoprod}
Let wiring diagrams $G$, $H$ of expressions
$s_{i_1} \ntnsp\cdots s_{i_k}$, $s_{i_{k+1}} \ntnsp\cdots s_{i_m}$
have weighted path matrices $B$, $C$, respectively. 
%so that $G \circ H$ corresponds to expression $\sprod im$ and 
%has weighted path matrix $BC$.
Then for all $u,w \in \sn$ we have
\begin{equation}\label{eq:sigmacoprod}
\sigma_{BC}(x^{u,w}) = \sum_{v \in \sn} \sigma_B(x^{u,v}) \sigma_C(x^{v,w}).
\end{equation}
\end{prop}
\begin{proof}
  Write $A = (a_{i,j}) = BC$ for the weighted path matrix of $G \circ H$,
  and consider
\begin{equation*} 
\begin{aligned}
  \sigma_{BC}(x^{u,w})
  &= [z_{G \circ H}]a_{u_1,w_1} \cdots a_{u_n,w_n}\\
  &= [z_{G \circ H}] \Big( \sum_{j_1=1}^n b_{u_1,j_1} c_{j_1,w_1} \Big)
  \cdots \Big( \sum_{j_n=1}^n b_{u_1,j_n} c_{j_n,w_1} \Big).
\end{aligned}
\end{equation*}
%  \prod_{t=1}^n(BC)_{u(t),w(t)}\\
%  & \label{eq:expansion} = \prod_{t=1}^n\sum_{j=1}^n B_{u(t),j}\cdot C_{j,w(t)}.
%\end{align}
In all but $n!$ of the $n^n$ resulting terms,
repeated indices among $j_1,\dotsc, j_n$ lead to
repeated indeterminates or matrix entries equal to $0$, which cause
the coefficient of $z_{G \circ H}$ to be $0$.
Thus we may consider only
%The only possible nonzero coefficients come from
the $n!$ terms in which $j_1, \dotsc, j_n$ are all distinct,
and the expression reduces to
\begin{equation}\label{eq:zgh}
  [z_{G \circ H}] \sum_{v \in \sn} b_{u_1,v_1}c_{v_1,w_1} \cdots b_{u_n,v_n}c_{v_n,w_n}.
\end{equation}
Now observe that $b_{u_i,v_i}$ commutes with $c_{v_j,w_j}$ for all $i,j$,
since our indexing of the expressions $s_{i_1} \cdots s_{i_k}$,
$s_{i_{k+1}} \cdots s_{i_m}$ guarantees all edge weights 
%\begin{equation*}
  $\{ z_{i_j, j,1}, z_{i_j,j,2}, z_{i_j+1,j,1}, z_{i_j+1,j,2} \,|\, 1 \leq j \leq k\}$
%\end{equation*}
of $G$ to
commute with all edge weights 
%\begin{equation*}
  $\{ z_{i_j, j,1}, z_{i_j,j,2}, z_{i_j+1,j,1}, z_{i_j+1,j,2} \,|\, k+1 \leq j \leq m\}$
%\end{equation*}
of $H$.
Thus (\ref{eq:zgh}) is equal to
\begin{equation*}
  \begin{aligned}
    \sum_{v \in \sn} [z_{G \circ H}] b_{u_1,v_1} \cdots b_{u_n,v_n} c_{v_1,w_1} \cdots c_{v_n,w_n} 
    &= \sum_{v \in \sn} [z_G] b_{u_1,v_1} \cdots b_{u_n,v_n} [z_H] c_{v_1,w_1} \cdots c_{v_n,w_n} \\
    &= \sum_{v \in \sn} \sigma_B(x^{u,v}) \sigma_C(x^{v,w}).
  \end{aligned}
\end{equation*}
%as desired.
%\Big( \sum_{j_1=1}^n b_{u_1,j_1} c_{j_1,w_1} \Big)
%  We can then expand Line \eqref{eq:expansion}, and each term corresponds to selecting a $j$ for each $t$, i.e. each term appears as
%\[ \prod_{t=1}^n B_{u(t),\tau(t)}C_{\tau(t),w(t)} \]
%for a unique function $\tau:[n]\rightarrow [n]$.
%Suppose that for one of these functions $\tau$, $\tau(t)=\tau(t')=j$ for $t\neq t'$.  Suppose that $i_\ell$ is the first occurrence of $j$ or $j-1$ after $i_k$.  Then both $C_{j,w(t)}$ and $C_{j,w(t')}$ begin with $z_{j,\ell,1}$.  So
%\[ \prod_{t=1}^n B_{u(t),\tau(t)}C_{\tau(t),w(t)}=0. \]
%Therefore, the terms are non-trivial if and only if $\tau=v$ is a permutation.
%(Uses Justin's thesis.)
\end{proof}

The special case of Proposition~\ref{p:sigmacoprod} in which we have
%$u = e$ and
$m = k+1$ (so that $H$ is the wiring diagram of the single
generator $s_{i_{k + 1}}$) leads to a simple formula for
%$\sigma_{BC}(x^{e,w})$
$\sigma_{BC}(x^{u,w})$
in terms of
%$w$, $s_{i_{k+1}}$ and
the matrix $B$.
\begin{cor}\label{c:sigmacoprod}
  Let wiring diagrams $G$, $H$
  %correspond to
  of expressions $s_{i_1}\cdots s_{i_k}$, $s_{i_{k+1}}$
  have weighted path matrices $B$, $C$, respectively.
%  so that $BC$ is the weighted path matrix of $G \circ H$.
  Then for all $w \in \sn$ we have
  \begin{equation*}
    \sigma_{BC}(x^{u,w}) =
    \qp12 \sigma_B(x^{u,ws_{i_{k+1}}}) +
    \begin{cases}
      q \sigma_B(x^{u,w}) &\text{if $ws_{i_{k+1}} < w$},\\
      \sigma_B(x^{u,w}) &\text{if $ws_{i_{k+1}} > w$}.
    \end{cases}  
%    \sigma_{BC}(x^{e,w}) =
%    \qp12 \sigma_B(x^{e,ws_{i_{k+1}}}) +
%    \begin{cases}
%      q \sigma_B(x^{e,w}) &\text{if $ws_{i_{k+1}} < w$},\\
%      \sigma_B(x^{e,w}) &\text{if $ws_{i_{k+1}} > w$}.
%    \end{cases}  
  \end{equation*}
\end{cor}
\begin{proof}
%  $BC$ is the weighted path matrix of $G \circ H$.
  %and
  By Corollary~\ref{c:sigmareduced}, the expansion in 
  Proposition~\ref{p:sigmacoprod} has only two nonzero terms:
\begin{equation*}
%\begin{aligned}
\sigma_{BC}(x^{u,w}) 
= \sigma_B(x^{u,ws_{i_{k+1}}}) \sigma_C(x^{ws_{i_{k+1}},w}) 
+ \sigma_B(x^{u,w}) \sigma_C(x^{w,w}).
\end{equation*}
Evaluating $\sigma_C$ as in Corollary~\ref{c:sigmareduced}
we then have the desired expression.
%&= \sigma_B(x^{u,ws_{i_{k+1}}}) \qp12 
%   + \sigma_B(x^{e,w}) \sigma_C(r_{w,w,e,e}(\qdiff)x^{e,e} + r_{w,w,e,s_{i_{k+1}}}(\qdiff)x^{e,s_{i_{k+1}}}).
%\sigma_{BC}(x^{e,w}) 
%&= \sigma_B(x^{e,ws_{i_{k+1}}}) \sigma_C(x^{ws_{i_{k+1}},w}) 
%   + \sigma_B(x^{e,w}) \sigma_C(x^{w,w})\\
%&= \sigma_B(x^{e,ws_{i_{k+1}}}) \sigma_C(r_{ws_{i_{k+1}},w,e,s_{i_{k+1}}}(\qdiff) x^{e,s_{i_{k+1}}})\\
%   &\quad + \sigma_B(x^{e,w}) \sigma_C(r_{w,w,e,e}(\qdiff)x^{e,e} + r_{w,w,e,s_{i_{k+1}}}(\qdiff)x^{e,s_{i_{k+1}}}).
%\end{aligned}
%\end{equation*}
%By (\ref{eq:strategicexample}) and Corollary~\ref{c:sigmareduced}
%this is
%\begin{equation*}
%\begin{gathered} 
%  \sigma_B(x^{e,ws_{i_{k+1}}}) \sigma_C(x^{e,s_{i_{k+1}}})
%+ \sigma_B(x^{e,w}) \sigma_C(x^{e,e})
%+ \begin{cases}
%  \ntnsp (\qdiff) \sigma_B(x^{e,w}) \sigma_C(x^{e,s_{i_{k+1}}})
%  &\ntnsp \text{if $ws_{i_{k+1}} < w$},\\
%  \ntnsp 0
%  &\ntnsp \text{if $ws_{i_{k+1}} > w$}
%\end{cases}\\
%= \sigma_B(x^{e,ws_{i_{k+1}}}) \qp12
%+ \sigma_B(x^{e,w}) + \begin{cases}
%(q-1) \sigma_B(x^{e,w}) &\text{if $ws_{i_{k+1}} < w$},\\
%0 &\text{if $ws_{i_{k+1}} > w$}.
%\end{cases}\\
%\end{gathered}
%\end{equation*}
\end{proof}
  
Another important property of the map $\sigma_B$ is that its
%The
evaluation
%of $\sigma_B$
at natural basis elements of $\Ann$ is closely related to
coefficients in the natural expansion of 
$(1 + T_{s_{i_1}}) \cdots (1 + T_{s_{i_m}})$ in $\hnq$.  
\begin{prop}\label{p:qewcoeff}
Let $G$ be the wiring diagram in (\ref{eq:wdnetwork}) 
with weighted path matrix $B$, and fix $w \in \sn$.
%and let $B$
%be its weighted path matrix.  
Then 
%for each $w \in \sn$,
$\sigma_B(x^{e,w})$ is equal to $\qew$ times the coefficient of $T_w$ in 
%the product
$(1 + T_{s_{i_1}}) \cdots (1 + T_{s_{i_m}})$.
%($\star$ Is there an analogous statement for $\sigma_B(x^{u,w})$?)
\end{prop}
\begin{proof}
%  ($\star$ Shorten this, possibly breaking it into lemmas.)
  Consider the wiring diagram $G = G_{[j,j+1]}$ of the simple transposition $s_j$
  %representing the $\hnq$ element $1 + T_{s_j}$
  and its weighted path matrix $B$.
  By Corollary~\ref{c:sigmareduced} we have
  %$\sigma_B(x^{e,w}) = 0$ for
  %$w \not \in \{e, s_i \}$ and $\sigma_B(x^{e,s_i}) = \qp12$.
  %Since $B$
  %%  The weighted path matrix $B$ of $G$
  %is the $n \times n$ identity matrix,
  %except that the $[i,i+1],[i,i+1]$ submatrix is replaced by
%\begin{equation*}
%\begin{bmatrix}
%z_{i,1,1} z_{i,1,2} & z_{i,1,1} z_{i+1,2,2} \\
%z_{i+1,2,1} z_{i,1,2} & z_{i+1,2,1} z_{i+1,2,2}
%\end{bmatrix},
%\end{equation*}
%we have
%%If $w \notin \{e,s_i\}$, then we have $\sigma_B(x^{e,w}) = 0$.
%%If $w = e$, then we have 
%\begin{equation*}
%\sigma_B(x^{e,e}) = 
%[z_{i,1,1}z_{i,1,2}z_{i+1,2,1}z_{i+1,2,2}] z_{i,1,1}z_{i,1,2}z_{i+1,2,1}z_{i+1,2,2}
%=1.
%\end{equation*}
%if $w = s_i$, then we have
%\begin{equation*}
%\begin{aligned}
%\sigma_B(x^{e,s_i}) 
%&= [z_{i,1,1}z_{i,1,2}z_{i+1,2,1}z_{i+1,2,2}] z_{i+1,1,1}z_{i,2,2}z_{i,2,1}z_{i+1,2,2}\\
%&= [z_{i,1,1}z_{i,1,2}z_{i+1,2,1}z_{i+1,2,2}] \qp12 z_{i,1,1}z_{i,1,2}z%_{i+1,2,1}z_{i+1,2,2}\\
%&= \qp12.
%\end{aligned}
%\end{equation*}
  \begin{equation*}
    \sigma_B(x^{e,w}) = \begin{cases}
      \qp12 %(= q_{e,s_j} \cdot 1)
      &\text{if $w = s_j$},\\
      1 %(= q_{e,e} \cdot 1)
      &\text{if $w = e$},\\
      0 %(= \qew \cdot 0)
      &\text{otherwise}.
    \end{cases}
    \end{equation*}
  Thus the result is true for any simple transposition 
  %elementary
  wiring diagram
  in (\ref{eq:elenets}).
  %of a single generator of $\sn$.
  %which is a single elementary.
%($\star$ Shorten this by referring to Proposition~\ref{p:sigmareduced} or an example?)

  Now assume that the result holds for concatenations 
  of $1,\dotsc, m-1$ simple transposition diagrams,
  and define $\{ a_w \,|\, w \in \sn \} \subset \mathbb Z[q]$
  by
  %to be the coefficients in
  \begin{equation*}
    (1 + T_{s_{i_1}}) \cdots (1 + T_{s_{i_{m-1}}}) = \sum_{w \in \sn} a_w T_w.
    \end{equation*}
Consider the wiring diagram
$G = G_{[i_1,i_1+1]} \circ \cdots \circ G_{[i_m,i_m+1]}$
%representing the $\hnq$ element $(1+T_{s_{i_1}}) \cdots (1 + T_{s_{i_m}})$.
%of $m$ elementary diagrams.
%, indexed as in (\ref{eq:wdnetwork}).
and decompose $G$ as $G' \circ H$, where
%, we will let $B$ be the weighted path matrix of
$G' = G_{[i_1,i_1+1]} \circ \cdots \circ G_{[i_{m-1},i_{m-1}+1]}$
has weighted path matrix $B'$ and
%represents
%the $\hnq$ element
%\begin{equation*}
%\sum_{v \in \sn} a_v T_v \defeq (1+T_{s_{i_1}}) \cdots (1 + T_{s_{i_{m-1}}}),
%\end{equation*}%while
$H = G_{[i_m,i_m+1]}$ has 
weighted path matrix $C$.
%and represents the $\hnq$ element $1 + T_{s_{i_m}}$, and $B = B'C$.
By Corollary~\ref{c:sigmacoprod}
we have
\begin{equation*}
\begin{aligned} 
\sigma_B(x^{e,w}) = \sigma_{B'C}(x^{e,w}) 
%&= \sigma_{B'}(x^{e,ws_{i_m}}) \sigma_C(x^{e,s_{i_m}})
%+ \sigma_{B'}(x^{e,w}) \sigma_C(x^{e,e})\\
%&\quad + \begin{cases}
%(\qdiff) \sigma_B(x^{e,w}) \sigma_C(x^{e,s_{i_m}}) &\text{if $ws_{i_m} < w$},\\
%0 &\text{if $ws_{i_m} > w$}.
%\end{cases}\\
&= \qp12 \sigma_{B'}(x^{e,ws_{i_m}})
%+ \sigma_B(x^{e,w})
+ \begin{cases}
q \sigma_{B'}(x^{e,w}) &\text{if $ws_{i_m} < w$},\\
\sigma_{B'}(x^{e,w}) &\text{if $ws_{i_m} > w$}.
\end{cases}\\
\end{aligned}
\end{equation*}
By induction, this is
\begin{equation*}
%  \begin{aligned}
    \qp12 q_{ws_{i_m}} a_{ws_{i_m}} 
    %+ \qew a_w
    + \begin{cases}
q \qew a_w &\text{if $ws_{i_m} < w$},\\
\qew a_w &\text{if $ws_{i_m} > w$}
\end{cases}
= 
\begin{cases}
\qew (a_{ws_{i_m}} + q a_w) &\text{if $ws_{i_m} < w$},\\
\qew (q a_{ws_{i_m}} + a_w) &\text{if $ws_{i_m} > w$}.
\end{cases}
%&= \sigma_B(x^{e,ws_{i_m}}) \sigma_C(x^{e,s_{i_m}})
%+ \sigma_B(x^{e,w}) \sigma_C(x^{e,e} + (\qdiff)x^{e,s_{i_m}}).
%\end{aligned}
\end{equation*}
%($\star$ Just refer to Corollary~\ref{c:sigmacoprod} here.)
On the other hand,
consider the element
%coefficient of $T_w$ in
%$\qew$ times the coefficient of $T_w$ in 
\begin{equation}\label{eq:element}
  (1 + T_{s_{i_1}}) \cdots (1 + T_{s_{i_m}})
  = \Big( \sum_{v \in \sn} a_v T_v \Big) (1 + T_{s_{i_m}}).  
\end{equation}
By (\ref{eq:hnqdef}) we have
\begin{equation*}
  T_w T_{s_{i_m}} = \begin{cases}
    (q-1)T_w + qT_{ws_{i_m}} &\text{if $ws_{i_m} < w$},\\
    T_{ws_{i_m}} &\text{if $ws_{i_m} > w$}.
  \end{cases}
  \end{equation*}
Thus $\qew$ times the coefficient of $T_w$ in (\ref{eq:element}) is  
\begin{equation*}
%\begin{aligned}
\qew \Bigg( \ntnsp a_w + \begin{cases} 
a_{ws_{i_m}} + (q-1)a_w &\text{if $ws_{i_m} < w$,}\\
qa_{ws_{i_m}} &\text{if $ws_{i_m} > w$}
\end{cases} \ntksp \Bigg)
= \begin{cases} 
\qew(a_{ws_{i_m}} + q a_w) &\text{if $ws_{i_m} < w$},\\
\qew(qa_{ws_{i_m}} + a_w) &\text{if $ws_{i_m} > w$}.
\end{cases}
\end{equation*}
\end{proof}

%($\star$ Include $T_sT_w$ formulas for $\hnq$ in an earlier section.
%Refer to them above?)
As a consequence of Proposition~\ref{p:qewcoeff},
we have a $q$-analog of Proposition~\ref{p:stem}.
The evaluation
%a $q$-analog of the evaluation
\begin{equation*}
  B \mapsto \imm{\theta}(B)
\end{equation*}
of $\imm{\theta}(x) \in \zxn$ at the $n \times n$ matrix $B$ is now replaced by
the map
\begin{equation*}
  B \mapsto \sigma_B(\imm{\theta_q}(x))
\end{equation*}
%Namely, when $B$ is the path matrix
%of a wiring diagram, and $\theta_q: \hnq \rightarrow \zqq$ is linear,
%we may ``evaluate'' $\imm{\theta_q}(x)$ at $B$ by 
%applying the map $\sigma_B$ to
for $\imm{\theta_q}(x) \in \A$.
%This provides
%an element of $\Ann$
%to as
%Now we may state

%The function $\sigma_B$ can be quite useful 
%in the study of
%linear functions $\theta_q:\hnq \rightarrow \zqq$.
%This %when one is interested in
%value of working with the algebra $\A$ and the 
%$\hnq$
%and path matrices of wiring diagrams
%While sometimes we lack an explicit formula for
%evaluations $\theta_q(g)$, $g \in \hnq$,
%we may have a nice generating function in $\Ann$ for $\theta_q$,
%\begin{equation}\label{eq:immdef}
%  \imm{\theta_q}(x) \defeq \sum_{v \in \sn} \qiev \theta_q(T_v) x^{e,v},
%\end{equation}
%which allows us to perform the desired evaluations.  If so, then we can
%apply the following result.
%Specifically we have the following.
%Let us illustrate this first for $q=1$,
%i.e., for the rings $\mathbb Z[x_{1,1}, \dotsc, x_{n,n}]$ and $\zsn$.
%Weight the wiring diagram of an expression $\sprod im$
%by $4m$ {\em commuting} indeterminates as in (\ref{eq:onestar})
%and letting $z_G$ be the product of these.
%For a linear function $\theta: \zsn \rightarrow \mathbb Z$, define
%the generating function
%\begin{equation*}
%  \imm{\theta}(x) = \sum_{v \in \sn} \theta(w) x^{e,w}
%\end{equation*}
%in $\mathbb Z[x_{1,1}, \dotsc, z_{n,n}]$.
%It is straightforward to show that
%\begin{equation*}
%  \theta((1 + s_{i_1}) \cdots (1 + s_{i_m})) = [z_G]\imm{\theta}(B).
%\end{equation*}
%More generally, we have the following.
\begin{thm}\label{t:immeval}
  Let $\theta_q: \hnq \rightarrow \zqq$ be linear, and let wiring diagram
  $G$ of $\sprod im$ have weighted path matrix $B$.
Then we have
\begin{equation}\label{eq:immeval}
\theta_q((1 + T_{s_{i_1}}) \cdots (1 + T_{s_{i_m}})) = 
\sigma_B(\imm{\theta_q}(x)).
\end{equation}
\end{thm}
%($\star$ Define $\imm{\theta_q}(x)$ here or earlier?)
\begin{proof}
%Define $\{ a_v \,|\, v \in \sn \}$ in $\mathbb Z[q]$ by
%\begin{equation*}
Write $(1 + T_{s_{i_1}}) \cdots (1 + T_{s_{i_m}}) = \sum_{v \in \sn} a_v T_v$.
%  = \theta_q((1 + T_{s_{i_1}}) \cdots (1 + T_{s_{i_m}})).
%\end{equation*}
Then the right-hand side of
%Equation~
(\ref{eq:immeval}) is
\begin{equation*}
%\begin{aligned}
\sigma_B \Big( \sum_{v \in \sn} \theta_q(T_v) \qiev x^{e,v} \Big) 
= \sum_{v \in \sn} \theta_q(T_v) \qiev \sigma_B(x^{e,v})
= \sum_{v \in \sn} \theta_q(T_v) \qiev \qev a_v
%&= \sum_{v \in \sn} \theta_q(a_vT_v)\\
= \theta_q \Big( \sum_{v \in \sn} a_vT_v \Big),
%\end{aligned}
\end{equation*}
where the second equality follows from Proposition~\ref{p:qewcoeff}.
But this is precisely the left-hand side of
%Equation~
(\ref{eq:immeval}).
\end{proof}

Now observe that if one fixes a reduced expression $\sprod im$ for
each $w \in \sn$ and uses each such expression to define an element
\begin{equation*}
  D_w \defeq (1 + T_{s_{i_1}}) \cdots (1 + T_{s_{i_m}}) \in \hnq,
\end{equation*}
%in $\hnq$,
then the set $\{ D_w \,|\, w \in \sn \}$ forms a basis of
$\hnq$: we have
%\begin{equation*}
$D_w \in T_w + \spn_{\mathbb Z[q]} \{ T_v \,|\, v < w \}$.
%\end{equation*}
(See also \cite[Cor.\,3.6]{Deodhar90}.)  Thus we can evaluate
$\theta_q(g)$ for every $g \in \hnq$, provided that we can expand
$g$ in this basis.

\section{$G$-tableaux and the combinatorics of the evaluation map}\label{s:tabx}
Theorem~\ref{t:immeval} provides half of the solution
to the problem of evaluating
%$\epsilon_q^\lambda((1 + T_{s_{i_1}}) \cdots (1 + T_{s_{i_m}}))$.
$\epsilon_q^\lambda(D_w)$.
The other half is a combinatorial interpretation of the
right-hand-side of (\ref{eq:immeval}), which is a linear
combination of expressions of the form $\sigma_B(x^{u,w}) \in \zqq$.
To combinatorially interpret such evaluations,
%Given a path family $\pi$ covering a wiring diagram $G$,
we will arrange the paths of a path family $\pi$ covering a wiring diagram $G$
into a (French) Young diagram.
We will call the resulting structure a {\em $G$-tableau},
or more specifically a {\em $\pi$-tableau}.
If $\type(\pi) = w$, we will say also that the tableau has type $w$.
For example,
the following path family $\pi$
%of type $213$
covering the wiring diagram of (\ref{eq:wdnetwork}) yields
six $\pi$-tableaux of shape $21$ and type $213$:
\begin{equation}\label{eq:1stGtableaux}
%\begin{tikzpicture}[scale=.5,baseline=-15]
\begin{tikzpicture}[scale=.6,baseline=-15]
  \draw[dotted, thick] (0,0) -- (1,0) -- (1.5,-.5) -- (2,0) -- (3,0);
  \draw[dashed] (0,-1) -- (.5,-1.5) -- (1.5,-.5) -- (2.5,-1.5) -- (3,-2);
  \draw[-,thick] (0,-2) -- (.5,-1.5) -- (1,-2) -- (2,-2) -- (2.5,-1.5) -- (3,-1);
  \node at (-.5,0) {$\pi_3$};
  \node at (-.5,-1) {$\pi_2$};
  \node at (-.5,-2) {$\pi_1$};
\end{tikzpicture} 
\;, \qquad
\tableau[scY]{\pi_3 | \pi_1,\pi_2} \;, \quad
\tableau[scY]{\pi_2 | \pi_1,\pi_3} \;, \quad
\tableau[scY]{\pi_3 | \pi_2,\pi_1} \;, \quad
\tableau[scY]{\pi_1 | \pi_2,\pi_3} \;, \quad
\tableau[scY]{\pi_2 | \pi_3,\pi_1} \;, \quad
\tableau[scY]{\pi_1 | \pi_3,\pi_2} \;.
\end{equation}
Given a $\pi$-tableau $U$, we define (integer) Young tableaux $L(U)$, $R(U)$
by replacing each path by its source index and sink index, respectively.
For example, if $U$ is the first $\pi$-tableau in (\ref{eq:1stGtableaux}),
then we have
\begin{equation*}
  L(U) = \tableau[scY]{3 | 1,2} \;, \quad
  R(U) = \tableau[scY]{3 | 2,1} \;.
\end{equation*}
It is easy to see that given two Young tableaux $P$, $Q$ of the same shape,
there is at most one $\pi$-tableau $U$ satisfying $L(U) = P$, $R(U) = Q$.

We will also define several statistics on $G$-tableaux.
Let $U$ be a $\pi$-tableau of any shape $\lambda \vdash n$.
Define $\incross(U)$, the number of {\em inverted noncrossings} of $U$,
to be the number of noncrossings $j$ of $\pi$ such that $\pi_a$, $\pi_b$
intersect at the central vertex of $G_{[i_j,i_j+1]}$ (\ref{eq:wdnetwork})
with $\pi_b$ above $\pi_a$,
%and $\pi_b$ appears in an earlier column of $U$,
%occurrences of
\begin{equation}\label{eq:inc}
\begin{tikzpicture}[scale=.7,baseline=-25]
  \draw[dashed] (0,0) -- (1,-1) -- (2,0);
  \draw[-] (0,-2) -- (1,-1) -- (2,-2);
  \node at (-.5,0) {$\pi_b$};
  \node at (-.5,-2) {$\pi_a$};
\end{tikzpicture}\;,
% \raisebox{-7mm}
 %{\includegraphics[height=15mm]{incross.eps}}\ ,
 \end{equation}
% \begin{equation}\label{eq:inc}
% \raisebox{-10mm}
% {\includegraphics[height=20mm]{incross.eps}}\ ,
% \end{equation}
and $\pi_b$ appears in an earlier column of $U$ than $\pi_a$
(whether or not $b > a$).  Thus inverted noncrossings may be proper or
defective.   
Define $\cross(U) = \cross(\pi)$ to be the number of crossings of $\pi$, i.e.,
the number of occurrences of
\begin{equation}\label{eq:cross}
\begin{tikzpicture}[scale=.7,baseline=-25]
  \draw[dashed] (0,0) -- (1,-1) -- (2,-2);
  \draw[-] (0,-2) -- (1,-1) -- (2,0);
  \node at (-.5,0) {$\pi_b$};
  \node at (-.5,-2) {$\pi_a$};
\end{tikzpicture}
\qquad
\mathrm{ or }
\qquad 
\begin{tikzpicture}[scale=.7,baseline=-25]
  \draw[-] (0,0) -- (1,-1) -- (2,-2);
  \draw[dashed] (0,-2) -- (1,-1) -- (2,0);
  \node at (-.5,0) {$\pi_a$};
  \node at (-.5,-2) {$\pi_b$};
\end{tikzpicture} \;.
\end{equation}
%\begin{equation*}
%\raisebox{-10mm}
%{\includegraphics[height=20mm]{cr1.eps}}\ 
%\qquad
%\mathrm{ or }
%\qquad 
%\raisebox{-10mm}
%{\includegraphics[height=20mm]{cr2c.eps}}\ .
%\end{equation*}
This depends only upon $\pi$; not upon the 
locations of $\pi_a$ and $\pi_b$ in $U$.
%($\star$ postpone $\cdncross$ definition until later? It is only used as a
%stepping stone.)
For example, each tableau $U$ in (\ref{eq:1stGtableaux})
satisfies $\cross(U) = 1$ because $\cross(\pi) = 1$.
The inverted noncrossings in these tableaux are appearances of $\pi_3$
in an earlier column than $\pi_2$, or $\pi_2$ in an earlier column than $\pi_1$.
The numbers of these for the six tableaux are $1, 0, 0, 0, 1, 1$, respectively.
%Since both cannot happen simultaneously in a tableau having only two columns,
%the

Combining the above tableau statistics, we have a combinatorial interpretation
of $\sigma_B(x^{u,w})$.
\begin{prop}\label{p:sigmastats}
  Let wiring diagram $G$ have weighted path matrix $B$.
For $u,w \in \sn$ we have 
\begin{equation}\label{eq:stats}
\sigma_B (x^{u,w}) = 
\sum_\pi \qp{\cross(\pi)}2
%\sum_U
q^{\incross(U)},
\end{equation}
where the sum is over path families $\pi$ of type $u^{-1}w$ covering $G$,
and $U = U(\pi,u,w)$ is the unique $\pi$-tableaux of shape $(n)$
satisfying $L(U) = u_1 \cdots u_n$, $R(U) = w_1 \cdots w_n$.
\end{prop}
\begin{proof}
  Let $G = G_{[i_1,i_1+1]} \circ \cdots \circ G_{[i_m,i_m+1]}$.
%  By Proposition~\ref{p:sigmareduced} (? uses nonreduced expression),
  If $\sprod im$ contains no reduced subexpression for $u^{-1}w$,
  then there is no path family of type $u^{-1}w$ which covers $G$,
  and the right-hand side of (\ref{eq:stats}) is $0$.
  By Proposition~\ref{p:sigmareduced}, the left-hand side is $0$ as well.
  Now suppose $m = \ell(u^{-1}w)$ and
  let $\sprod im$ be a reduced expression for $u^{-1}w$.
  Then there is exactly one path family of type $u^{-1}w$ that covers $G$.
  It has $\ell(u^{-1}w)$ crossings and no noncrossings.
  Thus the right-hand side of (\ref{eq:stats}) is
%$q_{e,u^{-1}w}$.
  $q_{u^{-1}w}$.
%  $\qp{\ell(u^{-1}w)}2$.
%  $\qp m2$.
  By Proposition~\ref{p:ruvtw}
%  Corollary~\ref{c:unitri}
  and Proposition~\ref{p:sigmareduced},
  the left-hand side is the same.
  %These expressions match those given by Proposition~\ref{p:sigmareduced}
  %for the left-hand side of (\ref{eq:stats}).
  Thus the claim is true for wiring diagrams $G$ which are concatenations of
  $0,\dotsc,\ell(u^{-1}w)$ simple transposition diagrams.
    
  Now suppose the claim is true for $G$ a concatenation of
  $m \geq \ell(u^{-1}w)$ simple transposition diagrams
  and consider $G = G' \circ H$ where
  \begin{equation*}
%    \begin{aligned}
    %   G &=
    G' = G_{[i_1,i_1+1]} \circ \cdots \circ G_{[i_m,i_m+1]},
    \qquad H = G_{[i_{m+1},i_{m+1}+1]}.
%    &= G' \circ H.
%    \end{aligned}
    \end{equation*}
    Let $B'$, $C$ be the path matrices of $G'$, $H$, respectively
    %the concatenation of the first $m$ factors and
    %let $C$ be the path matrix of the last factor
    so that $B = B'C$ is the path matrix of $G$.
    Then by Corollary~\ref{c:sigmacoprod} we have
  \begin{equation}\label{eq:sigmatwosum}
    \sigma_B(x^{u,v})
    =
    \qp12 \sigma_{B'}(x^{u,ws_{i_{m+1}}}) + \begin{cases}
      q \sigma_{B'}(x^{u,w}) &\text{if $ws_{i_{m+1}} < w$},\\
      \sigma_{B'}(x^{u,w}) &\text{if $ws_{i_{m+1}} > w$}.
      \end{cases}
  \end{equation}
  By induction we may interpret $\sigma_{B'}(x^{u,ws_{i_{m+1}}})$
  and $\sigma_{B'}(x^{u,w})$, respectively, as
  \begin{equation}\label{eq:sigmainterp}
    \sum_{\pi^{B'}} \qp{\cross(\pi^{B'})}2 q^{\incross(U(\pi^{B'}\ntksp,u,ws_{i_{m+1}}))},\qquad
    \sum_{\pi^{B'}} \qp{\cross(\pi^{B'})}2 q^{\incross(U(\pi^{B'}\ntksp,u,w))},
  \end{equation}
  where the sums are over path families of type $u^{-1}ws_{i_{m+1}}$ and $u^{-1}w$,
  respectively, which cover the wiring diagram $G'$.
  %$G_{[i_1,i_1+1]} \circ \cdots \circ G_{[i_m,i_m+1]}$.
%  $G(s_{i_1},\dotsc,s_{i_m})$.
  Substituting these two expressions into
  %the two terms on
  the right-hand side of (\ref{eq:sigmatwosum}), we obtain
  \begin{equation}\label{eq:sigmalong}
    \sum_{\pi^{B'}} \qp{\cross(\pi^{B'})+1}2 q^{\incross(U(\pi^{B'}\ntksp,u,ws_{i_{m+1}}))} +
    \begin{cases}
      \displaystyle{\sum_{\pi^{B'}}} \qp{\cross(\pi^{B'})+2}2 q^{\incross(U(\pi^{B'}\ntksp,u,w))} &\text{if $ws_{i_{m+1}} < w$,}\\
      \displaystyle{\sum_{\pi^{B'}}} \qp{\cross(\pi^{B'})}2 q^{\incross(U(\pi^{B'}\ntksp,u,w))} &\text{if $ws_{i_{m+1}} > w$,}\\
    \end{cases}
%  \end{gathered}
  \end{equation}
where the sums are as in (\ref{eq:sigmainterp}).
  
Concatenating a path family $\pi^{B'}$
%of type $u^{-1}vs_{i_{m+1}}$ or $u^{-1}v$, respectively,
which covers $G'$
%$G(s_{i_1},\dotsc,s_{i_m})$
to a path family $\pi^C$
%of type $s_{i_{m+1}}$ or $e$, respectively,
which covers $H$
%$G(s_{i_{m+1}})$,
we obtain a new path family $\pi^{{B'}C}$
%of type $u^{-1}v$
which covers $G$
%$G(s_{i_1},\dotsc,s_{i_{m+1}})$
and satisfies
$\type(\pi^{{B'}C}) = \type(\pi^{B'}) \type(\pi^C)$.
Conversely, every path family which covers $G$
%$G(s_{i_1},\dotsc,s_{i_{m+1}})$
decomposes this way.
If
%the two original path families have types
$\type(\pi^{B'}) = u^{-1}ws_{i_{m+1}}$ and
$\type(\pi^C) = s_{i_{m+1}}$ then we have
%$\type(\pi^{{B'}C}) = u^{-1}v$ then we have
 \begin{equation}\label{eq:sigmasum}
       \incross(U(\pi^{{B'}C}\ntnsp,u,w)) = \incross(U(\pi^{B'}\ntnsp,u,ws_{i_{m+1}})), \qquad
       \cross(\pi^{{B'}C}) = \cross(\pi^{B'}) + 1.
 \end{equation}
 Otherwise, if $\type(\pi^{B'}) = u^{-1}w$ and $\type(\pi^C) = e$,
 then let $j$ and $k$ be the source indices of
  the paths in $\pi^{{B'}}$ which terminate at sinks $i_{m+1}$ and $i_{m+1}+1$,
  respectively. Then we have
  %the path family $\pi^{{B'}C}$ satisfies
  \begin{equation}\label{eq:precedence}
    \begin{gathered}
    \incross(U(\pi^{{B'}C}\ntnsp,u,w)) = \incross(U(\pi^{B'}\ntnsp,u,w)) +
    \begin{cases}
%      1 &\text{if $k$ precedes $j$ in $u_1 \cdots u_n$,
      1 &\text{if $k$ precedes $j$ in $u$,
%      i.e., $(u^{-1})_k < (u^{-1})_j$},\\
     i.e., $u^{-1}_k < u^{-1}_j$},\\
      0 &\text{otherwise},
    \end{cases}\\
    \cross(\pi^{{B'}C}) = \cross(\pi^{B'}). 
%    \cross(U(\pi^{{B'}C},u,v)) = \cross(U(\pi^{B'},u,v)) + 0
%    = \cross(U(\pi^{B'},u,v)) + \cross(U(\pi^C,v,v)).
    \end{gathered}
\end{equation}
  Since $\type(\pi^{B'}) = u^{-1}w$, the index $k$ is given by
  %the source index of the path terminating at
  $k = (u^{-1}w)^{-1}_{i_{m+1}+1} = (w^{-1}u)_{i_{m+1}+1}$.
  Thus the first condition in (\ref{eq:precedence}) is equivalent to
  \begin{equation*}
    (u^{-1})_{(w^{-1}u)_{i_{m+1}+1}} < (u^{-1})_{(w^{-1}u)_{i_{m+1}}}.
  \end{equation*}
  Simplifying the two expressions in this inequality to
  %the left- and right-hand sides to
$(w^{-1}uu^{-1})_{i_{m+1}+1} = (w^{-1})_{i_{m+1}+1}$ and
$(w^{-1}uu^{-1})_{i_{m+1}} = (w^{-1})_{i_{m+1}}$, respectively,
we obtain the equivalent inequality 
%  \begin{equation}
$ws_{i_{m+1}} < w$.
%\end{equation}

It follows that the expression in (\ref{eq:sigmalong})
and therefore the right-hand side of (\ref{eq:sigmatwosum}) can be written as
\begin{equation}\label{eq:stats2}
%\sigma_B (x^{u,v}) = 
\sum_{\pi^{{B'}C}} \qp{\cross(\pi^{{B'}C})}2
%\sum_U
q^{\incross(U(\pi^{{B'}C}\ntnsp,u,w)},
\end{equation}
where the sum is over path families of type $u^{-1}w$ which cover $G$.
%$G_{s_{i_1}} \cdots G_{s_{i_{m+1}}}$. 
Thus the claim is true by induction.
\end{proof}

The special case $u = e$ of Proposition~\ref{p:sigmastats} yields another proof
%to a  simplification
of the formula (\ref{eq:defectdef2}).
\begin{cor}\label{c:sigmastats}
%  Let $G$ be the wiring diagram of expression $\sprod im$.
  %  have weighted path matrix $B$.
%  Then t
  The coefficients in the expansion
%  \begin{equation*}
    $(1 + T_{s_{i_1}}) \cdots (1 + T_{s_{i_m}}) = \sum_w a_w T_w$
%  \end{equation*}
are given by
\begin{equation*}
  a_w = \sum_\pi q^{\dfct(\pi)},
\end{equation*}
where the sum is over all path families of type $w$ which cover the wiring
diagram of $\sprod im$.
%For $w \in \sn$ we have 
%\begin{equation}\label{eq:cstats}
%\sigma_B (x^{e,w}) = 
%\sum_\pi \qp{\cross(\pi)}2
%%\sum_U
%q^{\dfct(\pi_1,\dotsc,\pi_n)},
%\end{equation}
%where the sum is over path families $\pi$ of type $w$ covering $G$.
%and $U = (\pi_1,\dotsc,\pi_n)$
%is the unique $\pi$-tableaux of shape $(n)$
%satisfying $L(U) = 1 \cdots n$, $R(U) = w$.
\end{cor}
\begin{proof} 
  Substitute $u = e$ in Proposition~\ref{p:sigmastats}.

  The right-hand side of (\ref{eq:stats}) is a sum over
  path families $\pi$ of type $w$, and
  each tableau $U = U(\pi,e,w)$ is simply
  the sequence $\pi = (\pi_1,\dotsc,\pi_n)$.
  Thus an inverted noncrossing of $U$ is simply a noncrossing (\ref{eq:inc})
  in which the upper path has index less than that of the lower path, i.e.,
  a defective noncrossing.  To relate these to all defects,
  %combine these with defective crossings,
  let us temporarily define
  \begin{equation*}
    \begin{aligned}
      \dnc(\pi) &= \text{ number of defective noncrossings of $\pi$},\\
      \dc(\pi) &= \text{ number of defective crossings of $\pi$},
    \end{aligned}
  \end{equation*}
  so that $\dfct(\pi) = \dnc(\pi) + \dc(\pi)$.
  Now observe that for any path familiy of type $w$ we have
  %$\cross(\pi) \geq \inv(\type(\pi))$,
  %and more specifically,
    \begin{equation*}
    \cross(\pi) = \inv(w) + 2\dc(\pi)
    \end{equation*}
    because if paths $\pi_a$, $\pi_b$ cross $k$ times,
    then at most one of those crossings contributes to $\inv(w)$,
    while exactly half of the remaining crossings are defective.
    Thus the right-hand side of (\ref{eq:stats}) becomes
    \begin{equation}\label{eq:rhsstats}
      \sum_\pi \qp{\inv(w) + 2\dc(\pi)}2 q^{\dnc(\pi)}
      = \sum_\pi \qew q^{\dc(\pi)+\dnc(\pi)} = \qew \sum_\pi q^{\dfct(\pi)},
    \end{equation}
    where the sum is over all path families of type $w$ which cover the
    wiring diagram of $\sprod im$.
    %$G$.

%    Substituting $u = e$ on the left-hand side of (\ref{eq:stats})
%    and applying
    By Proposition \ref{p:qewcoeff}, the left-hand side of (\ref{eq:stats})
    is $\qew a_w$.
  Combining this with (\ref{eq:rhsstats}), we have the desired result.
\end{proof}
%  Thus we have
%\begin{equation}\label{eq:statse}
%  \sigma_B (x^{e,w}) = 
%  \sumsb{\pi\\ \type(\pi) = w} \qp{\cross(\pi)}2 q^{\dfct(\pi)}.
%\end{equation}
%%where the sum is over all path families of type $w$ which cover $G$.
%By Proposition~\ref{p:qewcoeff} this is equal to
%%we also have
%%\begin{equation*}
%  $\sigma_B (x^{e,w}) = \qew a_w$,
%%\end{equation*}
%where
%\begin{equation*}
%  (1 + T_{s_{i_1}}) \cdots (1 + T_{s_{i_m}}) = \sum_v a_v T_v.
%\end{equation*}
%It follows that we have
%%Combining this fact with Deodhar's formula (\ref{eq:defectdef2}),
%we have
%\begin{equation*}
%  \sigma_B(x^{e,w}) = \qp{\ell(w)}2 \sumsb{\pi\\ \type(\pi) = w} q^{\dfct(\pi)}
%  \qp{\cross(\pi)\ell(w)}2,
%\end{equation*}
%as desired.
%\end{proof}
%
%and therefore
%\begin{equation}\label{eq:deodharandus}
% \qew \sumsb{\pi\\ \type(\pi) = w} q^{\dfct(\pi)}
%=
%\sumsb{\pi\\ \type(\pi) = w} \qp{\cross(\pi)}2 q^{\dfct(\pi)}.
%\end{equation}
%\begin{quest} Is Equation (\ref{eq:deodharandus}) false?
%  If so, where did I go wrong in deriving it?
%  \end{quest}

\section{Evaluation of induced sign characters}\label{s:indsgneval}

By Theorem~\ref{t:immeval}, the map $\sigma_B$ (\ref{eq:sigmadef})
can be used to evaluate $\epsilon_q^\lambda(D_w)$ when one has a simple
expression for the generating function
$\imm{\epsilon_q^\lambda}(x)$
%which allows us to evaluate
and can evaluate
$\sigma_B(\imm{\epsilon_q^\lambda}(x))$.
Such an expression was given by Konvalinka
and the third author in \cite[Thm.\,5.4]{KSkanQGJ}:
for $\lambda = (\lambda_1,\dotsc,\lambda_r)$, we have
\begin{equation}\label{eq:immepsilon}
  \imm{\epsilon_q^\lambda}(x) =
  \sum_I \qdet(x_{I_1,I_1}) \cdots \qdet(x_{I_r,I_r}),
%  \sum_{(I_1,\dotsc,I_r)} \qdet(x_{I_1,I_1}) \cdots \qdet(x_{I_r,I_r}),
\end{equation}
where
$\qdet$ and $x_{L,M}$ are defined as in Section~\ref{s:qmb},
and
the sum is over all ordered set partitions
$I = (I_1,\dotsc,I_r)$ of $[n]$
%$\{1,\dotsc,n\}$
satisfying $|I_j| = \lambda_j$.
%$\lambda = (\lambda_1,\dotsc,\lambda_r)$,
%and
We will say that such an ordered set
partition has {\em type} $\lambda$.

To evaluate $\sigma_B(\imm{\epsilon_q^\lambda}(x))$, we expand
each term
%in the sum
on the right-hand side of (\ref{eq:immepsilon})
%and recognize the resulting
%monomials as elements
in a monomial basis $\{x^{u,v} \,|\, v \in \sn \}$ of $\Ann$,
%presented in Section~\ref{s:qmb}.
%These bases are best described
where $u = u(I)$ is the concatenation of the $r$ strictly increasing subwords
\begin{equation}\label{eq:subwordsofu}
u_1 \cdots u_{\lambda_1}, \quad 
u_{\lambda_1 +1} \cdots u_{\lambda_1 + \lambda_2}, \quad
u_{\lambda_1 + \lambda_2 + 1} \cdots u_{\lambda_1 + \lambda_2 + \lambda_3}, \quad
%\dotsc, \quad
%u_{\lambda_1 + \cdots + \lambda_{j-1} + 1} \cdots u_{\lambda_1 + \cdots + \lambda_j}, \quad
\dotsc, \quad
u_{n-\lambda_r + 1} \cdots u_n
\end{equation}
formed by listing the elements of each block $I_1, \dotsc, I_r$
in increasing order.
As $I$ varies over all ordered set partitions of $[n]$ of type $\lambda$,
the
%resulting
%Write $u(I)$ for the permutation corresponding to $I$.
%The
permutations $u(I)$ vary over
%are precisely
the Bruhat-minimal representatives $\slambdamin$ of cosets $\slambda u$, where
$\slambda$ is the {\em Young subgroup} of $\sn$
%is described in terms of
%$\slambda$ of $\sn$,
%Define $\slambda$ to be the subgroup of $\sn$
generated by
\begin{equation*}
%\label{eq:youngsubgen}
%J \defeq 
\{ s_1, \dotsc, s_{n-1} \} \ssm 
\{ s_{\lambda_1}, s_{\lambda_1 + \lambda_2}, 
s_{\lambda_1 + \lambda_2 + \lambda_3}, 
\dotsc, 
s_{n-\lambda_r} \}.
\end{equation*}
%and let
%and
%Let $\slambdamin$ be the set of
%%the Bruhat-minimal representatives $\slambdamin$ of cosets
%Bruhat-minimal representatives of cosets
%of the form $\slambda v$, i.e., the elements $u \in \sn$ for which each
%of the subwords
%is strictly increasing.
%It is clear that elements of
%A bijection between $\slambdamin$ and
%ordered set partitions $(I_1, \dotsc, I_r)$ of $[n]$ of type $\lambda$
%is given by $u \mapsto I$ where
%\begin{equation}\label{eq:uI}
%  I_j = \{ u_{\lambda_1 +\cdots+ \lambda_{j-1} + 1}, \dotsc, u_{\lambda_1 +\cdots+ \lambda_j} \}
%  \text{ for } j = 1,\dotsc,r.
%\end{equation}
%Letting
%$u(I)$ denote the permutation in $\slambdamin$ which corresponds
%to $I$, we have
%, i.e., the permutation whose one-line notation is the increasing
%rearrangement of $I_1$, followed by the increasing rearrangement of $I_2$, etc.
%If $u \in \slambdamin$ corresponds to $(I_1, \dotsc, I_r)$,
%then the path indices in row $j$ of $U(u,\pi,\lambda)$
%are simply the elements of $I_j$, in increasing order. 
%Now we have
%In order to interpret $\epsilon_q^\lambda((1+T_{s_{i_1}}) \cdots (1+T_{s_{i_m}}))$
%we refer to \cite[Eqn.~(6.1)]{CHSSkanEKL}: 
Expanding each term on the right-hand side of (\ref{eq:immepsilon})
and applying $\sigma_B$ we have
\begin{equation}\label{eq:lambepsilon}
%\begin{aligned}
%\epsilon_q^\lambda((1+T_{s_{i_1}}) \cdots (1+T_{s_{i_m}}))
%&= \sigma_B(\imm{\epsilon_q^\lambda}(x))\\
  %&= \sum_{(I_1,\dotsc,I_r)}
  \sigma_B(\qdet(x_{I_1,I_1}) \cdots \qdet(x_{I_r,I_r}))
  %\\
  %= \sum_{u \in \slambdamin}
  = \sum_{y \in \slambda}
%&= \sum_{u \in \slambdamin} \sum_{y \in u^{-1}\slambda u} 
%(-1)^{\ell(uy)-\ell(u)}q_{u,uy}^{-1} \sigma_B (x^{u,uy}),
(-1)^{\ell(y)}\qiey \sigma_B (x^{u(I),yu(I)}).
%\end{aligned}
\end{equation}
%(See \cite[Eqn.~(6.1)]{CHSSkanEKL}.)
%where the first sum is over all ordered set partitions $(I_1,\dotsc, I_r)$
%of $[n]$ of type $\lambda$.
%Let us therefore consider evaluations of the form 
%\begin{equation*}
%$\sigma_B (\qiuv x^{u,v}) = \qiuv
%$\sigma_B (x^{u,v})$.
%\end{equation*}
%(See, e.g., \cite{CurtisLIT}.)
%satisfying
%\begin{equation*}
%u_1 < \cdots < u_\lambda, \qquad u_{\lambda_1 +1} < \cdots < u_{\lambda_1 + \la%mbda+2}, 
%\dotsc, u_{n-\lambda_r + 1} < \cdots < u_n.
%\end{equation*}
To combinatorially interpret the
%final
sum in (\ref{eq:lambepsilon})
we may apply Proposition~\ref{p:sigmastats}
%need to
%to expressions $\sigma_B(x^{u,yu})$
%\,|\, u \in \slambdamin, y \in \slambda \}$
%evaluations of the form
%(\ref{eq:stats}) only
%with $u \in \slambdamin$,
%for some partition $\lambda = (\lambda_1,\dotsc,\lambda_r) \vdash n$,
%and $v = uy$ for $y \in u^{-1}\slambdamin u$.
%and
%$v = yu$ for
%$y \in \slambda$,
%Since elements of $\slambdamin$ correspond bijectively to ordered set
%partitions $I$ of $[n]$ of type $\lambda$, and $\sigma_B(x^{u,yu})$ has
%the interpretation given in Proposition~\ref{p:sigmastats},
%it will be convenient to
and compute statistics for tableaux belonging to the set
%define
\begin{equation*}
%  \begin{aligned}
    \mathcal U_I = \mathcal U_I(G)
%    =\{ U(\pi, u, uy ) \,|\, \pi \text{ covers } G, u = u(I),
    \defeq \{ U(\pi, u, yu ) \,|\, \pi \text{ covers } G, u = u(I), 
    \type(\pi) = y \in \slambda \}.
%    \type(\pi) = y \in u^{-1}\slambda u \}.
%  \end{aligned}
\end{equation*}
Note that our restriction on $y$ forces the sink indices of paths in
components
\begin{equation}\label{eq:component}
  (\lambda_1 + \cdots + \lambda_{k-1}+1), \dotsc, (\lambda_1 + \cdots + \lambda_k)
\end{equation}
%of $U(\pi,u,uy)$ to be a permutation of the source indices of the same paths.
of $U(\pi,u,yu)$ to be a permutation of the source indices of the same paths.

On the other hand, the
%final
sum in (\ref{eq:lambepsilon}) has both positive and negative signs.
We will obtain a subtraction-free expression
for the sum
%$\sigma_B(\imm{\epsilon_q^\lambda}(x))$
%$\epsilon_q^\lambda((1+T_{s_{i_1}}) \cdots (1+T_{s_{i_m}}))$
by applying
%defining
a sign-reversing involution
%$\zeta = \zeta_I$ on
to the tableaux in each set $\mathcal U_I$.
%\rightarrow \mathcal U_I$,
%for each ordered set partition $I$ of type $\lambda$.
%This will cancel all negative
%contributions from expressions of the form $\sigma_B(x^{u,uy})$
%with some positive contributions of the same form.
It will be convenient to define this involution on a second set
$\mathcal T_I$ of tableaux, in obvious bijection $\mathcal U_I$.
%with the first set.
%Fix a partition $\lambda = (\lambda_1, \dotsc, \lambda_r) \vdash n$,
%an ordered set partition $I$ of $[n]$ of type $\lambda$, 
%and a star 
%%descending star 
%network $G$,
%%and let $\mathcal T_\lambda = \mathcal T_\lambda(F)$ 
%and
For a wiring diagram $G$
%of $\sprod im$,
%representing $(1 + T_{s_{i_1}}) \cdots (1+ T_{s_{i_m}})$,
let $\mathcal T_I = \mathcal T_I(G)$ be 
the set of all column-closed, left column-strict $G$-tableaux 
$W$ of shape $\lambda^\tr$
such that $L(W^\tr)_k = I_k$ (as sets) for $k = 1,\dotsc,r$.
%Observe that all column-strict $G$-tableaux of type $e$ 
%satisfying $L(U^\tr)_j = I_j$ (as sets) for $j = 1,\dotsc,r$
%belong to $\mathcal T_I$.  ($\star$ Is this important?)
The bijection $\delta = \delta_I\ntksp: \mathcal U_I \rightarrow \mathcal T_I$
maps $U \in \mathcal U$
to the left column-strict $G$-tableau $W$ of shape $\lambda^\tr$
%is given by
%defining a tableau $V$ of shape $\lambda$ whose $j$th row
whose $k$th column
%$V_j$
consists of entries
%to be the one-rowed tableau consisting of components
%$(\lambda_1 + \cdots + \lambda_{k-1}+1), \dotsc,
%(\lambda_1 + \cdots + \lambda_k)$
(\ref{eq:component})
of $U$.
%\in \mathcal U_I$,
%for $j = 1, \dotsc, r$,
%and by defining $\delta_I(U) = V^\tr$.

Since $U$ and $\delta(U)$ contain the same path family, it is easy to see
that $\delta$ does not affect the statistic $\cross$.
On the other hand, it changes the statistic $\incross$ in a very simple way.
Define $\cdncross(U)$ to be the number of defective noncrossings of pairs of 
paths appearing in the same column of $U$, i.e.,
the number of occurrences of (\ref{eq:inc}) where $b < a$ and $\pi_b$, $\pi_a$ 
appear in the same column of $U$.

\begin{lem}\label{l:TU}
%Let $W = \delta(U)$.
%  Then we have
  Let $I$ be an ordered set partition.
  For $U \in \mathcal U_I$ we have
\begin{equation}\label{eq:incdn}
\incross(U) = \incross(\delta(U)) + \cdncross(\delta(U)).
\end{equation}
\end{lem}
\begin{proof}
  Let $\lambda = (\lambda_1, \dotsc, \lambda_r)$ be the type of $I$,
  %let $u = u(I)$, and 
  and let $\pi$ be the path family in a $G$-tableau $U$,
  where $G = G_{[i_1,i_1+1]} \circ \cdots \circ G_{[i_m,i_m+1]}$.
  %is the wiring diagram in (\ref{eq:wdnetwork}).
  %and
  %an ordered set partition of type $\lambda$.
  Choose an index $j$, $1 \leq j \leq m$ and let $\pi_a$, $\pi_b$
  be the two paths which intersect at the central vertex of $G_{[i_j,i_j+1]}$,
  with $\pi_b$ entering from above,
  as in (\ref{eq:inc}) or as in the first figure in (\ref{eq:cross}).
  For some $k$, $\pi_b$ appears among the entries
  \begin{equation}\label{eq:subtableau}
    \lambda_1 + \cdots + \lambda_{k-1} + 1, \dotsc, \lambda_1 + \cdots + \lambda_k
  \end{equation}
  of $U$. The indices of these $\lambda_k$ paths increase from left to right
  in $U$, since the
  %source
  indices of all paths in $U$ form the
  permutation $u = u(I) \in \slambdamin$ (\ref{eq:subwordsofu}).
  
  Suppose that index $j$ is an inverted noncrossing of $U$ and therefore
  contributes $1$ to $\incross(U)$.
  %the left-hand side of (\ref{eq:incdn}).
  %$j$ is an inverted noncrossing of $U$,
  Then $\pi_a$ and $\pi_b$ intersect as in (\ref{eq:inc}) and
%  some paths $\pi_a$, $\pi_b$ intersect as in (\ref{eq:inc})
%  at the central vertex of $G_{[i_j,i_j+1]}$ of (\ref{eq:wdnetwork}),
%  with $\pi_b$ intersecting $\pi_a$ from above (\ref{eq:inc})
  %  with
  $\pi_b$ appears earlier than $\pi_a$ in $U$.
%  in these entries
  If $\pi_a$ appears among the entries (\ref{eq:subtableau}),
  then we must have $b < a$.
  Thus $\pi_a, \pi_b$ both appear in column $k$ of $\delta(U)$,
  and $j$ contributes
  $0$ to $\incross(\delta(U))$ and
  $1$ to $\cdncross(\delta(U))$.
  On the other hand, if $\pi_a$ does not
  appear in entries (\ref{eq:subtableau}) of $U$, then it appears
  strictly to the right of column $k$ of $\delta(U)$.
  Thus $j$ contributes $1$ to $\incross(\delta(U))$
  and $0$ to $\cdncross(\delta(U))$.

  Now suppose that $j$ is not an inverted noncrossing in $U$
  and therefore contributes $0$ to $\incross(U)$.
%  the left-hand side of (\ref{eq:incdn}).
  If $j$ is a crossing in $\pi$, then it
  %is also a crossing in
  %$\delta(U)$, and
  contributes $0$ to $\incross(\delta_i(U))$ and $\cdncross(\delta(U))$.
  If $j$ is a noncrossing of $\pi$, then $\pi_a$ appears before $\pi_b$ in $U$.
  If $\pi_a$ appears before the entries (\ref{eq:subtableau}) of $U$,
  then in $\delta(U)$ it appears in an earlier column than $\pi_b$ and
  contributes $0$ to $\incross(\delta_i(U))$ and $\cdncross(\delta(U))$.
  If $\pi_a$ appears as one of the entries (\ref{eq:subtableau}) of $U$,
  then it appears in the same column of $\delta(U)$ as $\pi_b$
  and satisfies $a < b$.  Again $j$
  contributes $0$ to $\incross(\delta_i(U))$ and $\cdncross(\delta(U))$.
%  $\incross(\delta(U)) + \cdncross(\delta(U))$.
%  Then some paths $\pi_a$, $\pi_b$ intersect as in (\ref{eq:inc})
%  at the central vertex of $G_{[i_j,i_j+1]}$.
%  If $j$ is an inverted noncrossing of $\delta(U)$, then
%  %of (\ref{eq:wdnetwork})
%  %with $\pi_b$ intersecting $\pi_a$ from above (\ref{eq:inc})
%  $\pi_b$ appears in an earlier column of $\delta(U)$ than $\pi_a$.
%  %and therefore earlier in $U$ than $\pi_a$.
%  %in $\delta(U)$.
%  By the definition of $\delta$, it must also appear earlier in $U$,
%  and thus $j$ contributes $1$ to $\incross(U)$.
%  If $j$ is not an inverted noncrossing of $\delta(U)$, then $\pi_a, \pi_b$
%  appear in the same column of $\delta(U)$, and we have $b < a$.
%  Since $\delta(U)$ is left column-strict, $\pi_b$ then appears
%  before $\pi_a$ in $U$, and again $j$ contributes $1$ to $\incross(U)$.
\end{proof}

Now we define the involution
%\begin{equation*}
%\label{eq:zeta}
%\zeta: \mathcal T_\lambda \rightarrow \mathcal T_\lambda
$\zeta = \zeta_I\ntksp: \mathcal T_I \rightarrow \mathcal T_I$
%\end{equation*}
as follows.
%Assume we have $W \in \mathcal T_I$.
%($\star$ or something like this).  
%($\star$ This definition isn't quite right.  Fix it.)  (Yes it is?)
\begin{enumerate}
\item If $W \in \mathcal T_I$ is column-strict,
  %tableau of type $e$,
  then define $\zeta(W) = W$.
\item Otherwise, 
  \begin{enumerate}
  \item Let $t$ be the 
%least
greatest index such that column $t$ of $W$ is not column-strict.
\item Let $k$ be the greatest index such that
two paths $\pi_j$, $\pi_{j'}$ with $j, j' \in I_t$
both pass through the central vertex of
%the factor network
$G_{[i_k, i_{k+1}]}$,
and let $\hat \pi = (\hat \pi_1, \dotsc, \hat \pi_n)$
be the path family
obtained from $\pi$ 
%$\pi'_j$, $\pi'_{j'}$ be the paths obtained 
by swapping the terminal subpaths of $\pi_j$ and $\pi_{j'}$, beginning
at the central vertex of $G_{[i_k,i_{k+1}]}$.
($\hat \pi_i = \pi_i$ for $i \notin \{j,j'\}.$)
%and let $\pi'_j$, $\pi'_{j'}$ be the paths which are equal to 
%$\pi_j$, $\pi_{j'}$, respectively, from sources $j$, $j'$ to the 
%central vertex of $G_{[i_k,i_{k+1}]}$, and which are equal to $\pi_{j'}$, $\pi_j$
%respectively from that point to their sinks.
%\item Let $(j,j')$ be the 
%%right-to-left 
%lexicographically greatest 
%%lexicographically least
%pair of indices in 
%%$L(U)_i$ 
%$I_t$ such that 
%$j < j'$, path $\pi_j$ appears to the left of path $\pi_{j'}$
%$\pi_j$ and $\pi_{j'}$ 
%%appear in $U_i$, and the two paths 
%intersect ($j < j'$).
%%Let $(j,j')$ be the lexicographically greatest pair of indices 
%%such that $j < j'$, path $\pi_j$ appears to the left of path $\pi_{j'}$
%%in $U_i$, and the two paths intersect.
%%  \item Let $(k,k')$ be the sink indices of paths $\pi_j$ and $\pi_{j'}$, respectively.
\item
  %Using the above indices $j, j'$,
  Define $\zeta(W)$ to be the tableau obtained from $W$ by replacing
$\pi$ by $\hat \pi$.
%  $\pi_j$, $\pi_{j'}$ by $\pi'_j$, $\pi'_{j'}$, respectively. 
%the unique paths in $F$ from
%source $j$ to sink $k'$ and source $j'$ to sink $k$.
  \end{enumerate}
\end{enumerate}

Observe that each fixed point $W$ of $\zeta$ has type $e$ since
it is column-closed and column strict.
On the other hand, when $W \in \mathcal T_I$
is not a fixed point of $\zeta$,
one can show that the two tableaux
$\delta^{-1}(W)$, $\delta^{-1}(\zeta(W))$ in $\mathcal U_I$
%which correspond to $W$, $\zeta(W)$
are closely related.
\begin{lem}\label{l:lengthdiff}
  Let $I = (I_1, \dotsc I_r)$ be an ordered set partition of type $\lambda$,
  and define $u = u(I)$ as in (\ref{eq:subwordsofu}).
  Let $G$-tableaux $W \in \mathcal T_I$ and $U, \widehat U \in \mathcal U_I$
  satisfy $W = \delta(U) \neq \zeta(W) = \delta(\widehat U)$,
  and define path families $\pi$, $\hat \pi$ as above.
%Let $U, V \in \mathcal U_I$ satisfy $U \neq V$,
%$\zeta(\delta(U)) = \delta(V)$.
%Then we have
%the following.
%\begin{enumerate}[(i)]
%\item $U = U(\pi, u, v)$, $V = U(\pi', u, v')$
%\begin{equation*}
%\end{equation*}
%for
%some
%Then for path families $\pi$, $\pi'$
%related as above,
%defined as above, and
Then for some generator $s \in \slambda$ and
some permutations $v, \hat v = sv \in \slambda u$
%satisfying
%\item The permutations $v$, $v'$ satisfy
%$|\ell(v) - \ell(v')| = 1$,
we have
$U = U(\pi, u, v)$, $\widehat U = U(\hat \pi, u, \hat v)$.
%\end{enumerate}
\end{lem}
\begin{proof}
  %(i)
  The tableaux $U$, $\widehat U$ contain the same path
  families as $W$ and $\zeta(W)$,
  respectively,
  %For any tableau $W \in \mathcal U_I$, the tableau $\delta(W)$
  %contains the same path family as $W$.
%  By definition of $\zeta$
%then implies that for $W \neq \zeta(W)$, the path families
  and these path families are $\pi$, $\hat \pi$, as defined in
  the definition of $\zeta$.
  %are related as claimed.

  Since elements of $\mathcal T_I$ are column-closed, it follows that
the set of sink indices of paths in each column 
is equal to the set of source indices of paths in the same column.
Thus the sequences $v$, $\hat v$ of sink indices,
read bottom-to-top in columns $1, \dotsc, r$, belong to $\slambda u$.
%But these are precisely $v$ and $\hat v$.
%\noindent
%(ii)
By the definition of $\zeta$, the permutations $v$ and $\hat v$
differ from one another in exactly two positions: those holding the 
letters $v_j = \hat v_{j'}$ and $v_{j'} = \hat v_j$.  Both letters 
belong to the same block of the ordered set partition $I$.
Since we used the rightmost vertex in $\pi_j \cap \pi_{j'}$
to define $\hat \pi_j$ and $\hat \pi_{j'}$, the letters must appear
consecutively in $v$ and in $\hat v$.
It follows that $\hat v = sv$ for some adjacent transposition $s \in \slambda$.
\end{proof}

Furthermore, when $W \in \mathcal T_I$ is not a fixed point of $\zeta$,
the values of the statistics $\incross$ and $\cdncross$ on
$W$ and $\zeta(W)$ are closely related.
%two tableaux $\delta^{-1}(W)$, $\delta^{-1}(\zeta(W))$ in $\mathcal U_I$
%%which correspond to $W$, $\zeta(W)$
%are closely related.
\begin{prop}\label{p:zeta}
%  Fix $G$ representing $(1 + T_{s_{i_1}}) \cdots (1+ T_{s_{i_m}})$, and
%  fix $W \in \mathcal T_I$ satisfying $\zeta(W) \neq W$.
%  Let $W$ contain path family $\pi = (\pi_1, \dotsc, \pi_n)$,
%  and let $\zeta(W)$ contain path family $\hat \pi$
%  ($= \pi$ with $\pi_j$, $\pi_{j'}$ replaced by $\hat \pi_j$, $\hat \pi_{j'}$
%  as in the definition of $\zeta$).
%  Let $v = R(\delta_I^{-1}(W))$, $\hat v = R(\delta_I^{-1}(\zeta(W)))$.
  Let $W \in \mathcal T_I$ satisfy
  %$W \neq \zeta(W)$,
%  $W = \delta(U) \neq \zeta(W) = \delta(\widehat U)$
    $W = \delta(U(\pi,u,v)) \neq \zeta(W) = \delta(U(\hat \pi,u,\hat v))$
  for $\pi$, $\hat \pi$, $v$, $\hat v$ as in
%  for $U = U(\pi, u, v)$, $\widehat U = U(\hat \pi, u, \hat v)$, as in
%  $U = U(\pi$, $\hat \pi$, $v$, $\hat v$ as in
  Lemma~\ref{l:lengthdiff}.
  Then we have
% involution $\zeta$ satisfies 
\begin{equation}\label{eq:incrosseq}
\incross(\zeta(W)) = \incross(W),
\end{equation}
%and
\begin{equation}\label{eq:cdncrosseq}
\cdncross(\zeta(W)) 
%+ \incross(\zeta(U)) 
%+ \cross(\zeta(U))/2
%+ \frac{\cross(\zeta(U))}2
+ \frac{\cross(\hat \pi)}2
=
\begin{cases}
\cdncross(W) 
%+ \incross(U) 
+ \frac{\cross(\pi)+1}2  
%&\text{if $\type(\pi) < \type(\hat \pi)$,}\\
&\text{if $v < \hat v$,}\\
\cdncross(W) 
%+ \incross(U) 
+ \frac{\cross(\pi)-1}2  
%&\text{otherwise.}
&\text{if $v > \hat v$.}
\end{cases}
\end{equation}
%and
%\begin{equation}\label{eq:signreverse}
%  |\ell(\type(\hat \pi)) - \ell(\type(\pi))| = 1.
%\end{equation}
%($\star$ or something like that)
%$\rinv(\zeta(U)^\tr) = \rinv(U^\tr)$.  
%($\star$ Or what is the identity we want it to satisfy?)
\end{prop}
\begin{proof}
  To verify (\ref{eq:incrosseq}),
  let $t$, $k$ be as in the definition of $\zeta$
  %be the wiring diagram covered by $\pi$,
  %$\hat \pi$,
%let $t$ be the greatest index such that column $t$ of $W$ is not column-strict,
and let $\ell$ be an index which does not belong to $I_t$.
Then $\pi_{\ell}$ is an entry of $W$ and $\zeta(W)$ which does not appear 
in column $t$.
%and consider the path
%s $\pi_j$, $\pi_{j'}$, $\hat \pi_j$, $\hat \pi_{j'}$ and 
%$\pi_{j''}$ in $U$ and $\widehat U$.
Clearly, any point of intersection between 
%$\pi_j$ and 
$\pi_{\ell}$ and 
$\pi_j$ or $\pi_{j'}$
which occurs in $G$ to the left of the central vertex of $G_{[i_k, i_k+1]}$
exactly matches an intersection between 
$\pi_{\ell}$ and 
$\hat \pi_j$ or $\hat \pi_{j'}$.
%$\hat \pi_j$ and $\pi_{j''}$.
On the other hand, suppose that $\pi_{\ell}$ and $\pi_j$ (or $\pi_{j'}$) 
have a point of intersection in $G$ to the right of 
the central vertex of $G_{[i_k, i_k+1]}$.
Since $j$ and $j'$ both belong to $I_t$,
this point is an inverted noncrossing of 
$\pi_{\ell}$ and $\pi_j$ (or $\pi_{j'}$) in $W$
if and only if it is an inverted noncrossing of
$\pi_{\ell}$ and $\hat \pi_{j'}$ (or $\hat \pi_j$) in $\zeta(W)$.
%Thus we have 
%$\incross(\zeta(U)) = \incross(U)$.
%To see that $\zeta$ satisfies 
%observe that since $j$, $j'$ both belong to $I_t$,
%the paths $\pi_j$, $\pi_{j'}$ both belong to the same column of $U$.
%Thus any inverted noncrossing of $\pi_j$ with some path $\pi_{j''}$ in a later
%column of $U$

To verify (\ref{eq:cdncrosseq}), consider the intersection of 
$\pi_j$ and $\pi_{j'}$ at the central vertex of $G_{[i_k, i_k+1]}$ and the
tableaux $U(\pi, u, v)$, $U(\hat \pi, u, \hat v)$.
%$U = \delta_I^{-1}(W)$, $\widehat U = \delta_I^{-1}(\zeta(W))$.
%of the map $\zeta$ on crossings and noncrossings of $\pi$.
%We have $v = R(U)$, $\hat v = R(\widehat U)$.
%If $\type(\pi) > \type(\hat \pi)$
If $v > \hat v$
%in the Bruhat order
then our choice of $(j,j')$ implies
that this intersection is either a defective
noncrossing or a crossing that sends $\pi_j$ above $\pi_{j'}$.
In the first case, the map $\zeta$ 
removes exactly one defective noncrossing and creates exactly one crossing, 
changing the statistic sum $\cdncross + \cross/2$ by $-1 + 1/2 = -1/2$.
In the second case, the map $\zeta$ 
removes exactly one crossing and creates exactly one proper noncrossing, 
changing the same sum by $-1/2 + 0 = -1/2$.
Similarly, if
%$\type(\pi) < \type(\hat \pi)$
$v < \hat v$
then our choice of $(j,j')$ implies
that the intersection is either a proper noncrossing or a crossing that 
sends $\pi_{j'}$ above $\pi_j$.  In both cases, the statistic sum increases
by $1/2$.
%To verify (\ref{eq:signreverse}), assume without loss of generality
%that $\pi$ has more crossings than $\hat \pi$.
%Then $\type(\pi)$ is equal in $\sn$ to a subexpression
%$s_{j_1} \cdots s_{j_m}$ of $s_{i_1} \cdots s_{i_r}$, and
%$\type(\hat \pi)$ is equal in $\sn$ to a subexpression of $s_{j_1} \cdots s_{j_m}$
%in which one generator has been removed.
%It follows that $\sgn(\type(\pi)) = (-1)^m$ and 
%$\sgn(\type(\hat \pi)) = (-1)^{m-1}$.
\end{proof}

Finally we can state and justify a
%the desired
subtraction-free formula for
$\epsilon_q^\lambda((1 + T_{s_{i_1}}) \cdots (1 + T_{s_{i_m}}))$.
\begin{thm}\label{t:qepsilon}
  Let $G$ be the wiring diagram of $\sprod im$.
  %represent $(1 + T_{s_{i_1}}) \cdots (1 + T_{s_{i_m}}) \in \hnq$.
Then for $\lambda \vdash n$ 
we have
\begin{equation}\label{eq:epsilonmain}
\epsilon_q^\lambda((1 + T_{s_{i_1}}) \cdots (1 + T_{s_{i_m}})) 
= \sum_W q^{\incross(W)+\cross(W)/2},
\end{equation}
where the sum is over all column-strict $G$-tableaux 
%$U$ 
of 
type $e$
and 
shape $\lambda^\tr$.
\end{thm}
\begin{proof}
  Let $B$ be the path matrix of $G$. Combining the Theorems~\ref{t:immeval}
  and \cite[Thm.\,5.4]{KSkanQGJ} (i.e., (\ref{eq:immepsilon}))
  with the identity (\ref{eq:lambepsilon}), we see that
  the left-hand side of (\ref{eq:epsilonmain}) is
  \begin{equation}\label{eq:epsilonsecond}
    \begin{aligned}
      \sigma_B(\imm{\epsilon_q^\lambda}(x)) &=
%      \sum_{I = (I_1,\dotsc,I_r)}
      \sum_I
      \sigma_B(\qdet(x_{I_1,I_1}) \cdots \qdet(x_{I_r,I_r}))\\
      &=
      \sum_I
      \sum_{\smash{y \in \slambda}} (-1)^{\ell(y)}\qiey
      \sigma_B (x^{u(I),yu(I)}),
%      \sum_{I = (I_1,\dotsc,I_r)}
      %\sum_{u \in \slambdamin}
%   \sum_{y \in \slambda} (-1)^{\ell(y)}q_{e,y}^{-1} \sigma_B (x^{u(I),yu(I)}),
    \end{aligned}
  \end{equation}
    where the first two sums are over ordered set partitions
    $I = (I_1,\dotsc,I_r)$ of $[n]$ of type $\lambda$.
    Fixing one such partition $I$ and writing $u = u(I)$, we may use
    Proposition~\ref{p:sigmastats} and Lemma~\ref{l:TU} to
    express the third sum as
%  and $u = u(I)$.  
%  (i.e., \cite[Thm.\,5.4]{KSkanQGJ}), 
%  with the identities
  %  and let $(I_1,\dotsc,I_r)$ be a set
%partition of $[n]$ of type $\lambda$.
%By 
%(\ref{eq:subwordsofu}) -- (\ref{eq:hqlambetaalt}),
%(\ref{eq:lambepsilon}),
%there is a permutation $u \in \slambdamin$ 
%corresponding to $(I_1,\dotsc,I_r)$ such that we have
%\begin{equation*}
%\sigma_B(\qdet(x_{I_1,I_1}) \cdots \qdet(x_{I_r,I_r})) = 
%\sum_{y \in \slambda} (-1)^{\ell(y)}  q_{e,y}^{-1} \sigma_B ( x^{u, yu} ).
%\end{equation*}
%By Proposition~\ref{p:sigmastats}
%this is equal to
\begin{equation}\label{eq:ypiinvuupi}
  %  \sum_{u \in \slambdamin}
%  \sum_I
%  \sum_{(y,\pi)} (-1)^{\ell(y)}
  \sum_{y \in \slambda}
  \sum_\pi 
  (-1)^{\ell(y)} \qiey
%  q^{\rinv(U(u,\pi,\lambda)^\tr)},
%\end{equation}
%\begin{equation}
%\sigma_B (x^{u,v}) = 
%\sum_\pi
\qp{\cross(\pi)}2
%\sum_U
%q^{\incross(U(\pi,u(I),yu(I)))},
q^{\incross(U(\pi,u,yu))}
%,
%\end{equation}
%\begin{equation}\label{eq:sumtocancel}
%  \sum_{(y,\pi)} (-1)^{\ell(uy)-\ell(y)}  q_{u,uy}^{-1}
%  \sum_I
  %  \sum_{\smash{u \in \slambdamin}}
=
\sum_{y \in \slambda}
  \sum_\pi 
(-1)^{\ell(y)} \qiey
  %\sum_{(y,\pi)} (-1)^{\ell(y)}  q_{e,y}^{-1}
\qp{\cross(\pi)}2 q^{\incross(W)+\cdncross(W)},
%  q^{\incross(\delta(U))+\cdncross(\delta(U))},
\end{equation}
where
the inner
%$u = u(I)$ and
sums are over
%pairs $(y,\pi)$
%such that 
%$y \in u^{-1}\slambda u$ and $\pi$ is a path family of type $y$ which covers $G$.
%with $y \in \slambda$ and $\pi$ is a
path families $\pi$ of type
%$u(I)^{-1}yu(I)$
$u^{-1}yu$
which cover $G$,
and where $W = \delta_I(U(\pi,u,yu))$.
As $y$ and $\pi$ vary in the above sums,
%$y$ and $\pi$ vary this way,
$U(\pi,u,yu)$
varies over all tableaux in $\mathcal U_I$,
and
%The tableau
%$U = U(\pi,u,yu) \in \mathcal U_I$ is then uniquely determined by $(y,\pi)$.
%If such a path family exists, it is necessarily unique.  
%By Lemma~\ref{l:TU} this in turn is equal to
%where $W = \delta_I(U(\pi,u,yu))$
%Thus as $y,\pi$ vary as above,
%we have that
%Now $W$
$W$ varies over all tableaux in $\mathcal T_I$. 
%Thus by (\ref{eq:luylu}) this sum is equal to
%\begin{equation}\label{eq:Vprecancel}
%\sum_{V \in \mathcal T_I} (-1)^{\rinv(V_1) + \cdots + \rinv(V_r)} q^{\rinv(V^\tr)}.
%\end{equation}

Now consider a tableau $W \in \mathcal T_I$ which satisfies $\zeta(W) \neq W$.
Let tableaux $W$ and $\zeta(W)$ contain path families
$\pi$ of type $u^{-1}yu$ and $\hat \pi$ of type $u^{-1}\hat yu$, respectively.
By Lemma~\ref{l:lengthdiff} we have $\hat y = sy$ for some $s \in \slambda$.
Assume without loss of generality that $y < \hat y$
%$\ell(y) < \ell(\hat y) = \ell(y) + 1$.
and consider the term
%of the innersum
on the right-hand side of
(\ref{eq:ypiinvuupi}) corresponding to
%$W$ and
$\zeta(W)$,
%Then the term of the inner sum of (\ref{eq:sumtocancel})
%The term corresponding to $W$ is
%\begin{equation}\label{eq:Wcont}
%  (-1)^{\ell(y)} \qm{\ell(y)}2 \qp{\cross(\pi)}2 q^{\incross(W) + \cdncross(W)},
%\end{equation}
%while the term corresponding to $\zeta(W)$ is
\begin{equation*}
  (-1)^{\ell(\hat y)}
%  \qm{\ell(\hat y)}2
  \qieyp
  \qp{\cross(\hat \pi)}2 q^{\incross(\zeta(W)) + \cdncross(\zeta(W))}.
\end{equation*}
By Proposition~\ref{p:zeta}, this is
\begin{equation*}
  (-1)^{\ell(y)+1}
  %  \qm{\ell(y)}2
  \qiey \qm12
  \qp{\cross(\pi)+1}2 q^{\incross(W) +\cdncross(W)}
  =
  -(-1)^{\ell(y)}
  \qiey
  \qp{\cross(\pi)}2 q^{\incross(W) +\cdncross(W)},
\end{equation*}
i.e., the opposite of the term corresponding to $W$.
%This is the opposite of the expression in (\ref{eq:Wcont}).
%Thus for $W \neq \zeta(W)$,
%the terms of (\ref{eq:sumtocancel}) indexed by
%$(y,\pi)$, $(\hat y,\hat \pi)$ corresponding to $W$, $\zeta(W)$
%\neq W$
%cancel to contribute $0$.
%to the sum (\ref{eq:sumtocancel}), leaving
Thus it suffices to sum the right-hand side of (\ref{eq:ypiinvuupi})
%corresponding to
over only the pairs $(y,\pi)$ corresponding
to tableaux $W$ satisfying $W = \zeta(W)$.
By the definition of $\zeta$, each such tableau $W$
is column-strict and therefore satisfies $\cdncross(W) = 0$.
By the definition of $\mathcal T_I$, each such tableau has type $e$.
%and shape $\lambda^\tr$.
Thus each tableau
\begin{equation*}
  \delta_I^{-1}(W) = U = U(\pi, u, yu) \in \mathcal U_I
\end{equation*}
satisfies $y = \type(\pi) = e$.
It follows that the
%corresponding term in the sum on the
right-hand side of (\ref{eq:ypiinvuupi})
and the third sum in (\ref{eq:epsilonsecond}) are equal to
\begin{equation*}
 \sum_W  \qp{\cross(\pi)}2 q^{\incross(W)},
\end{equation*}
where the sum is over all tableau $W$ in $\mathcal T_I$ which
are column-strict of type $e$.
By the definition of $\mathcal T_I$,
each such tableau has shape $\lambda^\tr$.
Thus the three expressions in (\ref{eq:epsilonsecond})
are equal to the right-hand side of (\ref{eq:epsilonmain}).
%which are fixed by $\zeta$.
%($\star$ Continue.)
\end{proof}

To illustrate the theorem, we compute 
$\smash{\epsilon_q^{21}}((1+T_{s_1})(1+T_{s_2})(1+T_{s_1}))$ 
using the wiring diagram
%consider the wiring diagram 
(\ref{eq:Gbraidfig}).
There are two path families of type $e$ which cover $G$, and one 
column-strict $G$-tableau of shape $21^\tr = 21$ for each:
\begin{equation}\label{eq:Gtableaux}
\begin{tikzpicture}[scale=.5,baseline=-15]
  \draw[dotted, thick] (0,0) -- (1,0) -- (1.5,-.5) -- (2,0) -- (3,0);
  \draw[dashed] (0,-1) -- (.5,-1.5) -- (1.5,-.5) -- (2.5,-1.5) -- (3,-1);
  \draw[-,thick] (0,-2) -- (.5,-1.5) -- (1,-2) -- (2,-2) -- (2.5,-1.5) -- (3,-2);
  \node at (-.5,0) {$\pi_3$};
  \node at (-.5,-1) {$\pi_2$};
  \node at (-.5,-2) {$\pi_1$};
\end{tikzpicture} 
\;, \quad
%\raisebox{-7mm}
%{\includegraphics[height=15mm]{braid1.eps}}\ , \quad
U_\pi = \tableau[scY]{\pi_3 | \pi_1,\pi_2}
\;; \qquad
%\raisebox{-7mm}
%{\includegraphics[height=15mm]{tableaupi.eps}}\ , \qquad \qquad
\begin{tikzpicture}[scale=.5,baseline=-15]
  \draw[dotted, thick] (0,0) -- (1,0) -- (1.5,-.5) -- (2,0) -- (3,0);
  \draw[dashed] (0,-1) -- (1,-2) -- (2,-2) -- (3,-1);
  \draw[-,thick] (0,-2) -- (1.5,-.5) -- (3,-2);
  \node at (-.5,0) {$\rho_3$};
  \node at (-.5,-1) {$\rho_2$};
  \node at (-.5,-2) {$\rho_1$};
\end{tikzpicture} 
\;, \quad
%\raisebox{-7mm}
%{\includegraphics[height=15mm]{braid2.eps}}\ , \quad
U_\rho = \tableau[scY]{\rho_3 | \rho_2,\rho_1} \;.
%\raisebox{-7mm}
%{\includegraphics[height=15mm]{tableausig.eps}}\ .
%\ \begin{matrix}
%\pi_3 & \\
%\pi_1 & \pi_2 
%\end{matrix},
\end{equation}
%where $\pi = (\pi_1, \pi_2, \pi_3)$ is the unique noncrossing path family
%covering $G$.  
Tableau $U_\pi$ contributes 
$q^{\incross(U_\pi)}q^{\cross(U_\pi)/2} = q^1q^{0/2} = q$,
since $\pi$ has no crossings,
and for only one of its noncrossings are the two paths
inverted in $U_\pi$:
$\pi_3$ intersects $\pi_2$ from above 
and appears in an earlier column of $U_\pi$.
Tableau $U_\rho$ contributes 
$q^{\incross(U_\rho)}q^{\cross(U_\rho)/2} = q^1q^{2/2} = q^2$,
since $\rho$ has two crossings, 
and for its unique noncrossing, the two paths are inverted in $U_\rho$:
$\rho_3$ intersects $\rho_1$ from above 
and appears in an earlier column of $U_\rho$.
Adding the two contributions together, we have
$\smash{\epsilon_q^{21}}((1+T_{s_1})(1+T_{s_2})(1+T_{s_1})) = q + q^2$. 

Using the other reduced expression $s_2s_1s_2$ for the long element of $\mfs 3$,
one similarly computes
$\smash{\epsilon_q^{21}}((1+T_{s_2})(1+T_{s_1})(1+T_{s_2})) = q + q^2$.
It is not generally true, however, that distinct reduced expressions
$\sprod im$ and $\sprod jm$ for $w \in \sn$ lead to equal evaluations
%($\star$ Mention that if $s_{i_1} \cdots s_{i_\ell}$ and $s_{j_1} \cdots s_{j_\ell}$
%are two different reduced expressions for $w \in \sn$, then
$\epsilon_q^\lambda((1+T_{s_{i_1}}) \cdots (1+T_{s_{i_\ell}}))$ and
$\epsilon_q^\lambda((1+T_{s_{j_1}}) \cdots (1+T_{s_{j_\ell}}))$.
%are not necessarily equal. )
For example, consider the reduced expressions
$s_3s_2s_1s_2$, $s_3s_1s_2s_1$ for $3241 \in \mfs 4$,
the corresponding wiring diagrams $G$, $H$,
and the $G$- and $H$-tableaux of type $e$ and shape $211$.
%but the wiring diagram $G = G_{[3,4]} \circ G_{[2,3]} \circ G_{[1,2]} \circ G_{[2,3]}$
It is easy to see that there is
only one column-strict $G$-tableau of this type and shape,
while there are no such column-strict $H$-tableaux.
Since the $G$-tableau has one inverted noncrossing and two crossings,
%of this shape, and that we have
%can be covered by one path family
%with
%Thus
we have
\begin{equation*}
  \begin{gathered}
  \epsilon_q^{31}((1+T_{s_3})(1+T_{s_2})(1+T_{s_1})(1+T_{s_2})) = q^2,\\
  \epsilon_q^{31}((1+T_{s_3})(1+T_{s_1})(1+T_{s_2})(1+T_{s_1})) = 0.
  \end{gathered}
  \end{equation*}
%because there is one $
%because
%in the corresponding wiring diagrams

A special case of Theorem~\ref{t:qepsilon} allows one to
%coincides with a special case of 
%\cite[Thm.\,6.4]{CHSSkanEKL}, which
combinatorially interpret evaluations of $\epsilon_q^\lambda$ at
certain elements $\qew C'_w(q)$ 
%\begin{equation*}
%  \{ \qew C'_w(q) \,|\, \text{$w$ \avoidsp} \}
%\end{equation*}
of the Kazhdan-Lusztig basis of $\hnq$.
Call a permutation {\em $321$-hexagon-avoiding}
if it avoids
the patterns $321$, $56781234$, $56718234$, $46781235$, $46718235$.
\begin{cor}\label{c:321hex}
  Let $G$ be the wiring diagram of
  a reduced expression $\sprod im$ for
  a $321$-hexagon-avoiding permutation $w \in \sn$.
Then we have
\begin{equation}\label{eq:321hex}
\epsilon_q^\lambda(\qew C'_w(q)) = \sum_U q^{\incross(U)+\cross(U)/2},
\end{equation}
where the sum is over all column-strict $G$-tableaux 
%$U$ 
of type $e$ and shape $\lambda^\tr$.
\end{cor}
\begin{proof}
  Billey and Warrington~\cite[Thm.\,1]{BWHex} showed that
  for every reduced expression $\sprod im$ of a $321$-hexagon-avoiding
  permutation $w$, we have
  $\qew C'_w(q) = (1 + T_{s_{i_1}}) \cdots (1 + T_{s_{i_m}}).$
\end{proof}

A subclass of the $321$-hexagon-avoiding permutations is the set
of permutations avoiding the patterns $321$ and $3412$.
%and $4231$.
The special case of Corollary~\ref{c:321hex} corresponding to these
permutations is equivalent to the special case of
\cite[Thm.\,6.4]{CHSSkanEKL} corresponding to these permutations.
%$321$-avoiding permutations.
Given a column-strict $G$-tableau $U$ containing a path family
$\pi = (\pi_1, \dotsc, \pi_n)$, define $\inv(U)$ to be the number 
of intersecting pairs $(\pi_i,\pi_j)$ with $j > i$
%$\pi_i \cap \pi_j \neq \emptyset$,
and $\pi_j$ appearing in an earlier column of $U$ than $\pi_i$.

\begin{cor}\label{c:dsn}
  Let $G$ be the wiring diagram of a reduced expression $\sprod im$ for
  a $321$-avoiding, $3412$-avoiding permutation $w \in \sn$.
Then we have
\begin{equation*}
\epsilon_q^\lambda(\qew C'_w(q)) = \sum_U q^{\inv(U)},
\end{equation*}
where the sum is over all column-strict $G$-tableaux
of type $e$ and shape $\lambda^\tr$.
\end{cor}
\begin{proof}
  Since $w$ is $321$-hexagon-avoiding, we have the formula (\ref{eq:321hex}).
  Since $w$ also avoids the pattern $3412$, the results
  \cite[Thm.\,4.3, Prop.\,4.4]{SkanNNDCB} imply that $G$
  has the structure of a {\em zig-zag} network \cite[Sec.\,5]{SkanNNDCB}.
  Thus the unique family $\pi = (\pi_1, \dotsc, \pi_n)$ which covers $G$
  and has type $e$ satisfies
\begin{enumerate}
\item $\pi$ has no crossings,
\item the intersection of any two paths of $\pi$ is either empty,
  or consists of a single connected component.
\end{enumerate}
%The first condition implies
Thus we have that
$\cross(U) = 0$
for each tableau $U$ in (\ref{eq:321hex}).
%each tableau $U$ in the sum (\ref{eq:321hex}) satisfies
%we have
Since $\pi$ covers $G$,
we also have that
%the second condition implies that
each nonempty intersection
%for each intersecting pair $(\pi_i, \pi_j)$
%the subgraph
$\pi_i \cap \pi_j$ is a single vertex, more specifically a noncrossing.
It follows that
each tableau $U$ in (\ref{eq:321hex}) satisfies
$\incross(U) = \inv(U)$.
\end{proof}

%As a consequence of Theorem~\ref{t:qepsilon}, we may extend the result
%\cite[Thm.\,6.4]{CHSSkanEKL}, which combinatorially interprets
%the evaluations of $\epsilon_q^\lambda$ at the subset
%\begin{equation*}
%  \{ \qew C'_w(q) \,|\, \text{$w$ \avoidsp} \}
%\end{equation*}
%of the Kazhdan-Lusztig basis of $\hnq$.
%For $n \geq 4$, the $321$-hexagon-avoiding permutations in $\sn$
%and the \pavoiding permutations in $\sn$ intersect, but neither
%set contains the other.

\section{Open problems}\label{s:probs}

While Theorem~\ref{t:qepsilon} provides a method for evaluating
all $\hnq$-characters at all elements
$(1 + T_{s_{i_1}}) \cdots (1 + T_{s_{i_m}})$ of $\hnq$, this method leads to
formulas involving subtraction: not all $\hnq$-characters
belong to the $\mathbb N[q]$-span of
%are nonnegative linear combinations of
the induced sign characters
$\{ \epsilon_q^\lambda \,|\, \lambda \vdash n \}$.
It would therefore be interesting to state and prove an analog of
Theorem~\ref{t:qepsilon} for irreducible characters,
since all other $\hnq$-characters belong to the $\mathbb N[q]$-span
%are nonnegative linear combinations
of these.
(See the table following Equation (\ref{eq:wdelt}).)
This would also provide a combinatorial proof of a weakening
of Haiman's result \cite[Lem.\,1.1]{HaimanHecke}.
\begin{prob}\label{p:qchi}
  For all irreducible characters $\chi_q^\lambda$ and
  all sequences $(s_{i_1}, \dotsc s_{i_m})$ of generators of $\sn$,
  combinatorially interpret the evaluation
  $\chi_q^\lambda((1 + T_{s_{i_1}}) \cdots (1 + T_{s_{i_m}}))$.
\end{prob}

%\bibliography{../my}
\end{document}